\documentclass[12pt,reqno]{article}
\usepackage[body={6.5in,9.0in},left=1in,top=1in,centering]{geometry}

\usepackage{amssymb,amsmath,amsthm,amsfonts,color}
\usepackage[colorlinks=true,urlcolor=blue,linkcolor=blue,citecolor=blue,pdfstartview=FitH]{hyperref}
\usepackage{authblk,pstricks}
\usepackage[sort,comma]{natbib}
\usepackage[title]{appendix}

\newtheorem{thm}{Theorem}[section]
\newtheorem{cor}[thm]{Corollary}

\newtheorem{prop}[thm]{Proposition}
\newtheorem{lem}[thm]{Lemma}
\theoremstyle{definition}
\newtheorem{defn}[thm]{Definition}

\newtheorem{cnd}[thm]{Condition}

\newtheorem{rem}[thm]{Remark}

\newcommand{\thmref}[1]{Theorem~{\rm \ref{#1}}}
\newcommand{\lemref}[1]{Lemma~{\rm \ref{#1}}}

\newcommand{\cndref}[1]{Condition~{\rm \ref{#1}}}
\newcommand{\propref}[1]{Proposition~{\rm \ref{#1}}}
\newcommand{\defref}[1]{Definition~{\rm \ref{#1}}}
\newcommand{\remref}[1]{Remark~{\rm \ref{#1}}}

\newcommand{\apref}[1]{Appendix~{\rm \ref{#1}}}
\newcommand{\sectref}[1]{Section~{\rm \ref{#1}}}

\makeatletter \@addtoreset{equation}{section}

\allowdisplaybreaks

\def\R{\ensuremath {\mathbb R}}
\newcommand{\EE}{\mathbb E}
\newcommand{\PP}{\mathbb P}
\newcommand{\NN}{\mathbb N}
\newcommand{\RR}{\mathbb R}

\newcommand{\I}{{\mathcal I}}
\newcommand{\A}{{\mathcal A}}

\newcommand{\yzstar}{{y_0^*}}
\newcommand{\zzstar}{{z_0^*}}

\newcommand{\gn}{G_n}
\newcommand{\wdh}{\widehat}
\newcommand{\wdt}{\widetilde}

\title{ A Weak Convergence Approach to Inventory Control Using a Long-term Average Criterion\thanks{This research was supported in part by     the Simons Foundation (grant award numbers  246271 and  523736) and the NSFC (No. 11671034).}}
 
\author[1]{K.L. Helmes}
\author[2]{R.H. Stockbridge}
\author[2]{C. Zhu}
\affil[1]{\small Institute for Operations Research, Humboldt University of Berlin, Spandauer Str. 1, 10178, Berlin, Germany, {\tt helmes@wiwi.hu-berlin.de}}
\affil[2]{Department of Mathematical Sciences,   University of Wisconsin-Milwaukee,   Milwaukee, WI 53201,   USA,   {\tt stockbri@uwm.edu}, {\tt zhu@uwm.edu}}
 
\date{}

\begin{document}
\maketitle

\begin{abstract}
  This paper continues the examination of inventory control in which the inventory is modelled by a diffusion process and a long-term average cost criterion is used to make decisions.  The class of such models under consideration have general drift and diffusion coefficients and boundary points that are consistent with the notion that demand should tend to reduce the inventory level.  The conditions on the cost functions are greatly relaxed from those in \cite{helm:15b}.  Characterization of the cost of a general $(s,S)$ policy as a function of two variables naturally leads to a nonlinear optimization problem over the ordering levels $s$ and $S$.  Existence of an optimizing pair $(s_*,S_*)$ is established for these models under very weak conditions; non-existence of an optimizing pair is also discussed.  Using average expected occupation and ordering measures and weak convergence arguments, weak conditions are given for the optimality of the $(s_*,S_*)$ ordering policy in the general class of admissible policies.  The analysis involves an auxiliary function that is globally $C^2$ and which, together with the infimal cost, solves a particular system of linear equations and inequalities related to but different from the long-term average Hamilton-Jacobi-Bellman equation.  This approach provides an analytical solution to the problem rather than a solution involving intricate analysis of the stochastic processes.  The range of applicability of these results is illustrated on a drifted Brownian motion inventory model, both unconstrained and reflected, and on a geometric Brownian motion inventory model under two different cost structures.
  
  \vspace{5 mm}
\vspace{3 mm}


\end{abstract}

\section{Introduction}
This paper further develops our examination of the long-term average cost criterion for inventory models of diffusion type in which the only control over the inventory levels is through the action of ordering additional stock; see \cite{helm:15b} for our previous investigation.  It identifies {\em weak}\/ sufficient conditions for optimality of an $(s,S)$ ordering policy in the {\em general class}\/ of admissible policies as well as providing some sufficient conditions for non-existence of an optimal $(s,S)$ policy. 

We model the inventory processes (in the absence of orders) as solutions to a stochastic differential equation
\begin{equation} \label{dyn}
dX_0(t) = \mu(X_0(t))\, dt + \sigma(X_0(t))\, dW(t), \qquad X_0(0) = x_0,
\end{equation}
taking values in an interval ${\mathcal I} = (a,b)$; negative values of $X_0(t)$ represent back-ordered inventory.  The detailed discussion in \cite{chen:10} indicates the validity of state-dependent diffusion models for inventory management.

An ordering policy $(\tau,Y)$ is a sequence of pairs $\{(\tau_k,Y_k): k \in \NN\}$ in which $\tau_k$ denotes the (random) time at which the $k^{th}$ order is placed and $Y_k$ denotes its size.  Since order $k+1$ cannot be placed before order $k$, $\{\tau_k: \in \NN\}$ is an increasing sequence of times.  The inventory level process $X$ resulting from an ordering policy $(\tau,Y)$ therefore satisfies the equation
\begin{equation} \label{controlled-dyn}
X(t) = x_0 + \int_0^t \mu(X(s))\, ds + \int_0^t \sigma(X(s))\, dW(s) + \sum_{k=1}^\infty I_{\{\tau_k \leq t\}} Y_k.
\end{equation}
Note the initial inventory level $X(0-)=x_0$ may be such that an order is placed at time $0$ resulting in a new inventory level at time $0$; this possibility occurs when $\tau_1 = 0$.  Also observe that $X(\tau_k-)$ is the inventory level just prior to the $k^{th}$ order being placed while $X(\tau_k)$ is the level with the new inventory.  Thus, this model assumes that orders are filled instantaneously.

Let $(\tau,Y)$ be an ordering policy and $X$ be the resulting inventory level process satisfying \eqref{controlled-dyn}.  Let $c_0$ and $c_1$ denote the holding/back-order cost rate and ordering cost functions, respectively.  We assume there is some constant $k_1 > 0$ such that $c_1 \geq k_1$; this constant represents the fixed cost for placing each order.  The long-term average expected holding/back-order plus ordering costs is
\begin{equation} \label{lta-obj-fn}
J_0(\tau,Y):= \limsup_{t\rightarrow \infty} t^{-1} \EE\left[\int_0^t c_0(X(s))\, ds + \sum_{k=1}^\infty I_{\{\tau_k \leq t\}} c_1(X(\tau_k-),X(\tau_k))\right].
\end{equation}
The goal is to identify an admissible ordering policy so as to minimize the cost.   

As mentioned earlier, we revisit the problem examined in \cite{helm:15b}.  We thus refer the reader to that paper for a discussion of the existing literature related to this problem; see also \cite{bens:11}.  The current manuscript relaxes assumptions imposed on $c_1$ such as the concavity requirement in \cite{yao:15}.  We focus the remainder of our comments on that which distinguishes this manuscript from our earlier publication.

The present paper differs from the published one in three significant respects: {\em solution approach, technical requirements}\/ and {\em generality of the class of admissible policies}.  Here, we employ weak convergence of average expected occupation and ordering measures to analyze the long-term average costs whereas \cite{helm:15b} relies heavily on intricate pathwise analysis of the stochastic inventory processes.  The benefit of the weak convergence approach lies in the ease of establishing the tightness of the expected {\em occupation}\/ measures and hence their sequential compactness.  This compactness ensures the existence of long-term average limiting measures that are key to the analysis of the costs.  Quite surprisingly, tightness is {\em not} required of the average expected ordering measures so there may not exist any limiting ordering measures.  Nevertheless, a limiting argument establishes optimality in the general class of admissible policies of an $(s_*,S_*)$ policy.  \apref{appendix-C} presents an example of an ordering policy having finite cost for which the average expected occupation measures are tight but the average expected ordering measures are not tight.

From a technical point of view, the weak convergence approach allows conditions on both cost functions $c_0$ and $c_1$ to be considerably relaxed by removing many structural requirements (compare \cndref{cost-cnds} below with Condition~2.3 of \cite{helm:15b}).  For example, the requirement that $c_0$ approaches $\infty$ at each boundary and the rather restrictive modularity condition (2.6) of that paper on $c_1$ are unnecessary.  In their places, we merely require that these functions be continuous, with $c_0$ continuous at the boundaries, an integrability condition on $c_0$ at the boundary $b$, as well as other boundedness relations contained in Conditions \ref{extra-cnd} and \ref{AG0-unif-int}.  

Most significantly, the two papers differ in how the auxiliary function is defined which, together with the infimal cost, form a solution of a particular system of linear equations and inequalities (see \propref{qvi-ish}) related to but different from  the long-term average Hamilton-Jacobi-Bellman equation.  In the present manuscript, the auxiliary function $G_0$ is a $C^2$ function defined globally on ${\cal I}$, contrasting with the function $G$ in the earlier one which $C^1$-smoothly pastes two functions at a significant point in ${\cal I}$.  Moreover, frequently in the literature, the  long-term average auxiliary function is obtained from the value functions of discounted problems using the method of vanishing discount.  The function $G_0$ in this paper is not obtained in this manner.

The functions $G$ and $G_0$ in the two papers play a central role in establishing general optimality of an $(s,S)$ ordering policy.  This manuscript provides an analytic solution for the general model that applies to all admissible policies.  In contrast, the earlier publication only shows optimality in a smaller class of ordering policies for the general models under a more restrictive set of conditions.  Optimality of $(s,S)$ policies for the general class of admissible policies for two examples from \cite{he:arxiv} and \cite{helm:15b} was established using ad hoc methods specifically tuned to the examples. 

The initial steps to the solution of the inventory control problem are similar to those in \cite{helm:15b}.  The next section briefly develops the model formulation, identifies two important functions and analyzes $(s,S)$ ordering policies with an emphasis on the differences between the manuscripts.  In particular, we concentrate on the more relaxed conditions in the current paper and we provide some new results that are required for the weak convergence analysis.  Some results are common to both papers; we refer to \cite{helm:15b} for those results whose proofs remain valid but we provide complete proofs in \apref{appendix-A} for \thmref{F-optimizers} since the previous proof is no longer valid for the general models of this paper.  \sectref{sect:form} is designed to briefly set some of the common foundations in the two papers so as to allow the reader to quickly access the new ideas and approaches in succeeding sections of this paper.  \sectref{sect:occ-meas} introduces the expected occupation and expected ordering measures and discusses their tightness (or not).    

The main contributions of the paper are contained in Sections \ref{sect-main result} and \ref{sect:optim}.  \sectref{sect-main result} defines a particular auxiliary function $G_0$ at the heart of our analytical approach and gives two important results related to $G_0$.  \sectref{sect:optim} then approximates $G_0$ by functions in a subclass and uses these approximations to establish optimality of the $(s_*,S_*)$ ordering policy in the general class of admissible policies.  Finally, \sectref{sect:examples} illustrates the ease of application of these results on examples involving a drifted Brownian motion, both unconstrained and with reflection at $\{0\}$, a geometric Brownian motion having two different cost structures and  Feller's branching diffusion process.  These examples are chosen to illustrate the more general applicability of this paper's results over those in \cite{helm:15b}.  

\section{Formulation and Preliminary Results} \label{sect:form}
This section briefly establishes the models under consideration, which are more general than those in \cite{helm:15b}.  The approach is very similar to the one taken in Sections 2 and 3 of that paper so we refer the reader to those sections for more details.  We emphasize the differences between the conditions.  Some proofs must be modified for the more general models of this paper.  Also some new results are required for our weak convergence arguments.  This overview will be kept brief so that the reader may reach the new ideas of this paper with a minimum of effort.  The reader may also find Chapter 15 of \cite{karl:81} to be a good reference for the boundary classifications of one-dimensional diffusions, especially for properties of the scale function and speed measure.

\subsection{Formulation}
Let ${\mathcal I} = (a,b) \subseteq \R$.  In the absence of ordering, the inventory process $X_0$ satisfies \eqref{dyn} and is a regular diffusion.  Throughout the paper we assume that the functions $\mu$ and $\sigma$ are continuous on $\I$, and that \eqref{dyn} is nondegenerate. The initial position of $X_0$ is taken to be $x_0$ for some $x_0 \in {\cal I}$.  Let $\{{\mathcal F}_t\}$ denote the filtration generated by $X_0$, augmented so that it satisfies the usual conditions.  We place the following assumptions on the underlying diffusion model.

\begin{cnd} \label{diff-cnd}
\begin{description}
\item[(a)] Both the speed measure $M$ and the scale function $S$ of the process $X_0$ are absolutely continuous with respect to Lebesgue measure.
\item[(b)] The left boundary $a$ is attracting and the right boundary $b$ is non-attracting.  Moreover, when $b$ is a natural boundary, $M[y,b) < \infty$ for each $y \in \I$.  The boundaries $a = -\infty$ and $b = \infty$ are required to be natural.
\end{description}
\end{cnd}

When $b$ is a natural boundary, the requirement that $M[y,b) < \infty$ for each $y\in \I$ is imposed in this paper as this is required for our general approach and does not hold in general.  Lemma~2.4 of \cite{helm:15b} shows that this condition follows from the more restrictive assumptions in that paper.

Associated with the scale function $S$ of \cndref{diff-cnd}, one can define the scale measure on the Borel sets of ${\mathcal I}$ by $S[y,z] = S(z) - S(y)$ for $[y,z] \subset {\mathcal I}$. 

From the modeling point of view, \cndref{diff-cnd}(b) is reasonable since it essentially says that, in the absence of ordering, demand tends to reduce the size of the inventory.  The  boundary point $a$ may be regular, exit or natural with $a$ being attainable in the first two cases and unattainable in the third.  In the case that $a$ is a regular boundary, its boundary behavior must also be specified as being either reflective or sticky.  The boundary point $b$ is either natural or entrance and is unattainable from the interior in both cases.  Following the approach in \cite{helm:15b}, we define the state space of possible inventory levels to be the interval ${\cal E}$ which excludes any natural boundary point and includes $a$ when it is attainable and $b$ when it is entrance.  Since orders typically increase the inventory level, define ${\cal R} = \{(y,z) \in {\cal E}^2: y < z\}$ in which $y$ denotes the pre-order and $z$ the post-order inventory levels, respectively.

Since we are using weak convergence methods for measures on ${\cal E}$ and ${\cal R}$, we will need their closures as well.
Define $\overline{\mathcal{E}}$ to be the {\em closure in $\RR$}\/ of $\mathcal{E}$; thus when a boundary is finite and natural, it is not an element of $\mathcal{E}$ but is in $\overline{\mathcal{E}}$.  Note $\pm \infty \notin \overline{\cal E}$.  Also set $\overline{\cal R} = \{(y,z) \in {\cal E}^2: y \leq z\}$; in contrast to ${\cal R}$, the set $\overline{\cal R}$ includes orders of size $0$.  Notice the subtle distinction between $\overline{\cal E}$ which includes boundaries that are finite and natural and $\overline{\cal R}$ which does not allow either coordinate to be such a point.  The reason for this distinction lies in the fact that the expected ordering measures are not required to be tight.  

We adopt the standard assumptions on the class ${\cal A}$ of admissible ordering policies; namely that $\{\tau_k:k\in \NN\}$ is an increasing sequence of $\{{\cal F}_t\}$-stopping times and for each $k \in \NN$, $Y_k$ is nonnegative, ${\cal F}_{\tau_k}$-measurable and satisfies $X(\tau_{k}) \in {\cal E}$.

Turning to the cost functions, we impose the following standing assumptions throughout the paper.
\begin{cnd} \label{cost-cnds}
\begin{itemize}
\item[(a)] The holding/back-order cost function $c_0: {\mathcal  I} \rightarrow \R^+$ is  continuous.  Moreover, at the boundaries 
$$\lim_{x\rightarrow a} c_0(x) =: c_0(a) \mbox{  exists in $\overline{\R^+}$ and } \lim_{x\rightarrow b} c_0(x) =: c_0(b) \mbox{  exists in $\overline{\R^+}$;}$$ 
we require $c_0(\pm\infty) = \infty$.  Finally, for each $y \in {\cal I}$,
\begin{equation} \label{c0-M-integrable}
\int_y^b c_0(v)\, \, dM(v)  < \infty.
\end{equation}
\item[(b)] The function $c_1:\overline{\cal R} \rightarrow \overline{\R^+}$ is in $C(\overline{\mathcal R})$ with $c_1 \geq k_1 > 0$ for some constant $k_1$. 
\end{itemize}
\end{cnd}

Again, since we will use weak convergence arguments for measures on $\overline{\cal E}$, we require $c_0$ to be continuous at the boundaries, even when they are natural.  The models in \cite{helm:15b} impose the stricter condition that $c_0(a)=c_0(b)=\infty$.  The condition \eqref{c0-M-integrable} is required in both papers but \cite{helm:15b} also imposes $\int_y^b \int_u^b c_0(v)\, dM(v)\, dS(u) = \infty$ for each $y\in {\cal I}$.  Thus, with the view for applications, the current paper allows a more flexible class of functions $c_0$.  Both papers assume $c_1$ is continuous and strictly bounded away from $0$.  \cite{helm:15b} also imposes on $c_1$ monotonicity ($c_1(y,z) \geq c_1(z,z)$ for $y \leq z$) and (restrictive) modularity ($c_1(w,z) + c_1(x,y) = c_1(w,y) + c_1 (x,z)$ for $w\leq x \leq y \leq z$). 

Both papers assume $c_0$ is inf-compact; that is, $\{x\in \overline{\cal E}: c_0(x) \leq k\}$ is compact for each $k \geq 0$.  Though this is not explicitly stated in \cndref{cost-cnds}(a), it is immediate when $c_0(a)=\infty$ and $c_0(b)=\infty$.  When $c_0(a) < \infty$ and $c_0(b) < \infty$, $a$ and $b$ must be finite.  Continuity of $c_0$, including at the boundaries, then establishes that $\{x\in \overline{\cal E}: c_0(x) \leq k\} \cap [a,b]$ is compact for every $k \geq 0$.

The generator of the process $X$ between jumps (corresponding to the diffusion $X_0$) is
$Af = \mbox{$\frac{\sigma^2}{2}$} f'' + \mu f'$, which is defined for all $f \in C^2({\mathcal I})$ or equivalently, $Af = \frac{1}{2} \frac{d~}{dM} \left(\frac{df}{dS}\right)$.   To capture the effect ordering has on the inventory process, define the jump operator $B: C({\cal E}) \rightarrow C(\overline{\mathcal{R}})$ by $Bf(y,z) = f(z) - f(y)$ for $(y,z) \in \overline{\cal R}.$

\subsection{Important functions}

As in \cite{helm:15b}, two functions play a central role in our search for an optimal ordering policy.  Using the initial position $x_0\in \mathcal{I}$, define the functions $g_0$ and $\zeta$ on $\I$ by
\begin{align} \label{g0-fn}
  g_{0}(x) := \int_{x_0}^{x} \int_{u}^{b} 2 c_{0}(v)\, dM(v)\, dS(u) , \qquad \text{and} \qquad
\zeta(x) := \int_{x_0}^{x} \int_{u}^{b} 2\, dM(v)\, dS(u)
\end{align}
and extend these functions to $\overline{\cal E}$ by continuity.  Observe that both $g_0$ and $\zeta$ are negative on $(a,x_0)$ and positive on $(x_0,b)$; also $g_0$ may take values $\pm \infty$ at the boundaries.  The definitions of $g_0$ and $\zeta$ differ slightly from the similar functions in \cite{helm:15b} in that we specify the lower limit of the outer integrals to be the initial point $x_0$.  Using the second characterization of $A$, it immediately follows that $g_0$ and $\zeta$, respectively, are particular solutions on $\mathcal{I}$ of
$$\left\{\begin{array}{l}
Af = - c_0, \\
f(x_0) = 0,
\end{array} \right. \quad \mbox{ and } \quad \left\{\begin{array}{l}
Af = -1, \\
f(x_0) = 0.
\end{array}\right.$$
Other solutions to these differential equations having value $0$ at $x_0$ include summands of the form $K(S(x) - S(x_0))$, $K \in \RR$, since the constant function and the scale function $S$ are linearly independent solutions of the homogeneous equation $Af = 0$.  However, such additional terms grow too quickly near the boundary $b$ so that a transversality condition in \propref{sS-in-Q2} fails (see \remref{rem-ho-homogeneous-solns}) and therefore the definitions of $g_0$ and $\zeta$ in \eqref{g0-fn} 
   exclude these terms.

To gain some intuition for the functions $g_0$ and $\zeta$, let $(y,z) \in {\cal R}$ and let $X_0$ satisfy \eqref{dyn} with $X_0(0) = z$.  Define $\tau_y = \inf\{t \geq 0: X_0(t) = y\}$.  Then
$$\EE_z\left[\int_0^{\tau_y} c_0(X_0(s))\, ds\right] = g_0(z) - g_0(y) \quad \mbox{and}\quad \EE_z[\tau_y] = \zeta(z) - \zeta(y).$$
This result is contained in Proposition~2.6 of \cite{helm:15b}, whose proof only requires \cndref{diff-cnd}(b) and \eqref{c0-M-integrable} of \cndref{cost-cnds}(a) in this paper.  Thus these relationships continue to hold for the models in this paper.

The proof of Proposition 3.5 in \cite[pp. 1848-9]{helm:15b} establishes that $\lim_{y\rightarrow a} \zeta(y) > -\infty$ when $a$ is attainable (regular or exit) and $\lim_{y\rightarrow a} \zeta(y) = -\infty$ when $a$ is unattainable (natural).  The proof remains valid for the models in this paper.

We now provide the first new result that arises from the possibility of finite values of $c_0(a)$ and $c_0(b)$.  Its technical proof is given in \apref{appendix-A}.   
\begin{lem} \label{g0-at-a}
Assume \cndref{diff-cnd}.  Suppose $a$ and $b$ are natural boundaries and let $c_0(a)$ and $c_0(b)$ be as in \cndref{cost-cnds}(a).  Then the following asymptotic behaviours hold:
\begin{align} \label{z-fixed-g0-zeta-diff-at-a}
& \lim_{y\rightarrow a} \frac{g_0(z) - g_0(y)}{\zeta(z) - \zeta(y)} = c_0(a), \quad \forall z\in \mathcal{I}; \quad & & \lim_{z\rightarrow b} \frac{g_0(z) - g_0(y)}{\zeta(z) - \zeta(y)}  = c_0(b), \quad \forall y\in \mathcal{I}; & \\
\label{double-g0-zeta-diff-at-a}
& \lim_{(y,z)\rightarrow (a,a)} \frac{g_0(z) - g_0(y)}{\zeta(z) - \zeta(y)} = c_0(a); \quad & & \lim_{(y,z) \rightarrow (b,b)} \frac{g_0(z) - g_0(y)}{\zeta(z) - \zeta(y)}   = c_0(b); & \\
\label{g0-zeta-ratio-at-a}
& \lim_{y\rightarrow a} \frac{g_0(y)}{\zeta(y)} = c_0(a); \quad & &  \lim_{z\rightarrow b} \frac{g_0(z)}{\zeta(z)}   = c_0(b),& 
\end{align}
implying $\lim_{y\rightarrow a} g_0(y) = -\infty$ when $c_0(a) > 0$ and $\lim_{z\rightarrow b} g_0(z) = \infty$ when $c_0(b) > 0$.
\end{lem}

\subsection{Analysis of $(s,S)$ Ordering Policies}\label{sect-sS-policy}
Both this paper and \cite{helm:15b} rely on characterizing the long-term average cost for $(s,S)$ ordering policies.  For $(y,z) \in {\cal R}$, define the $(y,z)$ ordering policy $(\tau,Y)$ such that $\tau_0=0$ and
\begin{equation} \label{sS-tau-def}
\tau_k = \inf\{t > \tau_{k-1}: X(t-) \leq y\}, \quad \mbox{ and } \quad Y_k = z-X(\tau_k-),
\quad  k \geq 1,
\end{equation}
in which $X$ is the inventory level process satisfying \eqref{controlled-dyn} with this ordering policy.  The above definition of $\tau_k$ must be slightly modified when $k=1$ to be $\tau_1 = \inf\{t \geq 0: X(t-) \leq y\}$ to allow for the first jump to occur at time $0$ when $x_0 \leq y$.

Theorem~2.1 of \cite{sigm:93} provides existence and uniqueness of the stationary distribution for the process $X$ arising from a $(y,z)$ ordering policy for any $y,z \in {\cal E}$ with $y < z$ and moreover the one-dimensional distributions $\PP(X(t) \in \cdot)$ converge weakly to the stationary distribution as $t$ tends to infinity.  Proposition~3.1 of \cite{helm:15b} derives the density $\pi$ of the stationary distribution for $X$ and the long-run frequency $\kappa = \frac{1}{B\zeta(y,z)}$ of orders.  Its proof remains valid for the models in this paper.

Proposition~3.4 of the same paper uses a renewal argument to characterize the long-term average cost $J_0(\tau,Y)$ for the $(y,z)$ ordering policy in \eqref{sS-tau-def}:  
\begin{equation}\label{eq-c0-integral-slln}
 \lim_{t\rightarrow \infty} \frac{1}{t} \int_{0}^{t} c_{0}(X(s))ds = \frac{Bg_0(y,z)}{B\zeta(y,z)} \;\;(a.s. \mbox{ and in } L^1)
\end{equation} 
and therefore 
\begin{equation}
\label{eq-sS-cost}
J_0(\tau, Y) = \frac{c_{1}(y,z) + Bg_{0}(y,z)}{B\zeta(y,z)}.
\end{equation}
The $L^1$ convergence is given in the proof of Proposition~3.4  though it is not stated directly in the proposition.  The proposition applies to models in which $c_0$ is asymptotically infinite at both boundaries and $(y,z)$ is in the interior of ${\cal R}$.  These conditions are relaxed in the current paper so $c_0$ may be finite at a boundary.  The proof remains valid whenever $y=a>-\infty$ is attainable or $z=b<\infty$ is an entrance point (in which cases $a,b \in \mathcal{E}$).  Note that when $b$ is an entrance boundary $g_0(b)$ is allowed to be finite or infinite.

Motivated by \eqref{eq-sS-cost}, define the function $F_0: \overline{\cal R} \rightarrow \overline{\RR^+}$ by
\begin{equation} \label{eq-F-fn}
F_0(y,z) := \left\{\begin{array}{cl}\displaystyle 
\frac{c_1(y,z) + Bg_0 (y,z)} {B\zeta(y,z)}, & \quad (y,z) \in {\mathcal R}, \rule[-15pt]{0pt}{15pt}\\
\infty, & \quad  (y,z) \in \overline{\cal R} \text{ with } y=z.
\end{array} \right.
\end{equation}
Observe that $F_0$ is well-defined and continuous on $\overline{\mathcal{R}}$.  $F_0$ is the same function as $F$ in \cite{helm:15b}, though its definition here is explicit on the boundary $y=z$ representing orders of size $0$.

The goal is to optimize $F_0$.  Since $c_1 > 0$, $F_0(y,z) > 0$ for every $(y,z) \in \overline{\mathcal{R}}$ and thus $\inf_{(y,z)\in \overline{\mathcal{R}}} F_0(y,z) =:F_0^* \geq 0$.  The models in this paper allow $F_0^* = 0$ in which case it immediately follows that there is no minimizing pair $(\yzstar,\zzstar)$ of $F_0$.  A new proposition gives some sufficient conditions under which this occurs.

\begin{prop} \label{no-optimizers}
Assume Conditions \ref{diff-cnd} and \ref{cost-cnds} hold.  Also assume any of the following conditions:
\begin{description}
\item[(a)] The point $a>-\infty$ is a natural boundary and either (i) or (ii) holds:
\begin{description}
\item[(i)] for some $z \in \mathcal{E}$, $\liminf_{y \rightarrow a} \frac{c_1(y,z) - g_0(y)}{\zeta(y)} = 0$;
\item[(ii)]$c_0(a) = 0$ and 
\begin{equation} \label{asymptotic-order-costa}
\liminf_{(y,z) \rightarrow (a,a)} \frac{c_1(y,z)}{\zeta(z) - \zeta(y)} = 0.
\end{equation}
\end{description}
\item[(b)] The point $b<\infty$ is a natural boundary and and either (i) or (ii) holds:
\begin{description}
\item[(i)] for some $y \in \mathcal{E}$, $\liminf_{z\rightarrow b} \frac{c_1(y,z) + g_0(z)}{\zeta(z)} = 0$;
\item[(ii)] $c_0(b) = 0$ and $\liminf_{(y,z) \rightarrow (b,b)} \frac{c_1(y,z)}{\zeta(z) - \zeta(y)} = 0.$
\end{description}
\end{description} 
Then $F_0^* = 0$ and there does not exist any pair $(\yzstar,\zzstar) \in {\cal R}$ which minimizes $F_0$.
\end{prop}

\begin{proof}
The arguments in cases (a) and (b) are essentially the same so we establish the result when (a) holds.  Since $a$ is a natural boundary, $\zeta(y)\rightarrow -\infty$ as $y \rightarrow a$.  Thus under the hypothesis in (a,i), for some $z\in \mathcal{E}$, 
$$0 = \liminf_{y \rightarrow a} \frac{c_1(y,z) - g_0(y)}{\zeta(y)} = \liminf_{y \rightarrow a} \frac{c_1(y,z) + g_0(z) - g_0(y)}{\zeta(z) - \zeta(y)} = \liminf_{y\rightarrow a} F_0(y,z).$$
Since $F_0(y,z) > 0$ for all $(y,z)\in \mathcal{R}$, the result follows.  If (a,ii) holds, then combining \eqref{asymptotic-order-costa} with \eqref{double-g0-zeta-diff-at-a} of \lemref{g0-at-a} establishes that $F_0^* = 0$ and, again, the result holds. 
\end{proof}

As in \cite{helm:15b}, the main results depend on the existence of an optimizing pair $(\yzstar,\zzstar)\in {\cal R}$ of $F_0$.  Central to the proof of this existence in Proposition~3.5 of \cite{helm:15b} is the assumption $c_0(a)=c_0(b)=\infty$.  The relaxed assumptions on $c_0$ in \cndref{cost-cnds} allow us to identify a weaker set of conditions for the existence of an optimizing pair.  (\sectref{sect:examples} illustrates models which satisfy these conditions but which fail to satisfy the more restrictive conditions in \cite{helm:15b}.) 
\begin{cnd} \label{extra-cnd}  
The following conditions hold: 
\begin{description}
\item[(a)] The boundary $a$ is regular; or exit; or $a$ is a natural boundary for which either (i) or (ii) hold:
\begin{itemize}
\item[(i)] $c_0(a)=\infty$; 
\item[(ii)] $c_0(a) < \infty$, the function $F_0(\cdot, z)$ is strictly decreasing in a neighborhood of $a$ for each $z\in \mathcal E$ and there exists some $(\wdh{y},\wdh{z}) \in \mathcal{R}$ such that $F_0(\wdh{y},\wdh{z}) < c_0(a)$.
\end{itemize}
\item[(b)] The boundary $b$ is entrance; or $b$ is natural for which either (i) or (ii) hold:
\begin{itemize}
\item[(i)] $c_0(b)=\infty$;  
\item[(ii)] $c_0(b) < \infty$, $F_0(y, \cdot)$ is strictly increasing in a neighborhood of $b$ for every $y\in \mathcal E$ and there exists some $(\widetilde{y},\widetilde{z}) \in \mathcal{R}$ such that $F_0(\widetilde{y},\widetilde{z}) < c_0(b)$.
\end{itemize}
\end{description}
\end{cnd}

Observe that the conditions $F_0(\wdh{y},\wdh{z}) < c_0(a)$ and $F_0(\wdt{y},\wdt{z}) < c_0(b)$ imply that $c_0(a) > F_0^*$ and similarly, $c_0(b) > F_0^*$.  

A sufficient condition for the monotonicity of $F_0(\cdot,z)$ in \cndref{extra-cnd}(a,ii) is that  for each $z \in \mathcal{E}\backslash\{a\}$, there exists some $y_z > a$ such that $c_1(\cdot,z)$ is differentiable in the interval $(a,y_z)$ and  
\begin{equation} \label{F-decr-at-a} 
\frac{-\frac{\partial c_1}{\partial y}(y,z) +g_0'(y)}{\zeta'(y)} > F_0(y,z).
\end{equation}
A similar sufficient condition for the monotonicity of $F_0$ in \cndref{extra-cnd}(b,ii) can be formulated.

The inequality \eqref{F-decr-at-a}  has an economic interpretation.  For simplicity, assume $(a,b) \subseteq (0,\infty)$. The function $F_0$ on ${\cal R}$ represents the long-term average cost \eqref{eq-sS-cost} for each $(y,z)$ ordering policy as the ratio of the expected cost $c_1(y,z) + g_0(z) - g_0(y)$ over a cycle to the expected cycle length $\zeta(z) - \zeta(y)$ for the $(y,z)$ ordering policy.  

Observe that $\zeta$ is a strictly increasing function so the denominator of \eqref{F-decr-at-a} is positive.  Thus using the definition of $F_0$ in \eqref{eq-F-fn}, the inequality \eqref{F-decr-at-a} can equivalently be written as
$$\frac{\frac{\partial}{\partial y}[c_1(y,z) + g_0(z) - g_0(y)]}{c_1(y,z) + g_0(z) - g_0(y)} < \frac{\frac{\partial}{\partial y}[\zeta(z)-\zeta(y)]}{\zeta(z)-\zeta(y)}$$
and multiplying by $z$ on both sides preserves the inequality.  Thus this sufficient condition says that for each $y\in \mathcal{E}$, there exists some ($y$-dependent) neighbourhood of $a$ on which the expected cost over a cycle {\em is less elastic}\/ than the expected cycle length (with respect to variations of $z$).  The similar analysis for the boundary $b$ using \cndref{extra-cnd}(b) results in the reverse inequality and more elasticity.

The key result of this section is the existence of a minimizing pair $(\yzstar,\zzstar)$.  This improves the result in Proposition~3.5 of \cite{helm:15b}; these relaxed conditions differ significantly enough that it is necessary to provide a careful proof.  The proof is given in \apref{appendix-A}.

\begin{thm} \label{F-optimizers}
Assume Conditions \ref{diff-cnd}, \ref{cost-cnds} and \ref{extra-cnd} hold.  Then  there exists a pair $(\yzstar,\zzstar) \in \mathcal R$ such that
\begin{equation} \label{eq-F-optimal}
F_0(\yzstar,\zzstar) = F_0^* = \inf\{F_0(y,z): (y,z) \in \overline{\mathcal R}\}.
\end{equation}
\end{thm}

An immediate corollary is that when the optimization is restricted to the class of $(y,z)$ ordering policies, the minimal cost is achieved by the $(\yzstar,\zzstar)$ ordering policy and has an optimal value of $F_0^*$.

\begin{cor}
Assume Conditions \ref{diff-cnd}, \ref{cost-cnds} and \ref{extra-cnd} hold and let $(\yzstar,\zzstar)$ denote a minimizing pair for the function $F$.  Then the ordering policy $(\tau^*,Y^*)$  defined using  $(\yzstar,\zzstar)$ in \eqref{sS-tau-def} is optimal in the class of all $(s,S)$ ordering policies with corresponding optimal value of $F_0^*=F_0(\yzstar,\zzstar) = J_0(\tau^*,Y^*)$.
\end{cor}

\section{Expected Occupation and Ordering Measures} \label{sect:occ-meas}
To establish general optimality of the $(\yzstar,\zzstar)$ policy, we apply weak convergence arguments with average expected occupation  and average expected ordering measures, which we now define.

For $(\tau,Y) \in \A$, let $X$ denote the resulting inventory level process satisfying \eqref{controlled-dyn}.  For each $t > 0$, define the average expected occupation measure $\mu_{0,t}$ on ${\cal E}$ and the average expected ordering measure $\mu_{1,t}$ on $\overline{\cal R}$ by
\begin{equation} \label{mus-t-def}
\begin{array}{rcll}
\mu_{0,t}(\Gamma_0) &=& \displaystyle \mbox{$\frac{1}{t}$} \EE\left[\int_0^t I_{\Gamma_0}(X(s))\, ds\right], & \quad \Gamma_0 \in {\cal B}({\cal E}), \rule[-15pt]{0pt}{15pt} \\
\mu_{1,t}(\Gamma_1)  &=& \displaystyle \mbox{$\frac{1}{t}$} \EE\left[\sum_{k=1}^\infty I_{\{\tau_k \leq t\}} I_{\Gamma_1}(X(\tau_k-),X(\tau_k))\right], & \quad \Gamma_1 \in {\cal B}(\overline{\cal R}).
\end{array}
\end{equation}
If $a$ is a reflecting boundary, define the average expected local time measure $\mu_{2,t}$ for each $t > 0$ to place a point mass on $\{a\}$ given by 
\begin{equation} \label{lta-local-time}
\mu_{2,t}(\{a\}) = \mbox{$\frac{1}{t}$} \EE[L_a(t)]
\end{equation}
in which $L_a$ denotes the local time of $X$ at $a$.

\begin{rem} \label{masses-observation}
For each $t > 0$, the average expected occupation measure $\mu_{0,t}$ is a probability measure on ${\cal E}$.  In addition, for each $(\tau,Y) \in \A$ with $J_0(\tau,Y) < \infty$, $\mu_{1,t}$ has finite mass for each $t$ and $\limsup_{t\rightarrow \infty} \mu_{1,t}(\overline{\cal R}) \leq J_0(\tau,Y)/k_1$.  Finally observe that when $a$ is a sticky boundary, $\mu_{0,t}$ places a point mass at $a$ for those policies $(\tau,Y)$ that allow the process $X$ to stick at $a$ with positive probability.
\end{rem}

For general policies $(\tau,Y)$, we are interested in the relative compactness of the collection of $\mu_{0,t}$ measures and the associated convergence (or not) of the functionals with integrand $c_0$.  The first proposition below shows that for any sequence $\{t_i\}$ going to $\infty$, $\{\mu_{0,t_i}\}$ is tight for policies $(\tau,Y)$ having finite costs.   To understand the second proposition, consider the $(y,z)$ ordering policy defined by \eqref{sS-tau-def} for $(y,z) \in {\cal R}$.  The $L^1$-convergence in \eqref{eq-c0-integral-slln} implies the remarkable result that for $(y,z)$ policies  
\begin{equation} \label{running-cost-conv}
\lim_{t\rightarrow \infty} \int_{\cal E} c_0(x)\, \mu_{0,t}(dx) = \int_{\cal E} c_0(x)\, \mu_0(dx),
\end{equation}
even when $c_0$ is unbounded, in which $\mu_0$ is the unique stationary measure on ${\cal E}$.  \propref{J-finite} demonstrates that \eqref{running-cost-conv} may not hold for general policies but rather an inequality relation holds.

\begin{prop} \label{mu0-tightness}
Assume Conditions~\ref{diff-cnd} and \ref{cost-cnds} hold.  For $(\tau,Y) \in \A$, let $X$ denote the resulting inventory process satisfying \eqref{controlled-dyn}.  Let $\{t_i: i \in \NN\}$ be a sequence such that $\lim_{i\rightarrow \infty} t_i = \infty$ and for each $i$, define $\mu_{0,t_i}$ by \eqref{mus-t-def}.  If $\{\mu_{0,t_i}: i \in \NN\}$ is not tight, then $J_0(\tau,Y) = \infty$.
\end{prop}

\begin{proof}
If $a,b$ are both finite, then each $\mu_{0,t}$ has its support in $\overline{\mathcal{E}}=[a,b]$ regardless of the types of boundary points.  The collection $\{\mu_{0,t}: t\geq 1\}$ is therefore tight.  Thus for this collection not to be tight, either $a=-\infty$ or $b=\infty$. 

Suppose $\{\mu_{0,t_i}\}$ is not tight.  By \cndref{cost-cnds}(a), $c_0(x) \rightarrow \infty$ as $x \rightarrow a$ or as $x \rightarrow b$.  The lack of tightness of $\{\mu_{0,t_i}\}$ means that there exists some $\epsilon > 0$ such that for all compact sets $K$ and $T > 0$, there exists $t_i \geq T$ for which $\mu_{0,t_i}(K^c) \geq \epsilon$.  Arbitrarily select $M < \infty$ and set $K = \{x \in \overline{\mathcal{E}}: c_0(x) \leq M/\epsilon\}$, noting that $K$ is compact.  Then there exists a subsequence $\{t_{i_j}: i_j\in \NN\}$ for which $\mu_{0,t_{i_j}}(K^c) \geq \epsilon$.  Therefore for each $i_j$,\
\begin{align*}
t_{i_j}^{-1} \EE\left[\int_0^{t_{i_j}} c_0(X(s))\,ds \right] &= \int_{{\cal E}} c_0(x)\, \mu_{0,t_{i_j}}(dx) 
\geq \int_{K^c} c_0(x)\, \mu_{0,t_{i_j}}(dx) 
&\geq \mbox{$\frac{M}{\epsilon}$} \mu_{0,t_{i_j}}(K^c) \geq M.
\end{align*}
Thus $\displaystyle \limsup_{t\rightarrow \infty} t^{-1} \EE\left[\int_0^t c_0(X(s))\, ds\right] \geq M$.  It therefore follows that $J_0(\tau,Y) = \infty$ since $M$ is arbitrary.  
\end{proof}

The above proof shows that tightness of $\{\mu_{0,t}\}$ follows from finiteness of the long-term average holding costs, $\lim_{t\rightarrow \infty} \frac{1}{t} \EE\left[\int_0^t c_0(X(s))\, ds\right]$, regardless of the limiting behaviour of the average expected ordering costs.

As a consequence of \propref{mu0-tightness}, when either $a=-\infty$ or $b=\infty$ we only need to consider those ordering policies $(\tau,Y)$ for which the collection $\{\mu_{0,t}: t \geq 0\}$ is tight.  We denote a generic weak limit as $t\rightarrow \infty$ by $\mu_0$.  Observe that the support of $\mu_0$ might be $\overline{\cal E}$.  The continuity assumptions on $c_0$ at the boundaries in \cndref{cost-cnds} are required due to the limiting measures possibly placing mass there.

\begin{prop} \label{J-finite}
Assume Conditions \ref{diff-cnd} and \ref{cost-cnds} hold.  Let $(\tau,Y) \in {\cal A}$ with $J_0(\tau,Y) < \infty$, $X$ satisfy \eqref{controlled-dyn}, and define $\mu_{0,t}$ by \eqref{mus-t-def} for each $t > 0$.  Then for each $\mu_0$ attained as a weak limit of some sequence $\{\mu_{0,t_j}\}$ as $t_j \rightarrow \infty$, 
$$\int_{\overline{\cal E}} c_0(x)\, \mu_0(dx) \leq J_0(\tau,Y) < \infty.$$
\end{prop}

\begin{proof}
Let $\mu_0$ and $\{t_j\}$ and $\{\mu_{0,t_j}\}$ be as in the statement of the proposition.  Since $\{\mu_{0,t_j}\}$, $\mu_0 \in {\cal P}(\overline{\cal E})$, the Skorokhod representation theorem implies there exist random variables $\{\Xi_j\}$ and $\Xi$, defined on a common probability space $(\wdt{\Omega},\wdt{\cal F},\wdt{\PP})$, such that $\Xi$ has distribution $\mu_0$, $\Xi_j$ has distribution $\mu_{0,t_j}$ for each $j$, and $\Xi_j \rightarrow \Xi$ almost surely as $j \rightarrow \infty$.  Since $c_0$ is bounded below, we may apply Fatou's lemma to obtain
$$\int_{\overline{\cal E}} c_0(x)\, \mu_0(dx) = \wdt{\EE}[c_0(\Xi)] \leq \liminf_{j\rightarrow \infty} \wdt{\EE}[c_0(\Xi_j)] = \liminf_{j\rightarrow \infty} \int_{\overline{\cal E}} c_0(x)\, \mu_{0,t_j}(dx) \leq J_0(\tau,Y) < \infty.$$
\end{proof}
We note that $c_0$ being infinite at a boundary implies that $\mu_0$ cannot assign any positive mass at this point.  In particular, for models in which $a$ is a sticky boundary and $c_0(a) = \infty$, any policy which allows $X$ to stick at $a$ on a set of positive probability incurs an infinite average expected cost for each $t$ and thus has $J_0(\tau,Y) = \infty$.  The condition $J_0(\tau,Y) < \infty$ therefore eliminates such $(\tau,Y)$ from consideration.

This inf-compactness of $c_0$ is central to finiteness of $J_0(\tau,Y)$ implying that $\{\mu_{0,t}\}$ is tight as $t\rightarrow \infty$.  A similar result follows in the case that $c_1$ is inf-compact as well.  However, inf-compactness fails for many natural ordering cost functions; in particular, the most commonly studied function $c_1(y,z) = k_1 + k_2(z-y)$, $(y,z) \in {\cal R}$, comprising fixed plus proportional costs is not inf-compact on an unbounded region of $\RR^2$.  \apref{appendix-C} contains an example which shows that even the requirement of a finite long-term average cost is insufficient to guarantee tightness of $\{\mu_{1,t}\}$ as $t\rightarrow \infty$.  Fortunately, our method of solution does not rely on  tightness of the average expected ordering measures.

\section{The Auxiliary Function $G_0$}\label{sect-main result} 
Our extension of the {\em optimality}\/ of the $(\yzstar,\zzstar)$ policy to the class $\mathcal{A}$ requires the existence of an optimizing pair $(\yzstar,\zzstar) \in \mathcal{R}$ with $F_0(\yzstar,\zzstar) = F_0^*$, which we now impose on the models.  Note that \cndref{extra-cnd} is a sufficient condition for such a pair to exist.  Recall, \cndref{cost-cnds} requires continuity of $c_0$ at the boundary, even for finite, natural boundaries; $c_0$ may take value $\infty$ at the boundaries.

Define the auxiliary function $G_0$ on $\mathcal{E}$ by
\begin{equation} \label{G0-def}
G_0 = g_0 - F_0^* \zeta
\end{equation}
and observe that $G_0 \in C(\mathcal{E})\cap C^2(\I)$.  Moreover, $G_0$ extends uniquely to $\overline{\cal E}$ due to the existence of $(\wdh{y},\wdh{z})$ and $(\wdt{y},\wdt{z})$ in \cndref{extra-cnd} or $c_0$ being infinite at the boundaries.  This observation follows immediately when $a$ is attainable and when $b$ is an entrance boundary since $\zeta$ is finite in these cases.  When $a$ or $b$ are natural boundaries, \lemref{g0-at-a} combined with \cndref{extra-cnd} shows that
\begin{equation} \label{nat-bdry-limits}
\lim_{x\rightarrow a} G_0(x) = \lim_{x\rightarrow a} (g_0(x) - F_0^* \zeta(x)) = \lim_{x\rightarrow a} \left(\mbox{$\frac{g_0(x)}{\zeta(x)}$} - F_0^*\right) \zeta(x) = -\infty
\end{equation}
and similarly $\lim_{x\rightarrow b} G_0(x) = \infty$.  

\begin{rem}
The function $G_0$ differs from the function $G$ used in \cite{helm:15b}.  $G_0$ is defined globally using \eqref{G0-def} whereas $G$ agrees with $G_0$ on the set $[\yzstar,b]$ but is $C^1$-smoothly pasted with $c_1(\cdot,\zzstar)$ on $[a,\yzstar]$ (using the notation in this paper).  The proof of optimality is significantly simplified since $G_0$ has a single expression.

$G_0$ has the following interpretation.  Let $y,z \in \mathcal{E}$.  Then
\begin{eqnarray*}
c_1(y,z) + BG_0(y,z) &=& c_1(y,z) + Bg_0(y,z) - F_0^* B\zeta(y,z) \\
&=& \left(\frac{c_1(y,z) + Bg_0(y,z)}{B\zeta(y,z)} - F_0^*\right) B\zeta(y,z) = (F_0(y,z) - F_0^*) \zeta(y,z).
\end{eqnarray*}
Notice the relation $F_0^* \leq F_0(y,z)$ holds for all $(y,z) \in \mathcal{R}$. Thus the function $c_1(y,z) + BG_0(y,z)$ gives the {\em increase in cost over a cycle}\/ incurred by using the $(y,z)$ ordering policy rather than an optimal ordering policy.

The function $G$ of \cite{helm:15b} has the same interpretation for $y \geq \yzstar$.  But for $y < \yzstar$, it arises from a policy which places an immediate order from $y$ to $\zzstar$ and then follows the $(\yzstar,\zzstar)$ policy.
\end{rem}

The function $G_0$ satisfies an important system of relations.

\begin{prop} \label{qvi-ish}
Assume Conditions \ref{diff-cnd}, \ref{cost-cnds} and \ref{extra-cnd} hold and let $G_0$ be as in \eqref{G0-def}.  Then $G_0$ is a solution of the system 
$$\left\{\begin{array}{rcll}
Af(x) + c_0(x) - F_0^*&=& 0, & \quad x\in \mathcal{I}, \\
Bf(y,z) + c_1(y,z) &\geq& 0,   & \quad (y,z) \in \overline{\mathcal{R}} \\
f(x_0) &=& 0,  & \\
Bf(\yzstar,\zzstar) + c_1(\yzstar,\zzstar) &=& 0. &
\end{array} \right.$$
Moreover, the first relation extends by continuity to $\overline{\cal E}$.
\end{prop} 

The proof is straightforward so is left to the reader.  

As a small digression to illuminate the definitions of $g_0$ and $\zeta$, the next proposition establishes an important transversality condition involving $G_0$ with the ensuing remark providing further clarification.

\begin{prop} \label{sS-in-Q2}
Assume Conditions \ref{diff-cnd}, \ref{cost-cnds} and \ref{extra-cnd}.  Let $x_0 \in \mathcal{I}$ be fixed.  For $a \leq y < z < b$, let $(\tau,Y)$ be the $(y,z)$ ordering policy defined by \eqref{sS-tau-def} and $X$ satisfy \eqref{controlled-dyn}.  Define the process $\wdt{M}$ by  
\begin{equation} \label{M-tilde-def}
\wdt{M}(t) := \int_0^t \sigma(X(s)) G_0'(X(s))\, dW(s), \quad t \geq 0.
\end{equation}
Then there exists a localizing sequence $\{\beta_n: n\in \NN\}$ of stopping times such that for each $n$, $\wdt{M}(\cdot \wedge \beta_n)$ is a martingale and the following transversality condition holds:
\begin{equation} \label{G0-transv-cnd}
\lim_{t\rightarrow \infty} \lim_{n\rightarrow \infty} \mbox{$\frac{1}{t}$} \EE[G_0(X(t\wedge \beta_n))] = 0.
\end{equation}
In addition, defining $\mu_0$ to be the stationary measure and $\mu_1$ to place point mass $\kappa=\frac{1}{B\zeta(y,z)}$ (the long-run frequency of orders) on $\{(y,z)\}$, we have
\begin{equation} \label{G0-constr}
\int_{\cal E} AG_0(x)\, \mu_0(dx) + \int_{\cal R} BG_0(y,z)\, \mu_1(dy\times dz) = 0. 
\end{equation}
\end{prop}

\begin{proof}
Let $(y,z)$, $(\tau,Y)$ and $X$ be as in the statement of the proposition.  Let $\{b_n:n\in\NN\}\subset \mathcal{I}$ be a strictly increasing sequence such that $x_0 \vee z \leq b_1$ and $\lim_{n\rightarrow \infty} b_n = b$ and define the sequence of localizing times $\{\beta_n: n\in \NN\}$ by $\beta_n = \inf\{t\geq 0: X(t) = b_n\}$.  Using It\^{o}'s formula, it follows that for each $n$,
\begin{eqnarray*}
G_0(X(t\wedge\beta_n)) = G_0(x_0) &+& \int_0^{t\wedge \beta_n} AG_0(X(s))\, ds + \int_0^{t\wedge \beta_n} \sigma(X(s)) G_0'(X(s))\, dW(s) \\
&+& \sum_{k=1}^\infty I_{\{\tau_k \leq t\wedge \beta_n\}} BG_0(X(\tau_k-),X(\tau_k));
\end{eqnarray*}
note that no local time term is included in this identity when $a$ is a reflecting boundary even when $y = a$ since an order is placed the first time $X$ hits $a$ so no local time is accrued.  Taking expectations and using the fact that $AG_0 = F_0^*-c_0$, we then obtain
\begin{eqnarray} \label{G0-id-6}
\EE[G_0(X(t\wedge \beta_n))] = G_0(x_0) + F_0^*\EE[t\wedge \beta_n] &-& \EE\left[\int_0^{t\wedge \beta_n} c_0(X(s))\, ds \right] \\ \nonumber
&+& \EE\left[\sum_{k=1}^\infty I_{\{\tau_k \leq t\wedge\beta_n\}} BG_0(X(\tau_k-),X(\tau_k))\right].
\end{eqnarray}
Since $b$ is inaccessible from the interior and $z < b$, it follows that $\beta_n \rightarrow \infty\; (a.s.)$ as $n\rightarrow \infty$.  Applying the monotone and dominated convergence theorems,  we have
\begin{align*}
& \lim_{n\rightarrow \infty} \EE[G_0(X(t\wedge \beta_n))] \\ &\ \ = G_0(x_0) + F_0^* t - \EE\left[\int_0^t c_0(X(s))\, ds+ \sum_{k=1}^\infty I_{\{\tau_k \leq t\}} BG_0(X(\tau_k-),X(\tau_k))\right] \\
&\ \ = G_0(x_0) + F_0^* t - \EE\left[\int_0^t c_0(X(s))\, ds\right] + BG_0(y,z) \EE\left[\sum_{k=1}^\infty I_{\{\tau_k \leq t\}}\right]  \\
& \ \  \quad +\; (BG_0(x_0,z) - BG_0(y,z))\EE[I_{\{\tau_1 = 0\}}].
\end{align*}
Divide by $t$ and let $t \rightarrow \infty$.  Note $\lim_{t\rightarrow \infty} \frac{1}{t} \EE\left[\sum_{k=1}^\infty I_{\{\tau_k \leq t\}}\right] = \frac{1}{B\zeta(y,z)}$ and the $L^1$ convergence in \eqref{eq-c0-integral-slln} gives $\lim_{t\rightarrow \infty} \frac{1}{t} \EE\left[\int_0^t c_0(X(s))\, ds\right] = \frac{Bg_0(y,z)}{B\zeta(y,z)}$.  Therefore
$$\lim_{t\rightarrow \infty} \lim_{n\rightarrow \infty} \mbox{$\frac{1}{t}$} \EE[G_0(X(t\wedge \beta_n))] = F_0^* - \frac{Bg_0(y,z)}{B\zeta(y,z)} + \frac{Bg_0(y,z) - F_0^* B\zeta(y,z)}{B\zeta(y,z)} = 0.$$
Using the definitions of $\mu_{0,t}$ and $\mu_{1,t}$ in \eqref{mus-t-def}, this argument establishes that
\begin{equation} \label{limit1}
\lim_{t\rightarrow \infty}\left[\int_{\cal E} AG_0(x)\, \mu_{0,t}(dx) + \int_{\cal R} BG_0(y,z)\, \mu_{1,t}(dy\times dz)\right] = 0.
\end{equation}
Since $AG_0 = F_0^* - c_0$, by \eqref{running-cost-conv} and the definition of weak convergence, 
\begin{equation} \label{limit2}
\lim_{t\rightarrow \infty} \int_{\cal E} AG_0(x)\, \mu_{0,t}(dx) = \int_{\cal E} AG_0(x)\, \mu_0(dx).
\end{equation}
Turning to the convergence involving $\mu_{1,t}$, the $(y,z)$ ordering policy places one order at time $0$ when $x_0 < y$ so $\mu_{1,t}$ has mass on $(x_0,z)$ but thereafter orders are only placed when $X(t-)=y$ so all further mass of $\mu_{1,t}$ is on $\{(y,z)\}$. When $x_0 \geq y$, $\mu_{1,t}$ only has mass on $\{(y,z)\}$.   
This implies $\mu_{1,t} \Rightarrow \mu_1$ as $t\rightarrow \infty$, which in turn implies  
\begin{equation} \label{limit3}
\lim_{t\rightarrow \infty} \int_{\cal R} BG_0(y,z)\, \mu_{1,t}(dy \times dz) = \int_{\cal R} BG_0(y,z)\, \mu_1(dy\times dz).
\end{equation}
Combining \eqref{limit1}, \eqref{limit2} and \eqref{limit3} establishes \eqref{G0-constr}. 
\end{proof}

\begin{rem} \label{rem-ho-homogeneous-solns}
By definition, the functions $g_0$ and $\zeta$ are solutions of $Af = -c_0$ and $Af = -1$, respectively, for which $g_0(x_0) = 0 = \zeta(x_0)$.  The choice of the function $G_0 = g_0-F_0^*\zeta$ is the only $C^2$ solution of the inhomogeneous equation $Af =F_0^*-c_0$ with $f(x_0)=0$ for which the transversality condition \eqref{G0-transv-cnd} holds.  To see this, recall the scale function $S$ and the constant function are two linearly independent solutions to the homogeneous equation $Af=0$.  Define the function $\wdh{S}(x) = S(x) - S(x_0)$ for $x \in \mathcal{E}$ so that $\wdh{S}$ is a solution of $Af=0$ with $\wdh{S}(x_0)=0$.  Following the previous proof to \eqref{G0-id-6} using $\wdh{S}$ yields
\begin{eqnarray*}
\EE[\wdh{S}(X(t\wedge \beta_n))] &=& \EE\left[\sum_{k=1}^\infty I_{\{\tau_k \leq t\wedge\beta_n\}} B\wdh{S}(X(\tau_k-),X(\tau_k))\right] \\
&=& B\wdh{S}(y,z) \EE\left[\sum_{k=1}^\infty I_{\{\tau_k \leq t\wedge \beta_n\}}\right] + [B\wdh{S}(x_0,z)-B\wdh{S}(y,z)]\EE[I_{\{\tau_k=0\}}].
\end{eqnarray*}
Thus letting $n\rightarrow \infty$, dividing by $t$ and letting $t\rightarrow \infty$ yields 
\begin{equation} \label{non-zero-lim}
\lim_{t\rightarrow \infty} \lim_{n\rightarrow \infty} \mbox{$\frac{1}{t}$} \EE[\wdh{S}(X(t\wedge \beta_n))] = \frac{B\wdh{S}(y,z)}{B\zeta(y,z)} = \frac{S(z) - S(y)}{\zeta(z) - \zeta(y)} > 0.
\end{equation}
Observe that for $K \in \RR$, $\wdh{G}_0 = G_0 + K\wdh{S}$ is the general solution to the inhomogeneous equation $Af=F_0^*-c_0$ with $f(x_0)=0$.  As a result of \propref{sS-in-Q2} and \eqref{non-zero-lim}, we have
$$\lim_{t\rightarrow \infty} \lim_{n\rightarrow \infty} \mbox{$\frac{1}{t}$} \EE[\wdh{G}_0(X(t\wedge\beta_n))] = \frac{K(S(z)-S(y))}{\zeta(z)-\zeta(y)}$$
and the transversality condition \eqref{G0-transv-cnd} fails for any $\wdh{G}_0$ with non-zero $K$.
\end{rem}

\section{Policy Class ${\cal A}_0$ and Optimality} \label{sect:optim}
For models having a reflecting boundary point $a$, we are able to prove optimality of the $(\yzstar,\zzstar)$ ordering policy within a slightly restricted class of admissible policies.  (Note there is no restriction on the class $\A$ when $a$ is not a reflecting boundary.)

\begin{defn} \label{admissible-class-A0-defn} 
For models in which $a$ is a reflecting boundary point, the class $\A_0 \subset \A$ consists of those policies $(\tau,Y)$ for which the transversality condition
\begin{equation} \label{no-local}
\lim_{t\rightarrow \infty} t^{-1} \EE[L_a(t)] = 0
\end{equation}
holds; recall, $L_a$ denotes the local time process of $X$ at $a$.
\end{defn} 

Referring to  Definition~3.11 of \cite{helm:15b}, restriction of the admissible class is also required in our previous paper.  In particular, the class $\A_1$ in that paper requires \eqref{no-local} along with the existence of a localizing sequence $\{\beta_n\}$ such that the stopped processes \eqref{M-tilde-def} are martingales and, using limits inferior, the additional transversality condition \eqref{G0-transv-cnd} holds; the conditions use the function $G$ of that paper.  This paper only requires the transversality condition \eqref{no-local} on the local time process.

For $(y,z)$ ordering policies, \propref{sS-in-Q2} shows that the corresponding limiting measures $(\mu_0,\mu_1)$ of the pairs $\{(\mu_{0,t},\mu_{1,t}): t > 0\}$ satisfy \eqref{G0-constr}.  However for a general ordering policy $(\tau,Y) \in {\cal A}_0$, the measures $\{\mu_{1,t}\}$ may not have weak limits as $t\rightarrow \infty$.  We establish the limiting adjoint relation \eqref{bar} for a class of test functions ${\cal D}$, defined below, that is used to show that $J_0(\tau,Y) \geq F_0^*$ for any $(\tau,Y) \in {\cal A}_0$.  

Recall, \propref{mu0-tightness} demonstrates that the measures $\{\mu_{0,t}\}$ are tight (as $t\rightarrow \infty$) whenever $J_0(\tau,Y) < \infty$.  It follows that any weak limit $\mu_0$ will have support in $\overline{\cal E}$.  However, when $c_0(a) = \infty$ or $c_0(b) = \infty$, $J_0(\tau,Y) < \infty$ implies that each limiting measure $\mu_0$ places no mass on the corresponding boundary.

We begin by identifying a particular class of test functions that plays a key role in the analysis.  

\begin{defn} \label{class-D-def}
A function $f$ is in ${\cal D}$ provided it satisfies
\begin{itemize}
\item[(a)] $f\in C(\overline{\cal E}) \cap C^2(\I)$ and there exists $L_f < \infty$ such that 
\begin{itemize}
\item[(i)] $|f| \leq L_f$;
\item[(ii)] $(\sigma f')^2 \leq L_f(1+c_0)$;
\item[(iii)] $|Af| \leq L_f$;
\end{itemize}
\item[(b)] 
\begin{itemize}
\item[(i)] for all models, at each boundary where $c_0$ is finite, $Af$ extends continuously to the boundary with a finite value;
\item[(ii)] when $a$ is a reflecting boundary, $|f'(a)| < \infty$; and 
\item[(iii)] when $a$ is a sticky boundary and $c_0(a) < \infty$, $\sigma f'$ extends continuously at $a$ to a finite value. 
\end{itemize} 
\end{itemize}
\end{defn}

As pointed out at the end of Section 3, any policy $(\tau,Y)$ which allows $X$ to stick at $a$ with positive probability when $c_0(a) = \infty$ has $J_0(\tau,Y) = \infty$, so is clearly not optimal.  Thus the stickiness of $a$ is only of concern when $c_0(a) < \infty$, in which case \defref{class-D-def}(b,i) and (b,iii) are required for our analysis.

\begin{prop} \label{bdd-adjoint}
Assume Conditions \ref{diff-cnd} and \ref{cost-cnds}.  Let $(\tau,Y) \in {\cal A}_0$ with $J_0(\tau,Y) < \infty$ and let $X$ satisfy \eqref{controlled-dyn}.  For $t > 0$, define $(\mu_{0,t},\mu_{1,t})$ by \eqref{mus-t-def} and let $\mu_0$ be such that $\mu_{0,t_j} \Rightarrow \mu_0$ as $j\rightarrow \infty$ for some sequence $\{t_j: j \in \NN\}$ with $\lim_{j\rightarrow \infty} t_j = \infty$.  Then the limiting adjoint relation
\begin{equation} \label{bar}
\forall f \in {\cal D}, \quad \int_{\overline{\cal E}} Af(x)\, \mu_0(dx) + \lim_{j\rightarrow \infty} \int_{\overline{\cal R}} Bf(y,z)\, \mu_{1,t_j}(dy\times dz) = 0
\end{equation}
holds.
\end{prop}

\begin{rem}
For policies $(\tau,Y)$ with $J_0(\tau,Y) < \infty$, \propref{mu0-tightness} only establishes tightness of the average expected occupation measures $\{\mu_{0,t}\}$ but not tightness of the average expected ordering measures $\{\mu_{1,t}\}$.  Ordering policies for which the collection $\{\mu_{1,t}\}$ is tight as $t\rightarrow \infty$ will have weak limits.  For any pair $(\mu_0,\mu_1)$ with $(\mu_{0,t_j},\mu_{1,t_j}) \Rightarrow (\mu_0,\mu_1)$ as $j\rightarrow \infty$, \eqref{bar} is more simply expressed as 
$$\forall f \in {\cal D}, \quad \int_{\overline{\cal E}} Af(x)\, \mu_0(dx) + \int_{\overline{\cal R}} Bf(y,z)\, \mu_1(dy\times dz) = 0.$$
\end{rem}

\begin{proof}
Let $(\tau,Y)$, $X$, $\{t_j\}$, $\{(\mu_{0,t_j},\mu_{1,t_j})\}$ and $\mu_0$ be as in the statement of the proposition.  The following analysis considers the case of $a$ being a reflecting boundary; when $a$ is not reflecting, the local time term is omitted.  For each $f \in {\cal D}$ and $j\in \NN$, It\^{o}'s formula yields
\begin{eqnarray*}
f(X(t_j)) = f(x_0) &+& \int_0^{t_j} Af(X(s))\, ds + \sum_{k=1}^\infty I_{\{\tau_k \leq t_j\}} Bf(X(\tau_k-),X(\tau_k)) \\
&+& \int_0^{t_j} \sigma(X(s))f'(X(s))\, dW(s) + f'(a) L_a(t_j).
\end{eqnarray*}
Observe that for each $t_j$, 
$$\EE\left[\int_0^{t_j} \left(\sigma(X(s)) f'(X(s))\right)^2\, ds\right] \leq 2\EE\left[\int_0^{t_j} L_f [1+c_0(X(s))]]\, ds\right] < \infty$$
so $\sigma f' \in L^2(\Omega \times [0,t_j])$ and the stochastic integral has mean $0$.  Thus taking expectations and dividing by $t_j$ yields
\begin{align} \label{simple-conv}\nonumber
{\mbox{$\frac{1}{t_j}$}\EE_{x_0}[f(X(t_j))]}&= \mbox{$\frac{f(x_0)}{t_j}$} + \int_{\cal E} Af(x)\, \mu_{0,t_j}(dx) + \int_{\overline{\cal R}} Bf(y,z)\, \mu_{1,t_j}(dy\times dz) + f'(a) \mbox{$\frac{\EE[L_a(t_j)]}{t_j}$}.
\end{align}
Observe carefully that $\overline{\cal E}$ differs from ${\cal E}$ only when at least one of the boundaries is finite and natural.  Moreover, for each $t_j$, $\mu_{0,t_j}$ places no mass on these boundaries so the integration may also be viewed as being over $\overline{\cal E}$ and this is required when we pass to the limit as $t_j \rightarrow \infty$ since the limiting measure $\mu_0$ may put positive mass on such a boundary point. 

Since by the definition of ${\cal D}$, the function $f\in C(\overline{\cal E})\cap C^2({\cal I})$ is bounded, implying $Bf$ is also bounded and continuous.  In addition, $Af$ is bounded and continuous and extends continuously to a finite value at $a$ when $a$ is a finite, natural boundary, and $f'(a)$ is finite.  Upon letting $j \rightarrow \infty$ we obtain
$$\int_{\overline{\cal E}} Af(x)\, \mu_0(dx) + \lim_{j \rightarrow \infty} \int_{\overline{\cal R}} Bf(y,z)\, \mu_{1,t_j}(dy\times dz) = 0;$$
the existence of $\lim_{j \rightarrow \infty} \int Bf\, d\mu_{1,t_j}$ follows from the limits existing for the other terms in the previous equation.  Notice, in particular, that $\mu_0$ places no mass on boundaries where $c_0$ is infinite so the weak convergence argument does not require $Af$ to extend continuously at such boundaries (cf., \defref{class-D-def}(b,i)).  When $a$ is a sticky boundary, however, the extensions in \defref{class-D-def}(b,i) and (b,iii) are required to apply It\^{o}'s formula.  
\end{proof}

\begin{cor}
Assume Conditions \ref{diff-cnd}, \ref{cost-cnds} and \ref{extra-cnd}.  Suppose $G_0 \in {\cal D}$.  Then for every $(\tau,Y) \in {\cal A}_0$, $J_0(\tau,Y) \geq F_0^*$ and hence the $(\yzstar,\zzstar)$ ordering policy is optimal in the class ${\cal A}_0$.
\end{cor}

\begin{proof} 
Let $(\tau,Y) \in {\cal A}_0$, $X$ satisfy \eqref{controlled-dyn} and let $\{t_j: j\in \NN\}$ be a sequence such that
$$J_0(\tau,Y) = \lim_{j\rightarrow \infty} \mbox{$\frac{1}{t_j}$} \EE\left[\int_0^{t_j} c_0(X(s))\, ds + \sum_{k=1}^\infty I_{\{\tau_k \leq t_j\}} c_1(X(\tau_k-),X(\tau_k))\right].$$
By considering a subsequence, if necessary, let $\mu_{0,t_j} \Rightarrow \mu_0$ for some $\mu_0 \in {\cal P}(\overline{\cal E})$.  The combination of Propositions \ref{qvi-ish} and \ref{bdd-adjoint} immediately implies that   
\begin{eqnarray*}
J_0(\tau,Y) &=& \lim_{j\rightarrow \infty} \left(\int_{\overline{\cal E}} c_0(x)\, \mu_{0,t_j}(dx) + \int_{\overline{\cal R}} c_1(y,z)\, \mu_{1,t_j}(dy\times dz)\right) \\
&=& \lim_{j\rightarrow \infty} \left(\int_{\overline{\cal E}} (AG_0 (x) + c_0(x))\, \mu_{0,t_j}(dx) + \int_{\overline{\cal R}} (BG_0(y,z) + c_1(y,z))\, \mu_{1,t_j}(dy\times dz)\right) \\
&\geq& \int_{\overline{\cal E}} F_0^* \, \mu_0(dx) = F_0^*
\end{eqnarray*}
and thus the $(\yzstar,\zzstar)$ ordering policy is optimal in the class ${\cal A}_0$.  
\end{proof}

In general $G_0 \notin \mathcal{D}$ so it is necessary to approximate $G_0$ by functions in $\mathcal{D}$ and pass to a limit.    Interestingly, the following analysis works with the approximating functions $\gn$ defined in \lemref{lem:G0-approx} without establishing an adjoint relation involving the function $G_0$.  

Recall from \eqref{nat-bdry-limits} that when $a$ is a natural boundary, $G_0(a) := \lim_{x\rightarrow a} G_0(x) = -\infty$ and similarly, $G_0(b):= \lim_{x\rightarrow b} G_0(x) = \infty$ when $b$ is natural.

To proceed, we impose another set of conditions.

\begin{cnd} \label{AG0-unif-int}
Let $G_0$ be as defined in \eqref{G0-def}.   
\begin{itemize}
\item[(a)] There exists some $L < \infty$ and some $y_1 > a$ such that 
\begin{itemize}
\item[(i)] for models having $c_0(a) = \infty$,
$$\frac{c_0(x)}{(1 + |G_0(x)|)^2} + \frac{(\sigma(x) G_0'(x))^2}{(1 + |G_0(x)|)^3} 
\leq L, \qquad a < x < y_1;$$
\item[(ii)] for models in which $c_0(a) < \infty$, there is some $\epsilon \in (0,1)$ such that 
$$\frac{(\sigma(x) G_0'(x))^2}{(1 + |G_0(x)|)^{2+\epsilon}} \leq L, \qquad a \leq x < y_1.$$
\end{itemize} 
\item[(b)] There exists some $L < \infty$ and some $z_1 < b$ such that 
\begin{itemize}
\item[(i)] for models having $c_0(b) = \infty$,
$$\frac{c_0(x)}{(1 + |G_0(x)|)^2} 
+ \frac{(\sigma(x)G_0'(x))^2}{(1+|G_0(x)|)(1+c_0(x))} \leq L, \qquad z_1 < x < b;$$
\item[(ii)] for models in which $c_0(b) < \infty$, there is some $\epsilon \in (0,1)$ such that 
$$\frac{(\sigma(x) G_0'(x))^2}{(1 + |G_0(x)|)^{2+\epsilon}} + \frac{(\sigma(x)G_0'(x))^2}{(1+|G_0(x)|)(1+c_0(x))} \leq L, \qquad z_1 < x \leq b.$$
\end{itemize} 
\item[(c)] \begin{itemize}
\item[(i)] When $G_0(a) > -\infty$, or when $a$ is a sticky boundary with $c_0(a) < \infty$,  $\displaystyle\lim_{x\rightarrow a} \sigma(x) G_0'(x)$ exists and is finite; 
\item[(ii)] when $a$ is a reflecting boundary, $G_0'(a)$ exists and is finite; and 
\item[(iii)] when $G_0(b) < \infty$,  $\displaystyle\lim_{x\rightarrow b} \sigma(x) G_0'(x)$ exists and is finite.
\end{itemize}
\end{itemize}
\end{cnd}

First note that the bound in \cndref{AG0-unif-int}(b,i) at the boundary $b$ is more restrictive than the similar bound in \cndref{AG0-unif-int}(a,i) at $a$ since
\begin{equation} \label{bi-implies-ai}
\frac{(\sigma(x) G_0'(x))^2}{(1 + |G_0(x)|)^3} = \frac{(\sigma(x)G_0'(x))^2}{(1+|G_0(x)|)(1+c_0(x))} \cdot \frac{1+c_0(x)}{(1+|G_0(x)|)^2} \leq L(1+L).
\end{equation}
The need for tighter restrictions at the boundary $b$ than at $a$ is not unexpected since there is no way to control the process from diffusing upwards whereas ordering can prevent the process from diffusing towards $a$.

The reason for having two different conditions in \cndref{AG0-unif-int}(a,b) based on whether $c_0$ at the boundary is finite or infinite is that any limiting pair of measures $(\mu_0,\mu_1)$ arising from an admissible policy $(\tau,Y)$ having finite cost $J_0(\tau,Y)$ must place no $\mu_0$-mass at a boundary where $c_0$ is infinite.  A weak limit $\mu_0$ may have positive mass at a boundary when $c_0$ is finite.  Also notice the subtle assumption in \cndref{AG0-unif-int}(a,ii) and (b,ii) that the bounds extend to the boundary whereas there is no assumption needed at the boundary in \cndref{AG0-unif-int}(a,i) and (b,i).

For models that satisfy Conditions~\ref{diff-cnd}, \ref{cost-cnds} and \ref{extra-cnd}, Condition \ref{AG0-unif-int} places additional restrictions.  Section~\ref{ex-branching} illustrates the interplay between the cost rate function and the dynamics of the inventory process necessary for this condition to be satisfied.

Now define the elementary function $h$ on $\RR$ which is central to the approximation of $G_0$:
$$h(x) = \left\{\begin{array}{cl}
-\frac{1}{8} x^4 + \frac{3}{4} x^2 + \frac{3}{8}, & \quad \mbox{for } |x| \leq 1, \rule[-10pt]{0pt}{12pt} \\
|x|, & \quad \mbox{for } |x| \geq 1.
\end{array}\right.$$
Note that $h \in C^2(\RR)$, $h > 0$ and $h(x) \geq |x|$.  For later reference, elementary calculations show that $h$ has the following properties on $\RR$: 
\begin{itemize}\parskip=2pt
\item $h'<0$ on $(-\infty,0)$ and $h'>0$ on $(0,\infty)$;
\item $h''(x) \geq 0$ for all $x \in \RR$ and $h''(x) = 0$ for $|x| \geq 1$;
\item $0 \leq h(x) - x h'(x) \leq \frac{3}{8}$; and 
\item $h(x) + x h'(x) \geq 0$. 
\end{itemize}

In the next two lemmas, we define a sequence of functions $\{\gn: n\in \NN\} \subset {\cal D}$ which approximate $G_0$ and examine the convergence of $A\gn$ and $B\gn$.  The proofs are given in \apref{appendix-B}.

\begin{lem}  \label{lem:G0-approx}
Assume Conditions \ref{diff-cnd}, \ref{cost-cnds}, \ref{extra-cnd} and \ref{AG0-unif-int} with $G_0$ defined by \eqref{G0-def}.  For each $n \in \NN$, define the function $\gn$ by 
\begin{equation} \label{Gn-def}
\gn = \frac{G_0}{1+\frac{1}{n} h(G_0)}.
\end{equation}
Then $\gn \in {\cal D}$.
\end{lem}

\begin{lem} \label{AGn-BGn-conv}
Assume Conditions \ref{diff-cnd}, \ref{cost-cnds}, \ref{extra-cnd} and \ref{AG0-unif-int}.  Let $G_n$ be as defined in \eqref{Gn-def}.  Then
$$\lim_{n\rightarrow \infty} AG_n(x) = AG_0(x) \quad \forall x \in {\cal I}$$
and 
$$\lim_{n\rightarrow \infty} BG_n(y,z) = BG_0(y,z) \quad \forall (y,z) \in \overline{\cal R}.$$
Moreover, at each boundary where $c_0$ is finite, $\lim_{n\rightarrow \infty} AG_n \geq AG_0$.
\end{lem}

The following proposition gives the first important result involving $AG_n$ and $c_0$.

\begin{prop} \label{AGn-c0-rel}
Assume Conditions \ref{diff-cnd}, \ref{cost-cnds}, \ref{extra-cnd} and \ref{AG0-unif-int} hold.  Let $(\tau,Y) \in {\cal A}_0$ with $J_0(\tau,Y) < \infty$, $X$ satisfy \eqref{controlled-dyn}, $\mu_{0,t}$ be defined by \eqref{mus-t-def} and let $\mu_0$ be any weak limit of $\{\mu_{0,t}\}$ as $t\rightarrow \infty$.  Define $G_n$ by \eqref{Gn-def}.  Then 
$$\liminf_{n\rightarrow \infty} \int_{\overline{\cal E}} (AG_n(x) + c_0(x))\, \mu_0(dx) \geq \int_{\overline{\cal E}} (AG_0(x) + c_0(x))\, \mu_0(dx) \geq F_0^*.$$
\end{prop}

\begin{proof}
Let $(\tau,Y)$, $\{\mu_{0,t}\}$ and $\mu_0$ be as in the statement of the proposition.  Using \eqref{AGn-bound} in Appendix B, write $A\gn = AG_n^{(1)} + e_n^{(2)}  + e_n^{(3)}$. Thus  for $x \in \mathcal I$,
\begin{align*}
A\gn^{(1)}(x)+ c_{0}(x)  & = (-c_{0}(x) + F_{0}^{*} ) \frac{ [1 + \frac{1}{n} h(G_0(x)) - \frac{1}{n} G_0(x) h'(G_0(x))]}{(1 + \frac{1}{n} h(G_0(x)))^2}  + c_{0} (x)\\ 
 & = c_{0}(x) \frac{\frac{1}{n^{2}} (h(G_{0}(x)))^{2}  + \frac1n h(G_{0}(x)) + \frac1n G_{0}(x) h'(G_{0}(x))}{ (1+ \frac1n h(G_{0}(x)))^{2}} \\  & \qquad + F_{0}^{*} \frac{ [1 + \frac{1}{n} h(G_0(x)) - \frac{1}{n} G_0(x) h'(G_0(x))]}{(1 + \frac{1}{n} h(G_0(x)))^2}.
 \end{align*}
 As noted earlier, $h$ satisfies $h(x) + xh'(x) > 0$ and $0 \le h(x) - x h'(x) \le \frac{3}{8}$ for all $x\in \R$. Thus both terms in the rewritten expression for  $AG_n^{(1)}(x)+ c_{0}(x) $ are positive on ${\cal I}$ and are non-negative on $\overline{\cal E}$. 
Therefore Fatou's lemma implies
 \begin{align*}
& \liminf_{n\to \infty} \int_{\overline{\mathcal E}} (AG_n^{(1)}(x)+ c_{0}(x))\, \mu_{0}(dx) \\
  &\ \ \ge \int_{\overline{\mathcal E}} \liminf_{n\to \infty} (AG_n^{(1)}(x)+ c_{0}(x))\, \mu_{0}(dx) \\ 
  & \ \ = \int_{ {\mathcal E}} \liminf_{n\to \infty} (AG_n^{(1)}(x)+ c_{0}(x))\, \mu_{0}(dx) + \int_{\overline{\mathcal E}\backslash\mathcal E} \liminf_{n\to \infty} (AG_n^{(1)}(x)+ c_{0}(x))\, \mu_{0}(dx). \end{align*}
When $c_{0}(a)$ or $c_{0}(b)$ are infinite, the limiting measure $\mu_{0}$ does not place any mass at $a$ or $b$.  Recall, $AG_0 = F_0^*-c_0$ and $|G_0(x)| < \infty$ for all $x\in {\cal I}$.  For all $x\in \overline{\mathcal E} $ with $|G_{0}(x)| < \infty$, it follows that $\lim_{n\to \infty} AG_n^{(1)}(x)  + c_{0}(x)= F_{0}^{*}$.  Turning to the boundary $a$, when $|G_{0}(a)| =\infty$ and $c_{0}(a) < \infty$, then $AG_n^{(1)}(a) = 0$ and hence \cndref{extra-cnd}(a,ii) implies that $AG_n^{(1)}(a) + c_{0}(a) > F_{0}^{*}$.  A similar observation applies to the right boundary $b$ when $|G_0(b)| = \infty$.  Therefore it follows that   
\begin{equation} \label{eq-e1+c0-conv} 
\liminf_{n\to \infty} \int_{\overline{\mathcal E}} (AG_n^{(1)}(x)+ c_{0}(x))\, \mu_{0}(dx) \ge   \int_{{\mathcal E}} F_{0}^{*}\, \mu_{0}(dx)  +  \int_{\overline{\mathcal E}\backslash\mathcal E}  F_{0}^{*}\, \mu_{0}(dx)  = F_0^*.
\end{equation} 
 
Now consider the term $e_n^{(2)}$.  Recall, $h''(x) = 0$ for all $x\in \R$ with $|x| \ge 1 $. When $|G_{0}(a)|=\infty$, then $\lim_{x\to a} e_n^{(2)}(x) =0$; when $|G_{0}(a)| < \infty$,   it follows from Condition \ref{AG0-unif-int}(c) that  $\lim_{x\to a} e_n^{(2)}(x) $ exists and is finite. A similar analysis applies to the right boundary $b$.  From the definition of $e_n^{(2)}$, it therefore follows that $|e_n^{(2)}(x)| \le \frac{K}{n} \leq K < \infty$ uniformly in $x$ and $n$. Thus the bounded convergence theorem implies that 
\begin{equation}
\label{eq-e2-conv}
\liminf_{n\to \infty} \int_{\overline{\mathcal E}} e_n^{(2)}(x)\, \mu_{0}(dx)= \int_{\overline{\mathcal E}}  \liminf_{n\to \infty} e_n^{(2)}(x)\, \mu_{0}(dx)=0.
\end{equation}

It remains to analyze the term $e_n^{(3)}(x)$. If $|G_{0}(a)| = \infty$, then as argued in the proof of Lemma \ref{AGn-BGn-conv}, we must have $G_{0}(a) =-\infty$.  
Consequently there exists some $y_1 > a$ such that $G_{0}(x) < 0$ for all $a \le x < y_1$.  Thus on $(a,y_1)$, $h'(G_{0}(x)) < 0$ and $e_n^{(3)}(x) \ge 0$ and these relations extend by continuity to $a$.  Applying Fatou's lemma, we have
\begin{equation}
\label{eq1-e3-conv}
\liminf_{n\to\infty} \int_{[a, y_{1})} e_n^{(3)}(x)\, \mu_{0}(dx) \ge \int_{[a, y_{1})}   \liminf_{n\to\infty} e_n^{(3)}(x)\, \mu_{0}(dx) \ge 0.
\end{equation} If $|G_{0}(a)|  <  \infty$, then \cndref{AG0-unif-int}(c) implies that $\lim_{x\to a} \sigma(x) G_{0}'(x)$ exists and is finite. This, together with the fact that $h'(x)[1+\frac1n h(x) - \frac1n x h'(x)]$ is uniformly bounded for all $x\in \R$ and $n\in \NN$, implies that there exist a   positive constant $K$  and some $y_{1} > a$ such that $|e_n^{(3)}(x)| \le \frac{K}{n} \le K $ for all $a\le x < y_{1}$ and $n\in \mathbb N$.  As a result, \eqref{eq1-e3-conv} still holds (with ``$\ge $'' replaced by ``$=$'') by virtue of the bounded convergence theorem.
For all $x\in \overline{\mathcal E}$, $|h'(G_0(x))| [1 + \frac{1}{n} h(G_0(x)) - \frac{1}{n} G_0(x) h'(G_0(x))] \leq 2$ so by \cndref{AG0-unif-int}(b) there exists some $z_1< b$ such that for all  $z_1 < x \le b$,   
\begin{displaymath}\begin{aligned}
|e_n^{(3)}(x)|&  \le \frac{2(\sigma(x)G_0'(x))^{2}}{n(1+\frac{1}{n} h(G_0(x)))^{3}} \\ 
&= \frac{2(\sigma(x)G_0'(x))^{2}}{(n+|G_0(x)|)(1+ c_{0}(x))} \cdot\frac{1+c_{0}(x)}{(1+\frac{1}{n} h(G_0(x)))^{2}} \le 2 L (1+ c_{0}(x)).
\end{aligned}\end{displaymath}  
The function $(1+ c_{0}(x))$ is integrable with respect to $\mu_{0}$ and hence the dominated convergence theorem implies that 
\begin{displaymath}
\lim_{n\to\infty} \int_{(z_{1}, b]} e_n^{(3)}(x)\, \mu_{0}(dx) = \int_{(z_{1}, b]}  \lim_{n\to\infty} e_n^{(3)}(x)\, \mu_{0}(dx).
\end{displaymath} 
It is immediate from the definition of $e_n^{(3)}$ that $\lim_{n\to\infty} e_n^{(3)}(x)=0$ for $x \in (z_1, b)$; this limit also holds when $c_{0}(b) < \infty$ and $G_{0}(b)  < \infty$. When $c_{0}(b) < \infty$ and $G_{0}(b) = \infty$, then continuity of $e_n^{(3)}$ along with \cndref{AG0-unif-int}(b) implies that $e_n^{(3)}(b) = \lim_{x\rightarrow b} e_n^{(3)}(x) = 0$. When $c_{0}(b) = \infty$, then $\mu_{0}$ places no mass at $b$.  Thus the dominated convergence theorem again implies  
\begin{equation}
\label{eq2-e3-conv}
\lim_{n\to\infty} \int_{(z_{1}, b]} e_n^{(3)}(x)\, \mu_{0}(dx) = \int_{(z_{1}, b]}  \lim_{n\to\infty} e_n^{(3)}(x)\, \mu_{0}(dx) =0.
\end{equation}
Continuity implies that there exists some $K < \infty$ such that $(\sigma(x)G_0'(x))^2<K$ on the interval $[y_1,z_1]$ so $|e_n^{(3)}(x)| < \frac{2K}{n} \leq 2K$ and $\lim_{n\rightarrow \infty} e_n^{(3)}(x) = 0$.  Thus the bounded convergence theorem also implies that 
\begin{equation}
\label{eq3-e3-conv}
\lim_{n\to\infty} \int_{[y_{1}, z_{1}]} e_n^{(3)}(x)\, \mu_{0}(dx) = \int_{[y_{1}, z_{1}]}  \lim_{n\to\infty}e_n^{(3)}(x)\, \mu_{0}(dx)=0.
\end{equation} 
Combining \eqref{eq1-e3-conv}--\eqref{eq3-e3-conv} yields 
\begin{equation}
\label{eq-e3-conv}
\liminf_{n\to \infty} \int_{\overline{\mathcal E}} e_n^{(3)}(x)\, \mu_{0}(dx) \ge 0. 
\end{equation} 
In light of \eqref{AGn-bound} in Appendix B, the result now follows from \eqref{eq-e1+c0-conv}, \eqref{eq-e2-conv}, and \eqref{eq-e3-conv}. 
\end{proof}

We next establish a similar result involving $BG_n$ and $c_1$, though the lack of tightness of $\{\mu_{1,t}\}$ means that the result cannot be expressed in terms of a limiting measure $\mu_1$.

\begin{prop} \label{BGn-c1-rel}
Assume Conditions \ref{diff-cnd}, \ref{cost-cnds}, \ref{extra-cnd} and \ref{AG0-unif-int} hold.  Let $(\tau,Y) \in {\cal A}_0$ with $J_0(\tau,Y) < \infty$ and $X$ satisfy \eqref{controlled-dyn}.  Let $\{t_j:j\in \NN\}$ be a sequence such that $\lim_{j\rightarrow \infty} t_j= \infty$ and 
$$J_0(\tau,Y) = \lim_{j\rightarrow \infty} \mbox{$\frac{1}{t_j}$}\EE\left[\int_0^{t_j} c_0(X(s))\, ds + \sum_{k=1}^\infty I_{\{\tau_k\leq t_j\}} c_1(X(\tau_k-),X(\tau_k))\right].$$ 
For each $j$, define $\mu_{1,t_j}$ by \eqref{mus-t-def} and define $G_n$ by \eqref{Gn-def}.  Then 
\begin{equation}  \label{Fatou-bdd} 
\liminf_{n\rightarrow \infty} \liminf_{j\rightarrow \infty} \int_{\overline{\cal R}} (BG_n(y,z) + c_1(y,z))\, \mu_{1,t_j}(dy\times dz) \geq 0.
\end{equation}
\end{prop}

\begin{proof}
Let $(\tau,Y)$, $X$ and $\{t_j\}$ be as in the statement of the proposition.  Observe that for $(y,z) \in \mathcal{R}$, 
\begin{align} 
\nonumber c_1(y,z) + B\gn(y,z) &= c_1(y,z) + \frac{G_0(z)}{1+\frac{1}{n} h(G_0(z))} - \frac{G_0(y)}{1+\frac{1}{n} h(G_0(y))}  \\
\nonumber &= \frac{BG_{0}(y,z)+c_1(y,z)}{[1+\frac{1}{n} h(G_0(z))][1+\frac{1}{n} h(G_0(y))]}  +\frac{G_{0}( z) h(G_{0}(y)) - G_{0}(y) h(G_{0}(z))}{n [1+\frac{1}{n} h(G_0(z))][1+\frac{1}{n} h(G_0(y))]}\\
\nonumber &  \ \ \ +\;c_1(y,z) \left(1 - \frac{1}{[1+\frac{1}{n} h(G_0(z))][1+\frac{1}{n} h(G_0(y))]}\right) \\
&\geq \frac{G_{0}( z) h(G_{0}(y)) - G_{0}(y) h(G_{0}(z))}{n [1+\frac{1}{n} h(G_0(z))][1+\frac{1}{n} h(G_0(y))]} =: R_n(y,z). \label{eq-Gn defn}
\end{align}
The first summand in the middle relation is positive due to \propref{qvi-ish} and the third summand is easily seen to be positive.  This relation also holds on the boundary $y=z$ and $R_n(y,y) = 0$ for all $y\in \mathcal{E}$. Therefore, 
$$\liminf_{n\rightarrow \infty} \liminf_{j\rightarrow \infty} \int_{\overline{\cal R}} (B\gn + c_1)(y,z)\, \mu_{1,t_j}(dy\times dz) \geq \liminf_{n\rightarrow \infty} \liminf_{j\rightarrow \infty} \int_{\overline{\cal R}} R_n(y,z)\, \mu_{1,t_j}(dy\times dz).$$
Notice that on $\overline{\cal R}$, $R_n$ takes both positive and negative values.

We examine the double limit of the remainder term $R_n$ in several cases.  Recall, \cndref{extra-cnd} implies the existence of $G_0(a) = \lim_{x\rightarrow a} G_0(x)$ in $\RR\cup\{-\infty\}$ and similarly, the existence of $G_0(b) \in \RR\cup \{\infty\}$.
\medskip

\noindent
{\em Case (i): $G_0(a) > -\infty$ and $G_0(b) < \infty$.}\/ In this case, $|G_{0}(z) h(G_{0}(y)) - G_{0}(y) h(G_{0}(z))|$ is bounded by some $K < \infty$ and hence $|R_n(y,z)| \leq \frac{K}{n}$.  Recalling \remref{masses-observation}, $\{\mu_{1,t_j}(\overline{\cal R})\}$ is uniformly bounded and thus 
$$\liminf_{n\rightarrow \infty} \liminf_{j\rightarrow \infty} \int_{\overline{\cal R}} R_n(y,z)\, \mu_{1,t_j}(dy\times dz) \geq \liminf_{n\rightarrow \infty} \liminf_{j\rightarrow \infty} \int_{\overline{\cal R}} \mbox{$-\frac{K}{n}$}\, \mu_{1,t_j}(dy\times dz) = 0.$$

\noindent
{\em Case (ii): $G_0(a) = -\infty$ and $G_0(b) = \infty$.}\/  Since $G_0(a) = -\infty$, there exists some $y_1$, with $y_1 > a$ such that $G_0(x) < -1$ for all $x < y_1$.  Recall $h(x) = |x|$ on $(-\infty,-1)$ and $h(x) \geq |x|$ for all $x$.  Thus it follows that for all $(y,z) \in {\cal R}$ with $y \leq y_1$,
$$R_n(y,z) = \frac{|G_0(y)|(G_0(z) + h(G_0(z)))}{n [1+\frac{1}{n} h(G_0(z))][1+\frac{1}{n} |G_0(y)|]} \geq 0.$$
Similarly, there exists some $z_1$ with $ z_1 < b$ such that $G_0(x) \geq 1$ for $z_1 < x < b$ and it follows that for $(y,z) \in \overline{\cal R}$ with $z > z_1$, 
$$R_n(y,z) = \frac{G_0(z)(h(G_0(y)) - G_0(y))}{n [1+\frac{1}{n} h(G_0(z))][1+\frac{1}{n} h(G_0(y))]} \geq 0.$$
Define $E_1 = \{(y,z)\in \overline{\cal R}: a < y \leq y_1\}$ and $E_2 = \{(y,z)\in \overline{\cal R}: z > z_1\}$ so that $R_n \geq 0$ on $E_1\cup E_2$.  Observe that $E_3:=\overline{\cal R}\backslash (E_1\cup E_2) = \{(y,z)\in \overline{\cal R}: y_1 \leq y \leq z \leq z_1\}$, $R_n$ defined in \eqref{eq-Gn defn} is continuous on this compact set and therefore bounded by some $\frac{K}{n}$ as in case (i).  It then follows that
$$\liminf_{n\rightarrow \infty} \liminf_{j\rightarrow \infty} \int_{\overline{\cal R}} R_n(y,z)\, \mu_{1,t_j}(dy\times dz) 
\geq \liminf_{n\rightarrow \infty} \liminf_{j\rightarrow \infty} \int_{E_3} R_n(y,z)\, \mu_{1,t_j}(dy \times dz) = 0.$$
\medskip

The two remaining cases are handled similarly.
\end{proof}

Pulling these results together, we obtain our main theorem.

\begin{thm} \label{G0-feasible}
Assume Conditions \ref{diff-cnd}, \ref{cost-cnds}, \ref{extra-cnd} and \ref{AG0-unif-int} hold.  Let $(\tau,Y) \in {\cal A}_0$ with $J_0(\tau,Y) < \infty$.  
Then
$$J_0(\tau,Y) \geq F_0^* = F_0(\yzstar,\zzstar) = J_0(\tau^*,Y^*)$$
in which $(\tau^*,Y^*)$ is the ordering policy \eqref{sS-tau-def} using an optimizing pair $(\yzstar,\zzstar) \in {\cal R}$.
\end{thm}

\begin{proof}
Let $(\tau,Y)\in {\cal A}_0$ satisfy $J_0(\tau,Y) < \infty$.  Let $X$ satisfy \eqref{controlled-dyn}, $\mu_{0,t}$ and $\mu_{1,t}$ be defined by \eqref{mus-t-def} for each $t > 0$.  Let $\{t_j\}$ be a sequence with $t_j \rightarrow \infty$ and
\begin{eqnarray} \label{cost-representation} \nonumber
J_0(\tau,Y) &=& \lim_{j\rightarrow \infty} \mbox{$\frac{1}{t_j}$} \EE\left[\int_0^{t_j} c_0(X(s))\, ds + \sum_{k=1}^\infty I_{\{\tau_k \leq t_j\}} c_1(X(\tau_k-),X(\tau_k))\right] \\
&=& \lim_{j\rightarrow \infty} \left(\int_{\overline{\cal E}} c_0(x)\, \mu_{0,t_j}(dx) + \int_{\overline{\cal R}} c_1(y,z)\, \mu_{1,t_j}(dy\times dz)\right).
\end{eqnarray}
The tightness of $\{\mu_{0,t_j}\}$ implies the existence of a weak limit $\mu_0$; without loss of generality, assume $\mu_{0,t_j} \Rightarrow \mu_0$ as $j\rightarrow \infty$.  \propref{J-finite} and its proof establish that 
$$\int_{\overline{\cal E}} c_0\, d\mu_0 \leq \liminf_{j\rightarrow \infty} \int_{\overline{\cal E}} c_0\, d\mu_{0,t_j} \leq J_0(\tau,Y) < \infty.$$

Since $\gn \in {\cal D}$, $\displaystyle\lim_{j\rightarrow \infty} \int_{\overline{\cal E}} A\gn\, d\mu_{0,t_j} = \int_{\overline{\cal E}} A\gn\, d\mu_0$.  \propref{bdd-adjoint} implies that for each $n$,
\begin{equation} \label{Gn-bar}
\lim_{j\rightarrow \infty} \left(\int_{\overline{\cal E}} A\gn(x)\, \mu_0(dx) + \int_{\overline{\cal R}} B\gn(y,z)\, \mu_{1,t_j}(dy\times dz)\right) = 0
\end{equation}
so adding \eqref{cost-representation} and \eqref{Gn-bar}  and taking the limit inferior as $n\rightarrow \infty$ yields, 
\begin{align*}
J&_0 (\tau,Y) \\ &= \liminf_{n\rightarrow \infty} \lim_{j\rightarrow \infty} \left(\int_{\overline{\cal E}} (A\gn(x) + c_0(x))\, \mu_{0,t_j}(dx) 
+ \int_{\overline{\cal R}} (B\gn(y,z) + c_1(y,z))\, \mu_{1,t_j}(dy\times dz)\right) \\
&\geq \liminf_{n\rightarrow \infty} \liminf_{j\rightarrow \infty} \int_{\overline{\cal E}} (A\gn(x) + c_0(x))\, \mu_{0,t_j}(dx) \\&  \qquad 
+\; \liminf_{n\rightarrow \infty} \liminf_{j\rightarrow \infty} \int_{\overline{\cal R}} (B\gn(x) + c_1(y))\, \mu_{1,t_j}(dy \times dz) \\
&\geq \liminf_{n\rightarrow \infty} \int_{\overline{\cal E}} (A\gn(x) + c_0(x))\, \mu_0(dx)  
+\; \liminf_{n\rightarrow \infty} \liminf_{j\rightarrow \infty} \int_{\overline{\cal R}} (B\gn(x) + c_1(y))\, \mu_{1,t_j}(dy \times dz) \\
&\geq F_0^*;
\end{align*}
Propositions \ref{AGn-c0-rel} and \ref{BGn-c1-rel} establish the last inequality.
\end{proof}

\section{Examples} \label{sect:examples}
For the cost structure, we consider four examples to illustrate some of the different possible ways in which \propref{no-optimizers} and \thmref{G0-feasible} can be employed.  Some of these cost structures include forms that do not satisfy the modularity condition imposed in \cite{helm:15b}, thus illustrating the breadth of application of the current results.

\subsection{Drifted Brownian motion inventory models}
We begin by examining the classical model that has been studied by \cite{bather-66}, \cite{sulem-86}, \cite{DaiY-13-average}, and many others.  In particular, we show that \thmref{G0-feasible} extends the result in \cite{he:arxiv} and verifies optimality as a result of our analytical approach.   We then examine a drifted Brownian motion process with reflection at $\{0\}$ using a non-traditional cost structure.  

\subsubsection{Classical model}
In the absence of ordering, the inventory level process $X_0$ evolves in the state space $\RR$ and satisfies
\begin{equation} \label{dbm-dyn}
dX_0(t) = -\mu\, dt + \sigma\, dW(t), \quad X_0(0) = x_0\in \I:=(-\infty, \infty)
\end{equation}
in which $\mu, \sigma > 0$ and $W$ is a standard Brownian motion process.  Observe the model includes both positive and negative inventory levels indicating real inventory and items that have been back-ordered, respectively.  

To specify the cost structure,  the holding/back-order cost function $c_0$ defined on $\RR$ and the ordering cost function $c_1$ defined on $\overline{\mathcal{R}} = \{(y,z)\in \RR^2: y\leq z\}$ are
\begin{equation} \label{dbm-cost} 
c_0(x) = \left\{\begin{array}{rl}
-c_b\, x, & \quad x < 0, \\
 c_h\, x, & \quad x \geq 0,
\end{array}\right. \hspace{5 mm} \mbox{ and } \hspace{5 mm}
c_1(y,z) = k_1 + k_2(z-y) 
\end{equation}
in which $c_b, c_h, k_1, k_2 > 0$.  The coefficient $c_b$ denotes the back-order cost rate per unit of inventory per unit of time while $c_h$ is the holding cost rate.   The ordering cost function is comprised of fixed plus proportional (to the order size) costs.

Section 4.1 of \cite{helm:15b} verifies Conditions~\ref{diff-cnd} and \ref{cost-cnds} hold and proves the existence of an optimizing pair $(\yzstar,\zzstar) \in \mathcal{R}$ such that $F_0(\yzstar,\zzstar) = F_0^*$.

Since $a=-\infty$ is a natural boundary, $\A_0=\A$.  The following theorem establishes the optimality of the $(\yzstar,\zzstar)$-policy in the class ${\cal A}$ {\em without using an ad hoc comparison result} as in \cite{he:arxiv}.

\begin{thm} \label{dbm-cnd5-5}
Let $(\yzstar,\zzstar)$ be an optimizing pair of $F_0$.  Then the $(\yzstar,\zzstar)$ ordering policy defined by \eqref{sS-tau-def} is optimal in the class $\A$ for the drifted Brownian motion inventory model under the cost structure \eqref{dbm-cost}.
\end{thm}

\begin{proof}
We need to verify \cndref{AG0-unif-int} holds in order to apply \thmref{G0-feasible} to show that the $(\yzstar,\zzstar)$ ordering policy is optimal.

First, define the functions $\tilde{g}_0$ and $\tilde{\zeta}$ on $\RR$ by
\begin{align*}
\tilde{g}_0(x) &= \displaystyle\int_0^x \int_u^\infty 2c_0(v)\, dM(v)\, dS(u) = \begin{cases}
-\frac{c_b}{2\mu}\, x^2 - \frac{\sigma^2 c_b}{2\mu^2}\, x + \frac{\sigma^4(c_b+c_h)}{4\mu^3} \left(\exp\{\frac{2\mu}{\sigma^2} x\} - 1\right), &\!\!\! x < 0, \\
\frac{c_h}{2\mu}\, x^2 + \frac{\sigma^2 c_h}{2\mu^2}\, x, &\!\!\! x \geq 0 \\
\end{cases}\\
\tilde{\zeta}(x) &= \displaystyle\int_0^x \int_u^\infty 2\, dM(v)\, dS(u) \quad\quad \, = \mbox{$\frac{1}{\mu}$} x.
\end{align*}
Note that $\tilde{g}_{0}, \tilde{\zeta} \in C^{2}(\I)$.  It therefore follows from \eqref{g0-fn} that
$$g_0(x) = \tilde{g}_0(x) - \tilde{g}_0(x_0) = \begin{cases}
-\frac{c_b}{2\mu}\, x^2 - \frac{\sigma^2 c_b}{2\mu^2}\, x + \frac{\sigma^4(c_b+c_h)}{4\mu^3} \left(\exp\{\frac{2\mu}{\sigma^2} x\} - 1\right) - \tilde{g}_0(x_0), &  x < 0, \\
\frac{c_h}{2\mu}\, x^2 + \frac{\sigma^2 c_h}{2\mu^2}\, x - \tilde{g}_0(x_0), &  x \geq 0,
\end{cases}
$$ 
and $\zeta(x) = \tilde\zeta(x) - \tilde{\zeta}(x_0) = \frac{1}{\mu} (x - x_0)$.  Defining $\tilde{G}_0 = \tilde{g}_0 - F_0^* \tilde{\zeta}$, it follows that
\begin{eqnarray*}
G_0(x) &=& \tilde{G}_0(x) - \tilde{G}_0(x_0) \\
&=& \begin{cases}
-\frac{c_b}{2\mu}\, x^2 - \left(\frac{\sigma^2 c_b}{2\mu^2}+\frac{F_0^*}{\mu}\right) x + \frac{\sigma^4(c_b+c_h)}{4\mu^3} \left(\exp\{\frac{2\mu}{\sigma^2} x\} - 1\right) - \tilde{G}_0(x_0), & \quad x < 0, \\
\frac{c_h}{2\mu}\, x^2 + \left(\frac{\sigma^2 c_h}{2\mu^2}-\frac{F_0^*}{\mu}\right) x  - \tilde{G}_0(x_0), & \quad x \geq 0.
\end{cases}
\end{eqnarray*}
We examine the cases $x < 0$ and $x > 0$ separately.  For $x > 0$, we have the two conditions 
\begin{equation} \label{dbm-G0-estimates} 
\begin{array}{rcl} \displaystyle 
\frac{(\sigma(x) G_0'(x))^2}{(1+|G_0(x)|)(1+c_0(x))} &=& \displaystyle \frac{\left(\frac{c_h \sigma}{\mu}\, x + \frac{\sigma^3 c_h}{2\mu^2} - \frac{F_0^*\sigma}{\mu}\right)^2}{\left(1+\left|\frac{c_h}{2\mu}\, x^2 + \left(\frac{\sigma^2 c_h}{2\mu^2}-\frac{F_0^*}{\mu}\right) x - \tilde{G}_0(x_0)\right|\right)(1+c_hx)},  \\ \displaystyle 
\frac{c_0(x)}{(1+|G_0(x)|)^2} &=& \displaystyle \frac{c_h x}{\left(1+\left|\frac{c_h}{2\mu}\, x^2 + \left(\frac{\sigma^2 c_h}{2\mu^2}-\frac{F_0^*}{\mu}\right) x - \tilde{G}_0(x_0)\right|\right)^2}, \rule{0pt}{22pt}
\end{array}
\end{equation}
and hence 
$$\lim_{x\to\infty} \frac{(\sigma(x) G_0'(x))^2}{(1+|G_0(x)|)(1+c_0(x))} = 0, \mbox{ and } \lim_{x\rightarrow \infty} \frac{c_0(x)}{(1+|G_0(x)|)^2} = 0.$$
For $x < 0$, the ratios are
\begin{eqnarray*}
\frac{(\sigma(x) G_0'(x))^2}{(1+|G_0(x)|)^3} &=& \frac{\left(\frac{-c_b \sigma}{\mu}\, x - \left(\frac{\sigma^3 c_b}{2\mu^2}+\frac{F_0^*\sigma}{\mu}\right) + \frac{\sigma^2(c_b+c_h)}{2\mu^2} \exp(\frac{2\mu}{\sigma^2}x)\right)^2}{\left(1+\left|\frac{-c_b}{2\mu}\, x^2 - \left(\frac{\sigma^3 c_b}{2\mu^2}+\frac{F_0^*\sigma}{\mu}\right) x + \frac{\sigma^4(c_b+c_h)}{4\mu^3}(\exp(\frac{2\mu}{\sigma^2}x)-1) - \tilde{G}_0(x_0)\right|\right)^3}, \\
\frac{c_0(x)}{(1+|G_0(x)|)^2} &=& \frac{-c_b x}{\left(1+\left|\frac{-c_b}{2\mu}\, x^2 - \left(\frac{\sigma^3 c_b}{2\mu^2}+\frac{F_0^*\sigma}{\mu}\right) x + \frac{\sigma^4(c_b+c_h)}{4\mu^3}(\exp(\frac{2\mu}{\sigma^2}x)-1) - \tilde{G}_0(x_0)\right|\right)^2},
\end{eqnarray*}
yielding 
$$\lim_{x\rightarrow -\infty} \frac{(\sigma(x) G_0'(x))^2}{(1+|G_0(x)|)^3} = 0, \mbox{ and } \lim_{x\rightarrow -\infty} \frac{c_0(x)}{(1+|G_0(x)|)^2} = 0.$$
As a result, both ratios are uniformly bounded and the result holds.
\end{proof}

\subsubsection{Drifted Brownian motion with reflection at $\{0\}$}
For this model, $X$ is a drifted Brownian motion process that reflects at $\{0\}$.  Thus when $X(t) > 0$, $X$ follows the dynamics in \eqref{dbm-dyn} and \cndref{diff-cnd} is established in Section~4.1 of \cite{helm:15b}.  

The cost structure is given by the functions $c_0$ on $[0,\infty)$ and $c_1$ on $\{(y,z): 0\leq y \leq z\}$ by 
\begin{equation} \label{rdbm-cost}
\begin{aligned}
c_0(x) =  k_3 x + k_4 e^{-x}, \ \text{ and }\
c_1(y,z) = k_1 + k_2\sqrt{z-y},
\end{aligned}
\end{equation}
with $k_1,k_2,k_3,k_4 > 0$.  Notice that the holding costs for small inventory levels are penalized by the term $k_4e^{-x}$ which remains bounded at $0$.  Also, the ordering cost function is a concave function of the order size so incorporates savings due to economies of scale.  This $c_1$ function does not satisfy the critical modularity condition required in \cite{helm:15b}.  We begin by establishing the existence of a pair of minimizers for $F_0$.

The analysis below verifies the existence of an optimal $(s,S)$ ordering policy in the class $\mathcal{A}_0$ which does not allow any positive amount of long-term average expected local time of the inventory process at $\{0\}$.

\begin{prop} \label{rdbm-F-optimizers}
There exists an optimizing pair $(\yzstar,\zzstar) \in {\cal R}$ of $F_0$ for the drifted Brownian motion model with reflection having cost structure \eqref{rdbm-cost}.
\end{prop}

\begin{proof}
Verification of Condition~\ref{cost-cnds} follows by straightforward computations so is left to the reader.  By assumption on the model, the boundary $0$ is regular (reflective) while the boundary $\infty$ is natural with $c_0(\infty) = \infty$.  Thus \cndref{extra-cnd} holds and an application of \thmref{F-optimizers} establishes the result.
\end{proof}

\begin{thm}
Let $(\yzstar,\zzstar)$ be as in \propref{rdbm-F-optimizers}.  Then the $(\yzstar,\zzstar)$ ordering policy is optimal in the class $\A_0$ for the reflected drifted Brownian motion inventory model under the cost structure \eqref{rdbm-cost}.
\end{thm}

\begin{proof}
It suffices to show that \cndref{AG0-unif-int} holds.  The proof of \thmref{dbm-cnd5-5} shows that $\zeta(x) = \tilde{\zeta}(x)- \tilde{\zeta}(x_0)$ with $\tilde{\zeta}(x) = \frac{1}{\mu} x$ and, as in the same proof, straightforward computation establishes that
$$\tilde{g}_0(x) = \mbox{$\frac{k_3}{2\mu} x^2 + \frac{k_3\sigma^2}{2\mu^2} x + \frac{2k_4\sigma^2}{(2\mu+\sigma^2)^2}(1-e^{-x})$}$$
so $g_0(x) = \tilde{g}_0(x) - \tilde{g}_0(x_0)$.  Setting $\tilde{G}_0 = \tilde{g}_0 - F_0^* \tilde{\zeta}$, the function $G_0$ is 
$$G_0(x) = \mbox{$\frac{k_3}{2\mu} x^2 + \left(\frac{k_3\sigma^2}{2\mu^2}-\frac{F_0^*}{\mu}\right) x + \frac{2k_4\sigma^2}{(2\mu+\sigma^2)^2}(1-e^{-x})$}- \tilde{G}_0(x_0).$$
At the left boundary $0$, $c_0(0) = k_4$ and $G_0(0) = \frac{2k_4\sigma^2}{(2\mu+\sigma^2)^2} - \tilde{G}_0(x_0)$. It immediately follows from continuity that \cndref{AG0-unif-int}(a,ii) holds and direct computation verifies that \cndref{AG0-unif-int}(c,ii) is satisfied.  We thus need to examine the ratios of \cndref{AG0-unif-int} when $x$ is large.  These ratios differ only slightly from \eqref{dbm-G0-estimates} (equating $c_h$ with $k_3$) and the exponential term does not affect the limits.  Thus \cndref{AG0-unif-int} holds and \thmref{G0-feasible} then yields the results.
\end{proof}

\begin{rem}
Consider models in which $(k_1\vee\frac{k_2}{2})\wedge \frac{k_3}{2\mu} \geq \frac{2k_4\sigma^2}{(2\mu+\sigma^2)^2}$ so that either the fixed or unit ordering cost is expensive relative to the holding cost rate near $0$ and the linear holding cost rate is expensive relative to this same holding cost rate near $0$.  Taking the partial derivative with respect to $y$ yields
\begin{align*}
\frac{\partial F_0}{\partial y}(y,z)  &= \frac{\mu}{z-y} \bigg[\frac{k_1}{ z-y} +  \frac{k_2}{2\sqrt{z-y}} +  \frac{k_3}{2\mu}  (z-y) -  \frac{2k_4\sigma^2}{(2\mu+\sigma^2)^2}  e^{-y} 
+  \frac{2k_4\sigma^2}{(2\mu+\sigma^2)^2}  \frac{e^{-y}-e^{-z}}{z-y}\bigg]\!.
\end{align*}
Since $y \geq 0$, observe that
\begin{itemize}
\item  for all $(y,z) \in {\cal R}$ with $z-y \leq 1$, $\frac{k_1}{z-y} \geq \frac{2k_4\sigma^2}{(2\mu+\sigma^2)^2}\, e^{-y}$ or $\frac{k_2}{2\sqrt{z-y}} \geq \frac{2k_4\sigma^2}{(2\mu+\sigma^2)^2}\, e^{-y}$; and
\item $\frac{k_3}{2\mu}\,(z-y) \geq \frac{2k_4\sigma^2}{(2\mu+\sigma^2)^2}\, e^{-y}$ for all $(y,z) \in {\cal R}$ with $z-y > 1$.  
\end{itemize}
The positivity of the last term in the brackets of $\frac{\partial F_0}{\partial y}$ therefore implies that $\frac{\partial F_0}{\partial y}(y,z) > 0$ for all $(y,z) \in {\cal R}$ and hence $y_0^*=0$.  Thus the optimal ordering policy waits until the inventory hits $0$ before ordering.  (The optimal level $z_0^*$ is determined from the transcendental equation obtained by setting $\frac{\partial F_0}{\partial z}(0,z) = 0$.)
\end{rem}

\subsection{Geometric Brownian motion storage model} \label{sect:gBM}
Without any ordering, the inventory process is a geometric Brownian motion evolving in the state space $(0,\infty)$ and satisfying the stochastic differential equation 
$$dX_0(t) = - \mu X_0(t)\, dt + \sigma X_0(t) \, dW(t), \qquad X(0) = x_0\in \I=(0,\infty),$$
in which $\mu, \sigma > 0$ and $W$ is a standard Brownian motion process; the drift rate is negative.  It is well-known that both boundaries are natural.  Section~4.2 of \cite{helm:15b} shows that this model satisfies Condition~\ref{diff-cnd}.

\subsubsection{Nonlinear holding costs with ordering costs that are a concave function of order size}
The cost functions $c_0$ on $(0,\infty)$ and $c_1$ on $\{(y,z): 0<  y \leq z\}$ are given by 
\begin{equation} \label{gbm-nonlinear}
\begin{aligned}
c_0(x) = k_3 x + k_4 x^\beta, \ \text{ and } \ 
c_1(y,z) = k_1 + k_2\sqrt{z-y}, 
\end{aligned} 
\end{equation}
in which $k_1, k_2, k_3, k_4 > 0$ and $\beta < 0$.  For this $c_0$ function, $k_4 x^\beta$ is extremely costly as the level approaches $0$.  The proof of the next lemma is straightforward so is left to the reader.

\begin{lem}
\cndref{cost-cnds} holds for this geometric Brownian motion storage model having cost structure \eqref{gbm-nonlinear}.
\end{lem}

We now turn to the existence of an optimal $(s,S)$ ordering policy.

\begin{thm}
There exists an optimal $(\yzstar,\zzstar)$ ordering policy in the class ${\cal A}$ for the geometric Brownian motion storage model having nonlinear cost structure given by \eqref{gbm-nonlinear}.
\end{thm}

\begin{proof}
By \thmref{G0-feasible}, it suffices to show that Conditions \ref{extra-cnd} and \ref{AG0-unif-int} hold.  As previously noted, $0$ and $\infty$ are natural boundaries and we see that $c_0(0) = \infty$ and $c_0(\infty) = \infty$ so \cndref{extra-cnd}(a,i) and (b,i) hold.  

Let $\rho = \frac{\sigma^2 \beta^2}{2} - \mu \beta$ and note $\rho > 0$. Straightforward computations using \eqref{g0-fn} 
determine 
\begin{displaymath}  
\begin{aligned}
& g_0(x) = \mbox{$\frac{k_3}{\mu} (x - x_0) - \frac{k_4}{\rho} (x^\beta - x_0^\beta)$}, \quad  &  \text{ and }\quad 
& \zeta(x) = \mbox{$\frac{2}{2\mu+\sigma^2}$} (\ln(x) - \ln(x_0)) , \quad  &  x \in \I,
\end{aligned} 
\end{displaymath}
and hence
$$G_0(x) = \mbox{$\frac{k_3}{\mu} (x - x_0) - \frac{k_4}{\rho} (x^\beta - x_0^\beta) - \frac{2F_0^*}{2\mu+\sigma^2} (\ln(x) - \ln(x_0))$}.$$
As a result, we have
$$\frac{(\sigma(x) G_0'(x))^2}{(1+|G_0(x)|)^3} = \frac{\left(\frac{k_3\sigma}{\mu} x + \frac{k_4 \sigma (-\beta)}{\rho} x^\beta - \frac{2F_0^*}{\mu+(\sigma^2/2)}\right)^2}{\left(1+\left|\frac{k_3}{\mu} (x - x_0) - \frac{k_4}{\rho} (x^\beta - x_0^\beta) - \frac{2F_0^*}{2\mu+\sigma^2} (\ln(x) - \ln(x_0))\right|\right)^3}$$
so at the left boundary $a=0$, 
\begin{equation} \label{estimate1a}
\lim_{x\rightarrow 0} \frac{(\sigma(x) G_0'(x))^2}{(1+|G_0(x)|)^3} = 0.
\end{equation}
Turning to the boundary $b=\infty$,
\begin{eqnarray*}
\lefteqn{\frac{(\sigma(x) G_0'(x))^2}{(1+|G_0(x)|)(1+c_0(x))}} \\
&=& \frac{\left(\frac{k_3\sigma}{\mu} x + \frac{k_4 \sigma (-\beta)}{\rho} x^\beta - \frac{2F_0^*}{\mu+(\sigma^2/2)}\right)^2}{\left(1+\left|\frac{k_3}{\mu} (x - x_0) - \frac{k_4}{\rho} (x^\beta - x_0^\beta) - \frac{2F_0^*}{2\mu+\sigma^2} (\ln(x) - \ln(x_0))\right|\right)(1+k_3x+k_4x^\beta)}
\end{eqnarray*}
so
\begin{equation} \label{estimate1b}
\lim_{x\rightarrow \infty} \frac{(\sigma(x) G_0'(x))^2}{(1+|G_0(x)|)(1+c_0(x))} = \mbox{$\frac{\sigma^2}{\mu}$}.
\end{equation}

Finally, we have
$$\frac{c_0(x)}{(1+|G_0(x)|)^2} = \frac{k_3 x + k_4 x^\beta}{\left(1+\left|\frac{k_3}{\mu} (x - x_0) - \frac{k_4}{\rho} (x^\beta - x_0^\beta) - \frac{2F_0^*}{2\mu+\sigma^2} (\ln(x) - \ln(x_0))\right|\right)^2}$$
which gives
\begin{equation} \label{estimate3}
\lim_{x\rightarrow 0} \frac{c_0(x)}{(1+|G_0(x)|)^2} = 0 \quad \mbox{and} \quad \lim_{x\rightarrow \infty} \frac{c_0(x)}{(1+|G_0(x)|)^2} = 0.
\end{equation}
As a result, \eqref{estimate1a}, \eqref{estimate1b} and \eqref{estimate3} imply that \cndref{AG0-unif-int}(a,i) and (b,i) hold.
\end{proof}

\begin{rem}
When $k_4=0$ in \eqref{gbm-nonlinear}, the conditions of \propref{no-optimizers}(a) hold so there does not exist any optimal $(s,S)$ ordering policy and the ``no order'' policy is optimal.  See also Section~4.2.2 of \cite{helm:15b} for an alternate analysis.
\end{rem}
\subsubsection{Piecewise linear holding costs with modular ordering costs}
The cost functions $c_0$ on $[0,\infty)$ and $c_1$ on $\{(y,z)\in \RR_+^2: y\leq z\}$ are given by
\begin{equation} \label{gbm-piecewise-linear}
\begin{array}{rcl}
c_0(x) &=& \left\{\begin{array}{ll}
k_4(1-x), & \quad 0 \leq x \leq 1, \\
k_3 (x-1), & \quad x \geq 1
\end{array}\right. \\
c_1(y,z) &=& k_1 + \frac{k_2}{2}(y^{-1/2} - z^{-1/2}) + \frac{k_2}{2} (z-y) \rule{0pt}{12pt}
\end{array}
\end{equation}
with $k_1, k_2, k_3 > 0$ and $k_4 > (\sigma^{2}+ 2\mu)\left[k_{1} + \frac{k_{2}}{2}(1-e^{-\frac12}) + \frac{k_{2}\mu + 2 k_{3}}{2\mu} (e-1)\right] -2k_{3}$.  Note $\lim_{y\rightarrow 0}c_1(y,z) = \infty$, placing a very strong penalty on waiting until the inventory level is nearly $0$.  In this example, the holding cost rate decreases at rate $k_4$ on $[0,1]$ and increases at rate $k_3$ thereafter.  The function $c_1$ is modular in the sense that for any $0 < w \leq x \leq y \leq z$, $c_1(x,z) - c_1(x,y) - c_1(w,z) + c_1(w,y) = 0$.  By writing $c_1$ in the form
$$c_1(y,z) = k_1 + k_2 \left(\frac{\frac{(y^{-1/2}-z^{-1/2})}{z-y}+1}{2}\right) (z-y)$$
the cost per unit portion of the charge is smaller with large values of $y$ and $z$ so encourages the decision maker to order before the inventory falls very low and prefers large orders.

We note that the modularity condition on $c_1$ is a critical condition in \cite{helm:15b}.  That paper also requires $c_0(x)\rightarrow \infty$ as $x\rightarrow 0$, which is violated in \eqref{gbm-piecewise-linear}.

The proof of the next result is straightforward and left to the reader.

\begin{lem}
The geometric Brownian motion model with cost structure given in \eqref{gbm-piecewise-linear} satisfies \cndref{cost-cnds}.
\end{lem}

\begin{prop} \label{gbm-pl-F-optimizer}
The geometric Brownian motion model with cost functions given by \eqref{gbm-piecewise-linear} satisfies \cndref{extra-cnd} and hence there exists an optimizing pair $(\yzstar,\zzstar) \in {\cal R}$ of the function $F_0$.
\end{prop}

\begin{proof}
Begin by defining the functions $\tilde{\zeta}(x) = \frac{2}{2\mu+\sigma^2}\ln(x)$,   $x\in(0,\infty)$ and $\tilde{g}_0$ on $(0,\infty)$ by
\begin{eqnarray*}
\tilde{g}_0(x) &=& \int_1^x \int_u^\infty 2c_0(v)\, dM(v)\, dS(u) \\
 &=& \left\{\begin{array}{ll}
\frac{2k_4}{2\mu+\sigma^2} \ln(x) - \frac{k_4}{\mu} (x-1) + \frac{\sigma^4(k_3+k_4)}{\mu(2\mu+\sigma^2)^2} (x^{(2\mu+\sigma^2)/\sigma^2} - 1), & \quad x < 1 \\
\frac{k_3}{\mu} (x-1) - \frac{2k_3}{2\mu+\sigma^2}\ln(x), & \quad x \geq 1. \rule{0pt}{12pt}
\end{array}\right.
\end{eqnarray*}
It then follows from \eqref{g0-fn} 
that $g_0(x) = \tilde{g}_0(x) - \tilde{g}_0(x_0)$ and $\zeta(x) = \tilde{\zeta}(x) - \tilde{\zeta}(x_0)$. 
 Note that $g_0' = \tilde{g}_0'$ and $\zeta' = \tilde{\zeta}'$.

Since $0$ is a natural boundary and $c_0(0) = k_4 < \infty$, it suffices to show that for each $z \in (0,\infty)$, there exists $0 < y_z < z$ such that for all $0 < y <  y_z$,
$$\frac{-\frac{\partial c_1}{\partial y}(y,z) + g_0'(y)}{\zeta'(y)} \geq F_0(y,z)$$
and the existence of $(\overline{y},\overline{z})$ for which $F_0(\overline{y},\overline{z}) < k_4$.  Arbitrarily fix $z \in (0,\infty)$.  Since we need $y$ sufficiently close to $0$, using $y < 1$ gives 
\begin{eqnarray*}
\frac{-\frac{\partial c_1}{\partial y}(y,z) + g_0'(y)}{\zeta'(y)} &=& \frac{\frac{k_2}{2} (y^{-3/2}-1) + \frac{2k_4}{(2\mu+\sigma^2)} y^{-1} - \frac{k_4}{\mu} + \frac{\sigma^2(k_3+k_4)}{\mu(2\mu+\sigma^2)} y^{2\mu/\sigma^2}}{\frac{2}{(2\mu+\sigma^2)} y^{-1}} \\
&=& \mbox{$\frac{k_2(2\mu+\sigma^2)}{4}$} (y^{-1/2}-y) + k_4 + \mbox{$\frac{k_4(2\mu+\sigma^2)}{2\mu}$} y + \mbox{$\frac{\sigma^2(k_3+k_4)}{2\mu}$} y^{(2\mu+\sigma^2)/\sigma^2}.
\end{eqnarray*}
Similarly
\begin{eqnarray*}
F_0(y,z) &=& \frac{k_1 + \frac{k_2}{2}(y^{-1/2}-z^{-1/2}) + \frac{k_2}{2}(z-y)}{\frac{2}{2\mu+\sigma^2}(\ln(z)-\ln(y))} \\
& & +\; \frac{g_0(z) - \frac{2k_4}{2\mu+\sigma^2} \ln(y) + \frac{k_4}{\mu} (y-1) - \frac{\sigma^4(k_3+k_4)}{\mu(2\mu+\sigma^2)^2} (y^{(2\mu+\sigma^2)/\sigma^2} - 1)}{\frac{2}{2\mu+\sigma^2}(\ln(z)-\ln(y))} \\
&=& k_4 + \frac{k_1 + \frac{k_2}{2}(y^{-1/2}-z^{-1/2}) + \frac{k_2}{2}(z-y)}{\frac{2}{2\mu+\sigma^2}(\ln(z)-\ln(y))} \\
& & +\; \frac{g_0(z) - \frac{2k_4}{2\mu+\sigma^2} \ln(z) + \frac{k_4}{\mu} (y-1) - \frac{\sigma^4(k_3+k_4)}{\mu(2\mu+\sigma^2)^2} (y^{(2\mu+\sigma^2)/\sigma^2} - 1)}{\frac{2}{2\mu+\sigma^2}(\ln(z)-\ln(y))}. 
\end{eqnarray*}

Since the leading order terms as $y \rightarrow 0$ are $y^{-1/2}$ and $\frac{y^{-1/2}}{\ln(y)}$, respectively, it follows that 
$$\lim_{y\rightarrow 0} \frac{F_0(y,z)}{\frac{-\frac{\partial c_1}{\partial y}(y,z) + g_0'(y)}{\zeta'(y)}} = 0$$
and hence there exists some $y_z > a$ so that \eqref{F-decr-at-a} holds for all $a < y < y_z$.

Finally observe that the choice $(\overline{y},\overline{z}) = (1,e)$ has 
$$F_{0}(1, e) = \mbox{$(\sigma^{2}+ 2\mu)\left[k_{1} + \frac{k_{2}}{2}(1-e^{-\frac12}) + \frac{k_{2}\mu + 2 k_{3}}{2\mu} (e-1)\right] -2k_{3}$} < k_4$$ 
by the model restriction on $k_4$.
\end{proof}

\begin{thm}
There exists an optimal $(\yzstar,\zzstar)$ ordering policy in the class ${\cal A}$ for the geometric Brownian motion model having nonlinear cost structure given by \eqref{gbm-piecewise-linear}.
\end{thm}

\begin{proof}
Existence of $(\yzstar,\zzstar) \in {\cal R}$ with $F_0(\yzstar,\zzstar) = F_0^*$ follows from \propref{gbm-pl-F-optimizer}.  To apply \thmref{G0-feasible}, we need to show that \cndref{AG0-unif-int} holds.

Defining $\tilde{G}_0 = \tilde{g}_0 - F_0^* \tilde{\zeta}$, it follows that $G_0(x) = \tilde{G}_0(x) - \tilde{G}_0(x_0)$ and so $G_0'(x) = \tilde{g}_0'(x) - F_0^* \tilde{\zeta}'(x)$.  As a result, for $x$ sufficiently large, we have 
\begin{eqnarray*}
(\sigma(x) G_0'(x))^2 &=& \left(\mbox{$\frac{\sigma k_3}{\mu} x - \frac{2(\sigma k_3+F_0^*)}{2\mu+\sigma^2}$}\right)^2, \\
1+|G_0(x)| &=& 1 + \mbox{$\frac{k_3}{\mu} (x-1) - \frac{2(k_3+F_0^*)}{2\mu+\sigma^2} \ln(x) - \tilde{g}_0(x_0)$} \mbox{ and} \\
c_0(x) &=& k_3(x-1).
\end{eqnarray*}
Thus 
$$\lim_{x\rightarrow \infty} \frac{(\sigma(x)G_0'(x))^2}{(1+|G_0(x)|)(1+c_0(x))} = \mbox{$\frac{\sigma^2}{\mu}$} \quad \mbox{and} \quad \lim_{x\rightarrow \infty} \frac{(\sigma(x)G_0'(x))^2}{(1+|G_0(x)|)^{2+\epsilon}} = 0.$$
For $x$ in a neighbourhood of the left boundary $0$, 
\begin{eqnarray*}
G_0(x) &=& \mbox{$\frac{2(k_4-F_0^*)}{2\mu+\sigma^2} \ln(x) - \frac{k_4}{\mu} (x-1) + \frac{\sigma^4(k_3+k_4)}{\mu(2\mu+\sigma^2)^2} (x^{(2\mu+\sigma^2)/\sigma^2} - 1) - \tilde{g}_0(x_0)$}, \mbox{ and} \\
(\sigma(x) G_0'(x))^2 &=& \left(-\mbox{$\frac{\sigma^2(\sigma^2k_4-2\mu F_0^*)}{\mu(2\mu+\sigma^2)} x + \frac{\sigma^4(k_3+k_4)}{\mu(2\mu+\sigma^2)} x^{(2\mu+\sigma^2)/\sigma^2}$}\right)^2. 
\end{eqnarray*}
It follows that $\lim_{x\rightarrow 0} G_0(x) = -\infty$, and $\lim_{x\rightarrow 0} (\sigma^2(x)G_0'(x))^2=0$, which then implies that \cndref{AG0-unif-int}(a) is satisfied. 
\end{proof}

\subsection{Feller's Branching Diffusion Inventory Model}\label{ex-branching}
In the absence of ordering, let the inventory dynamics be described Feller's branching diffusion given by the following  stochastic differential equation\begin{equation}
\label{eq-Feller-diffusion}
dX_0(t) = -\mbox{$\frac{1}{2}$} X_0(t)\, dt + \sqrt{X_0(t)}\, dW(t), \ X_0(0) =x_{0} \in\I:= (0,\infty).
\end{equation} Suppose without loss of generality that $x_{0} =1$ in this subsection. The scale and speed densities of \eqref{eq-Feller-diffusion} are given by $s(x) = e^{x-1}$ and $m(x) = x^{-1} e^{-x+1}$, respectively. 
Consequently it is straightforward to verify that  $0$ is an attracting point and $\infty$ is a non-attracting point. In fact, one can show that 
$0 $ is an absorbing point, see, for instance, Theorem 13.1 of \cite{Klebaner-05}; 
and $\infty$ is a natural point. Moreover, for any $y \in \I$, we have $M[y,\infty) = \int_{y}^{\infty} \frac{1}{z} e^{1-z} dz \le \frac{e}{y} \int_{y}^{\infty} e^{-z} \, dz < \infty. $This verifies Condition \ref{diff-cnd}. 

Suppose the holding and ordering  costs functions are given by 
\begin{equation}\label{cost-branching-diffusion}
c_{0}(x) = x^{\gamma_{1}} + x^{\gamma_{2}}, \ \ x\in \I, \qquad c_{1}(y,z) : = k_{1} + \widehat c(y,z), \ (y,z) \in \mathcal R,
\end{equation} in which $\gamma_{1} > 0$,  $\gamma_{2} < 0$, $k_{1}>0$, and $\hat c: \mathcal R\mapsto \R_{+}$ is a nonnegative and continuous function. 
It is immediate to show that for any $y\in \I$,  $\int_{y}^{\infty} c_{0}(v) \, dM(v)  < \infty$; establishing Condition \ref{cost-cnds}.
Moreover, Condition \ref{extra-cnd} is trivially satisfied. Therefore, by Theorem \ref{F-optimizers}, there exists a pair $(y_{0}^{*}, z_{0}^{*} ) \in \mathcal R$ such that $F_{0}(y_{0}^{*}, z_{0}^{*}) = F_{0}^{*} = \inf\{F_{0}(y,z): (y,z) \in \overline{\mathcal R} \}.$

\begin{thm}\label{thm-braching-diffusion}
If $\gamma_{2} \le -2$, then there exists an optimal $(\yzstar,\zzstar)$  ordering policy in the class $\A$ for the Branching diffusion model \eqref{eq-Feller-diffusion} having nonlinear cost structure given by \eqref{cost-branching-diffusion}.
\end{thm}
\begin{proof}
The assertion follows from Theorem \ref{G0-feasible} directly if we can verify Condition \ref{AG0-unif-int} holds.  Using the definitions of $\zeta$, $g_{0}$ and $G_{0}$ in \eqref{g0-fn} and \eqref{G0-def}, respectively, we  have 
 \begin{align*}
&G_{0}(x) = 2 \int_{1}^{x}\int_{u}^{\infty} \big(v^{\gamma_{1} -1} + v^{\gamma_{2}-1} - F_{0}^{*} v^{-1} \big)  e^{-v}e^{u } \, dv \, du,\qquad  x \in \I.\end{align*}

Let us first study the asymptotic behaviors of $ \frac{c_{0}(x)}{ (1+|G_{0} (x)|)^{2}}$ and $ \frac{(\sigma(x)G_{0}'(x))^{2}}{(1+|G_{0} (x)|) (1+ c_{0}(x))}$ when $x\to \infty$. 
Since $\lim_{x\to\infty}\int_{x}^{\infty} (v^{\gamma_{1} -1}+ v^{\gamma_{2}-1} - F_{0}^{*} v^{-1}) e^{-v} \, dv =0$, we can  use L'Hospital's Rule to  compute 
  \begin{align*}
 \lim_{x\to\infty} \frac{\int_{x}^{\infty}  (v^{\gamma_{1} -1}+ v^{\gamma_{2}-1} - F_{0}^{*} v^{-1}) e^{-v} \, dv}{x^{\gamma_{1}-1}e^{-x}}& =\lim_{x\to\infty}\frac{-(x^{\gamma_{1}-1} + x^{\gamma_{2}-1} - F_{0}^{*} x^{-1})e^{-x}}{(\gamma_{1}-1)x^{\gamma_{1}-2}e^{-x} - x^{\gamma_{1}-1}e^{-x} }=1.
\end{align*} Hence there exists some $\delta  > 1$ such that 
\begin{displaymath}
\mbox{$\frac{1}{2}$} x^{\gamma_{1}-1}e^{-x}\le \int_{x}^{\infty}  (v^{\gamma_{1} -1}+ v^{\gamma_{2}-1} - F_{0}^{*} v^{-1}) e^{-v} \, dv \le \mbox{$\frac{3}{2}$} x^{\gamma_{1}-1}e^{-x}, \text { for all }x\ge \delta.
\end{displaymath} On the other hand, the integral $ 2 \int_{1}^{\delta}\!\int_{u}^{\infty} \big(v^{\gamma_{1} -1} + v^{\gamma_{2}-1} - F_{0}^{*} v^{-1} \big)  e^{-v}e^{u } \,dv  du$ 
is uniformly bounded by a constant $K = K(\delta, \gamma_{1}, \gamma_{2}, F_{0}^{*})$.  In the rest of  the proof, we shall denote by $K$   a generic positive constant whose exact value may be different from line to line.  
 Thus it follows that for $x \ge \delta$, we have
 \begin{align*}
 |G_{0}(x)|   & \ge 2 \int_{\delta}^{x}  \int_{u}^{\infty} \big(v^{\gamma_{1} -1} + v^{\gamma_{2}-1} - F_{0}^{*} v^{-1} \big)  e^{-v}e^{u } \, dv \, du - K    \\
    &  \ge  2 \int_{\delta}^{x} \mbox{$\frac{1}{2}$} u^{\gamma_{1}-1}e^{-u} e^{u }\, du -K 
     = \frac{x^{\gamma_{1}} - \delta^{\gamma_{1}}}{\gamma_{1}} - K.
\end{align*}
Likewise, for all $x\ge \delta$, we have \begin{align*}
    G_{0}'(x)  & = 2 \int_{x}^{\infty} \big(v^{\gamma_{1} -1} + v^{\gamma_{2}-1} - F_{0}^{*} v^{-1} \big)  e^{-v}e^{x} \, dv       
      \le 2 \cdot \mbox{$\frac{3}{2}$} x^{\gamma_{1}-1}e^{-x}e^{x} =3 x^{\gamma_{1}-1}.
\end{align*}
Hence it follows that  
\begin{align}\label{eq1-cnd56}
  \frac{c_{0}(x)}{ (1+|G_{0} (x)|)^{2}}   &\le   \frac{x^{\gamma_{1}} + x^{\gamma_{2}}}{\big(1 +  |\frac{x^{\gamma_{1}} - \delta^{\gamma_{1}}}{\gamma_{1}}- K|\big)^{2}}  \to 0, && \text{ as }x \to \infty,\\
\label{eq2-cnd56}   \frac{(\sigma(x)G_{0}'(x))^{2}}{(1+|G_{0} (x)|) (1+ c_{0}(x))}   &   \le  \frac{(x^{\frac12} 3 x^{\gamma_{1}-1})^{2}}{\big( 1+  |\frac{x^{\gamma_{1}} - \delta^{\gamma_{1}}}{\gamma_{1}}- K|\big) (1+ x^{\gamma_{1} } + x^{\gamma_{2}})}  \to 0,  &&\text{ as }x \to \infty.  
   \end{align} 
   
   Next we consider the asymptotic behaviors of $\frac{c_{0}(x)}{ (1+|G_{0} (x)|)^{2}} $ and  $\frac{(\sigma(x)G_{0}'(x))^{2}}{(1+|G_{0} (x)|)^{3}}$ when $x\downarrow 0$.
When $ 0 < x \ll 1$, we can write 
\begin{equation}\label{G0-decomposition} \begin{aligned}
 G_{0}(x) =&  -2 \int_{x}^{\kappa}\int_{u}^{\kappa} \big(v^{\gamma_{1} -1} + v^{\gamma_{2}-1} - F_{0}^{*} v^{-1} \big)  e^{-v}e^{u } \, dv\, du \\ & - 2 \int_{x}^{\kappa}\int_{\kappa}^{1} \big(v^{\gamma_{1} -1} + v^{\gamma_{2}-1} - F_{0}^{*} v^{-1} \big)  e^{-v}e^{u } \, dv \, du   \\ & - 2 \int_{\kappa}^{1}\int_{u}^{1} \big(v^{\gamma_{1} -1} + v^{\gamma_{2}-1} - F_{0}^{*} v^{-1} \big)  e^{-v}e^{u } \, dv \, du \\ & +2 \int_{1}^{x}\int_{1}^{\infty} \big(v^{\gamma_{1} -1} + v^{\gamma_{2}-1} - F_{0}^{*} v^{-1} \big)  e^{-v}e^{u } \, dv\, du,\end{aligned}\end{equation} 
 where $\kappa= \kappa(F_{0}^{*},\gamma_{2})  \in (0,1)$ is chosen so that $F_{0}^{*} < \frac12 v^{\gamma_{2}}$ (and hence $v^{\gamma_{1} -1} + v^{\gamma_{2}-1} - F_{0}^{*} v^{-1} > \frac12 v^{\gamma_{2}-1}> 0$) for all $v\in (0, \kappa]$. 
It is easy to see that the second,  third, and fourth integrals of \eqref{G0-decomposition} are uniformly bounded.
Thus    we have
\begin{align}\label{|G0|-lower-bd}
\nonumber  |G_{0}(x)|   
    &   \ge 2    \int_{x}^{\kappa}\int_{u}^{\kappa} \mbox{$\frac{1}{2}$} v^{\gamma_{2}-1} e^{-v}e^{u}\, dv\, du -K \ge e^{-\kappa} e^{x} \int_{x}^{\kappa}\,\mbox{$\frac{1}{\gamma_{2}}$} (\kappa^{\gamma_{2}} - u^{\gamma_{2}}) \, du- K\\
    & = e^{-\kappa} e^{x}\, \mbox{$\frac{1}{\gamma_{2}}$} \bigg(\kappa^{\gamma_{2}} (\kappa-x) - \frac{\kappa^{\gamma_{2}+ 1}-x^{\gamma_{2}+ 1}}{\gamma_{2}+1}\bigg) - K   \ge \frac{e^{-\kappa}x^{\gamma_{2}+ 1}}{\gamma_{2}(\gamma_{2}+1)} - K.
\end{align} 
  Then it follows that $\lim_{x\downarrow 0} G_{0}(x) = -\infty$ and   for all $0 < x \ll 1$, \begin{equation}\label{eq3-cnd56}
 \frac{c_{0}(x)}{ (1+|G_{0} (x)|)^{2}}  \le  \frac{x^{\gamma_{1}} + x^{\gamma_{2}}} {\big(1+ | \frac{e^{-\kappa}x^{\gamma_{2}+ 1}}{\gamma_{2}(\gamma_{2}+1)} - K|\big)^{2}}  \le  \frac{x^{\gamma_{1}} + x^{\gamma_{2}}} {Kx^{2\gamma_{2} + 2} }  \le K.
\end{equation}    

    Next  using the $\kappa \in (0,1)$ chosen before, we write
  \begin{align*}
 \sigma(x) G_{0}'(x)   
    &   =2 x^{\frac12} e^{x} \int_{x}^{\kappa} \big(v^{\gamma_{1} -1} + v^{\gamma_{2}-1} - F_{0}^{*} v^{-1} \big)  e^{-v} \, dv   \\ & \quad  
    + 2 x^{\frac12} e^{x} \int_{\kappa}^{\infty} \big(v^{\gamma_{1} -1} + v^{\gamma_{2}-1} - F_{0}^{*} v^{-1} \big)  e^{-v} \, dv.   
\end{align*}  Observe that the second term above is uniformly bounded for all $x \in (0,1)$. 
 Since $\gamma_{2} \le -2$, $\gamma_{1} > 0$ and $F_{0}^{*} > 0$, we have 
\begin{align*}
    | \sigma(x) G_{0}'(x)  |  & \le  2 e^{1}x^{\frac12}\int_{x}^{\kappa}    2  v^{\gamma_{2}-1}  \, dv  + K= 4 e x^{\frac12} \frac{\kappa^{\gamma_{2}} - x^{\gamma_{2}}}{\gamma_{2}} + K \le  \mbox{$\frac{4 e}{-\gamma_{2}}$} x^{\gamma_{2} + \frac12} + K .    
\end{align*} This, together with \eqref{|G0|-lower-bd}, implies that for $0 < x \ll 1$,
 \begin{align}\label{eq4-cnd56}
   \frac{(\sigma(x)G_{0}'(x))^{2}}{(1+|G_{0} (x)|)^{3}} & \le \frac{\big( \frac{4 e}{-\gamma_{2}} x^{\gamma_{2} + \frac12} + K\big)^{2}}{\big( \frac{e^{-\kappa}x^{\gamma_{2}+ 1}}{\gamma_{2}(\gamma_{2}+1)} - K\big)^{3}} \le K   \frac{x^{2 \gamma_{2} + 1}}{x^{3 \gamma_{2} + 3}} = K x^{-\gamma_{2} -2} \le K. 
\end{align} Equations \eqref{eq1-cnd56}, \eqref{eq2-cnd56}, \eqref{eq3-cnd56}, and \eqref{eq4-cnd56}  establish Condition \ref{AG0-unif-int}. The proof is therefore complete. 
\end{proof}

\begin{rem}\label{rem-about-cnd56}
For $\gamma_2 \in (-2,0)$, similar computations show that $\lim_{x\rightarrow 0}\frac{c_{0}(x)}{(1+|G_{0}(x)|)^{2}} =\infty$.  Thus Condition \ref{AG0-unif-int}(a) fails.  Observe that the speed measure is very large in any neighborhood of $0$ indicating that the inventory moves slowly while $\gamma_2$ determines the penalty rate that the holding cost imposes near $0$.  It is therefore the delicate interplay between the dynamics of the model and the cost structure which determines whether or not Condition \ref{AG0-unif-int} holds.
\end{rem}

\begin{appendices}

\section{Proofs for \sectref{sect:form}} \label{appendix-A} 


\begin{proof}[Proof of \lemref{g0-at-a}]
The arguments related to the right boundary $b$ are very similar to and a bit more straightforward than those related to the left boundary $a$.  We therefore leave the verification of the limits at $b$ to the reader.

Similarly, the argument when $c_0(a) = \infty$ is essentially the same (using the definition of a limit being $\infty$) as when $c_0(a) < \infty$ so we give the details in the latter instance.  Choose $\epsilon > 0$ arbitrarily and let $y_\epsilon \in \mathcal{I}$ be such that for all $a < x < y_\epsilon$, $|c_0(x) - c_0(a)| < \epsilon$.  For $a < y < y_\epsilon$, define $\tau_y$ to be the first hitting time of $\{y\}$ by the process $X_0$ of \eqref{dyn} when $X_0(0) = y_\epsilon$. Since $a$ is a natural boundary, $\lim_{y\rightarrow a} \left(\zeta(y_\epsilon) - \zeta(y)\right) = \lim_{y\rightarrow a} \EE_{y_\epsilon}[\tau_y] = \infty$.  Using the definition of $\zeta$ in \eqref{g0-fn}, observe that  
\begin{align} \label{zeta-diff} 
\zeta(y_\epsilon) - \zeta(y) = \int_y^{y_\epsilon} \int_u^b 2\, dM(v)\, dS(u) 
= \int_y^{y_\epsilon} \int_u^{y_\epsilon} 2\, dM(v)\, dS(u) + 2 M[y_\epsilon,b) S[y,y_\epsilon].
\end{align}
By \cndref{diff-cnd}(b), $M[y_\epsilon,b) < \infty$ and since $a$ is attracting, $\lim_{y\rightarrow a} S[y,y_\epsilon] = S(a,y_\epsilon] < \infty$.  Thus the limit of the double integral as $y \rightarrow a$ is infinite.  

Similarly to \eqref{zeta-diff} using the definition of $g_0$ in \eqref{g0-fn}, we have 
$$g_0(y_\epsilon) - g_0(y) = \int_y^{y_\epsilon} \int_u^{y_\epsilon} 2c_0(v)\, dM(v)\, dS(u) + 2\left(\int_{y_\epsilon}^b c_0(v)\, dM(v)\right) S[y,y_\epsilon].$$ 
Due to the integrability condition \eqref{c0-M-integrable} on $c_0$ relative to the speed measure $M$, the second term on the right-hand side remains bounded as $y \rightarrow a$.  Therefore writing 
\begin{align*}
\frac{g_0(y_\epsilon)-g_0(y)}{\zeta(y_\epsilon)-\zeta(y)} = \frac{\int_y^{y_\epsilon} \int_u^{y_\epsilon} 2c_0(v)\, dM(v)\, dS(u)}{\zeta(y_\epsilon)-\zeta(y)} + \frac{2\left(\int_{y_\epsilon}^b c_0(v)\, dM(v)\right) S[y,y_\epsilon]}{\zeta(y_\epsilon)-\zeta(y)},
\end{align*}
the second summand converges to $0$ as $y\rightarrow a$ so the asymptotics is determined by the first summand.  Using \eqref{zeta-diff}, observe that
\begin{eqnarray*}
\frac{\int_y^{y_\epsilon} \int_u^{y_\epsilon} 2c_0(v)\, dM(v)\, dS(u)}{\zeta(y_\epsilon)-\zeta(y)} &=& \frac{\int_y^{y_\epsilon} \int_u^{y_\epsilon} 2c_0(v)\, dM(v)\, dS(u)}{\int_y^{y_\epsilon} \int_u^{y_\epsilon} 2\, dM(v)\, dS(u)} \cdot \frac{1}{1+\frac{2M[y_\epsilon,b) S[y,y_\epsilon]}{\int_y^{y_\epsilon} \int_u^{y_\epsilon} 2\, dM(v)\, dS(u)}}. 
\end{eqnarray*}
The second factor converges to $1$ as $y\rightarrow a$ so by the choice of $y_\epsilon$, we have
\begin{equation} \label{g0-zeta-ratio-at-a2}
c_0(a) - \epsilon \leq \liminf_{y\rightarrow a} \frac{g_0(y_\epsilon)-g_0(y)}{\zeta(y_\epsilon)-\zeta(y)} \leq \limsup_{y\rightarrow a} \frac{g_0(y_\epsilon)-g_0(y)}{\zeta(y_\epsilon)-\zeta(y)} \leq c_0(a) + \epsilon.
\end{equation}
Recall, $\zeta(y) \rightarrow -\infty$ as $y\rightarrow a$.  In the case $g_0(y)$ remains bounded as $y\rightarrow a$, then $\lim_{y\rightarrow a} \frac{g_0(y)}{\zeta(y)} = 0$.  Moreover, the limit exists in \eqref{g0-zeta-ratio-at-a2} and equals 0 so
$$c_0(a) - \epsilon \leq \lim_{y\rightarrow a} \frac{g_0(y_\epsilon) - g_0(y)}{\zeta(y_\epsilon) - \zeta(y)} = 0 \leq c_0(a) + \epsilon.$$ 
Since $\epsilon$ is arbitrary, it follows that $c_0(a) = 0$ and hence \eqref{g0-zeta-ratio-at-a} holds.  

When $g_0(y) \rightarrow -\infty$ as $y\rightarrow a$, we have
$$c_0(a)-\epsilon \leq \liminf_{y\rightarrow a} \left(\frac{g_0(y)}{\zeta(y)} \cdot \frac{\frac{g_0(y_\epsilon)}{g_0(y)} - 1}{\frac{\zeta(y_\epsilon)}{\zeta(y)}-1}\right) = \liminf_{y\rightarrow a} \frac{g_0(y)}{\zeta(y)}$$
and similarly 
$$\limsup_{y\rightarrow a} \frac{g_0(y)}{\zeta(y)} \leq c_0(a) + \epsilon.$$
Since $\epsilon$ is arbitrary, it again follows that $\lim_{y\rightarrow a} \frac{g_0(y)}{\zeta(y)}  = c_0(a)$, establishing \eqref{g0-zeta-ratio-at-a}.  Picking $z \in \mathcal{E}$ arbitrarily, the identity \eqref{z-fixed-g0-zeta-diff-at-a} now follows.

A similar argument addresses \eqref{double-g0-zeta-diff-at-a}.  Again using the definitions of $g_0$ and $\zeta$, we have
\begin{eqnarray*}
\frac{g_0(z) - g_0(y)}{\zeta(z) - \zeta(y)} &=& \frac{\int_y^z \int_u^{y_\epsilon} 2 c_0(v)\, dM(v)\, dS(u) + \left(\int_{y_\epsilon}^b 2 c_0(v)\, dM(v)\right) S[y,z]}{\int_y^z 2 M[u,y_\epsilon]\, dS(u) + 2 M[y_\epsilon,b) S[y,z]} \\
&=& \frac{\frac{\int_y^z \int_u^{y_\epsilon} 2 c_0(v)\, dM(v)\, dS(u)}{\int_y^z \int_u^{y_\epsilon} 2\, dM(v)\, dS(u)} + \frac{\left(\int_{y_\epsilon}^b 2 c_0(v)\, dM(v)\right) S[y,z]}{\int_y^z 2 M[u,y_\epsilon]\, dS(u)}}{1 + \frac{2 M[y_\epsilon,b) S[y,z]}{\int_y^z 2 M[u,y_\epsilon]\, dS(u)}}.
\end{eqnarray*}
Since $a$ is attracting, $S(a) < \infty$ and since it is a natural boundary, $\zeta(a) = -\infty$.  Together, these imply that $\lim_{y\rightarrow a} M[y,z] = \infty$; this holds, in particular, when $z=y_\epsilon$.  Examining the second summands in both the numerator and denominator above, set $K$ below to be $K = M[y_\epsilon,b)$ for the denominator term and $K=\int_{y_\epsilon}^b c_0(v)\, dM(v)$ for the numerator summand.  Then 
$$\frac{2K S[y,z]}{\int_y^z 2 M[u,y_\epsilon]\, dS(u)} \leq \frac{2K S[y,z]}{2 M[z,y_\epsilon] S[y,z]} $$
and the right-hand side converges to $0$ as $z\rightarrow a$.  Thus for all $y < z \leq y_\epsilon$, 
$$c_0(a) - \epsilon \leq \liminf_{(y,z)\rightarrow (a,a)} \frac{g_0(z) - g_0(y)}{\zeta(z) - \zeta(y)} \leq \limsup_{(y,z)\rightarrow (a,a)} \frac{\int_y^z \int_u^{y_\epsilon} 2 c_0(v)\, dM(v)\, dS(u)}{\int_y^z \int_u^{y_\epsilon} 2\, dM(v)\, dS(u)} \leq c_0(a) + \epsilon$$
and the result follows since $\epsilon$ is arbitrary.
\end{proof}

\begin{proof}[Proof of \thmref{F-optimizers}]
We examine the behaviour of $F_0$ at each boundary using the same order of analysis as in the proof of Proposition~3.5 of \cite{helm:15b}.
\smallskip

\noindent
{\em Case i: the diagonal $z=y$, excluding natural boundaries.}\/ By definition, $F_0(y,y) = \infty$ which is not the minimal value of $F_0$. \smallskip

\noindent
{\em Case ii: the boundary $z=b$, excluding $(a,b)$ when $a$ is natural.}\/
Suppose $b$ is an entrance boundary so $\zeta(b) < \infty$.  It follows that for each $y \in \mathcal{E}$ with $y < b$, $F_0(y,b) = \lim_{z\rightarrow b} F_0(y,z)$ exists in $\overline{\RR^+}$.  The optimization of $F_0$ therefore includes the possibility $z=b$.

Next consider when $b<\infty$ is a natural boundary with $c_0(b) < \infty$ for which \cndref{extra-cnd}(b,ii) holds.  Then for each $y \in \mathcal{E}\backslash\{b\}$, $F_0(y,\cdot)$ is increasing for $z > z_y$ and thus the infimal value of $F_0$ is not obtained in the limit as $z\rightarrow b$.

Now assume $b$ is a natural boundary for which $c_0(b) = \infty$.  Case (ii) of the proof of Proposition 3.5 of \cite{helm:15b} establishes that $F(y,b) = \infty$ for each $y \in {\cal E}\backslash\{b\}$.
\smallskip

\noindent
{\em Case iii: the vertex $(b,b)$ with $b$ being natural.}\/  Using \eqref{double-g0-zeta-diff-at-a} of \lemref{g0-at-a},
$$\liminf_{(y,z)\rightarrow (b,b)} F_0(y,z) \geq \liminf_{(y,z)\rightarrow (b,b)} \frac{g_0(z)-g_0(y)}{\zeta(z)-\zeta(y)} = c_0(b).$$
Thus the infimum does not occur in the limit as $(y,z) \rightarrow (b,b)$ since $(\wdt{y},\wdt{z})$ of \cndref{extra-cnd}(b,ii) satisfies $F_0(\wdt{y},\wdt{z}) < c_0(b)$ or $c_0(b) = \infty$ in \cndref{extra-cnd}(b,i).  
\smallskip

\noindent
{\em Case iv: the boundary $y=a$, excluding $(a,b)$ when $b$ is natural.}\/ 
As shown in case (iv) of the proof of Proposition 3.5 in \cite{helm:15b}, when $a$ is a regular or an exit boundary, $F_0(a,z)$ exists in $\overline{\RR^+}$ for each $z\in \mathcal{E}$ with $z > a$.  The optimization of $F_0$ therefore includes values with $y=a$.
This proof also shows that when $a$ is a natural boundary for which $c_0(a) = \infty$,  
\begin{equation}
\label{asymp-at-a2}
 \lim_{y\rightarrow a} F_0(y,z) \ge \lim_{y\rightarrow a}c_{0} (y) = \infty.
\end{equation} 

Proposition 3.5 of \cite{helm:15b} does not consider the case when $a>-\infty$ is a natural boundary with $c_0(a) < \infty$ so a proof is required.  Since \cndref{extra-cnd}(a,ii) holds, $F_0(\cdot,z)$ is strictly decreasing for each $z\in \mathcal{E}$ and hence the infimal value of $F_0$ does not occur in the limit as $y \rightarrow a$.
\smallskip

\noindent
{\em Case v: the vertex $(a,b)$ with both boundaries being natural.}\/ 
When $c_0(a) = \infty$ or $c_0(b)=\infty$, case (v) of the proof of Proposition 3.5 in \cite{helm:15b} establishes that 
$$\lim_{(y,z) \rightarrow (a,b)} F_0(y,z) = \infty.$$

We therefore need to consider models in which $a > -\infty$ and $b < \infty$ with $c_0(a), c_0(b) < \infty$.  Note that both \cndref{extra-cnd}(a,ii) and (b,ii) hold.

Choose $\epsilon$ such that $0 < \epsilon < (c_0(a)-F_0(\wdh{y},\wdh{z})) \wedge (c_0(b)-F_0(\wdt{y},\wdt{z}))$.  Again, \lemref{g0-at-a} establishes that there exists some $y_\epsilon, z_\epsilon \in \mathcal{I}$ such that 
$$\frac{g_0(y_\epsilon) - g_0(y)}{\zeta(y_\epsilon) - \zeta(y)} \geq c_0(a) - \epsilon \quad \forall y \in (a,y_\epsilon)\quad \mbox{and} \quad \frac{g_0(z)-g_0(z_\epsilon)}{\zeta(z)-\zeta(z_\epsilon)} \geq c_0(b) - \epsilon \quad \forall z \in (z_\epsilon,b).$$ 
Thus for each $y$ with $a < y < y_\epsilon$ and $z$ such that $z_\epsilon < z < b$, we have
\begin{eqnarray*}
F_0(y,z) &\geq& \frac{g_0(z)-g_0(y)}{\zeta(z)-\zeta(y)} \\
&=& \frac{g_0(z) - g_0(z_\epsilon) + g_0(z_\epsilon) - g_0(y_\epsilon) + g_0(y_\epsilon) - g_0(y)}{\zeta(z) - \zeta(y)} \\
&=& \frac{g_0(z) - g_0(z_\epsilon)}{\zeta(z)-\zeta(z_\epsilon)} \cdot \frac{\zeta(z)-\zeta(z_\epsilon)}{\zeta(z) - \zeta(y)} + \frac{g_0(z_\epsilon) - g_0(y_\epsilon)}{\zeta(z) - \zeta(y)} + \frac{g_0(y_\epsilon) - g_0(y)}{\zeta(y_\epsilon)-\zeta(y)}\cdot \frac{\zeta(y_\epsilon)-\zeta(y)}{\zeta(z) - \zeta(y)} \\
&\geq& [c_0(b)-\epsilon] \cdot \frac{\zeta(z)-\zeta(z_\epsilon)}{\zeta(z) - \zeta(y)} + \frac{g_0(z_\epsilon)-g_0(y_\epsilon)}{\zeta(z) - \zeta(y)} + [c_0(a)-\epsilon] \cdot \frac{\zeta(y_\epsilon)-\zeta(y)}{\zeta(z) - \zeta(y)} \\
&\geq& [(c_0(a)-\epsilon)\wedge(c_0(b)-\epsilon)]\cdot \frac{\zeta(z)-\zeta(z_\epsilon) + \zeta(y_\epsilon)-\zeta(y)}{\zeta(z) - \zeta(y)}. 
\end{eqnarray*}
Since $\lim_{(y,z)\rightarrow (a,b)} (\zeta(z) - \zeta(y)) = \infty$, it follows that
\begin{equation} \label{ab3}
\liminf_{(y,z) \rightarrow (a,b)} F_0(y,z) \geq [(c_0(a)-\epsilon)\wedge(c_0(b)-\epsilon)] > F_0(\overline{y},\overline{z}) \wedge F_0(\widetilde{y},\widetilde{z}).
\end{equation}
Therefore the infimum of $F_0$ is not obtained in the limit as $(y,z)\rightarrow (a,b)$.  
\smallskip

\noindent
{\em Case vi: the vertex $(a,a)$ with $a$ being natural.}\/  This argument is essentially the same as for the vertex $(b,b)$ using \lemref{g0-at-a} and \cndref{extra-cnd}(a,ii).

In summary, we have established that the infimal value of $F_0$ does not occur at any boundary so there exists some $(\yzstar,\zzstar) \in \mathcal{R}$ which minimizes the function $F_0$.  
\end{proof}

\section{Proofs for \sectref{sect:optim}} \label{appendix-B} 

\begin{proof}[Proof of \lemref{lem:G0-approx}]
Fix $n \in \NN$ arbitrarily.  To show that $\gn$ is bounded, notice first that on the set $\{x\in \mathcal{E}: |G_0(x)| < 1\}$, 
$$|\gn(x)| = \frac{|G_0(x)|}{1+\frac{1}{n} h(G_0(x))} < 1.$$ 
On the set $\{x \in \mathcal{E}: |G_0(x)| \geq 1\}$, we have
$$|\gn(x)| = \frac{n}{\frac{n}{|G_0(x)|}+1} \leq n.$$
Combining these estimates indicates that $|\gn(x)| \leq n$.  In addition, when $a$ and $b$ are finite and natural, $G_0(x) \rightarrow \pm \infty$ implies $\lim_{x\rightarrow a} G_n(x) = -n$ and $\lim_{x\rightarrow b} G_n(x) = n$, respectively, so we may define $G_n$ at such boundaries to be the appropriate limiting value.

Straightforward calculations establish that for $x \in {\cal E}$,
\begin{eqnarray} \label{G0-prime-expression} 
\gn'(x) &=& \frac{G_0'(x) [1 + \frac{1}{n} h(G_0(x)) - \frac{1}{n} G_0(x) h'(G_0(x))]}{(1 + \frac{1}{n} h(G_0(x)))^2}, \quad \mbox{and} \rule[-18pt]{0pt}{18pt} \\ \nonumber 
\gn''(x)&=& \frac{G_0''(x)[1 + \frac{1}{n} h(G_0(x)) - \frac{1}{n} G_0(x) h'(G_0(x))]}{(1+\frac{1}{n} h(G_0(x)))^{2}} - \frac{(G_0'(x))^{2} G_0(x) h''(G_0(x))}{n(1+\frac{1}{n} h(G_0(x)))^{2}} \rule[-18pt]{0pt}{18pt} \\ \nonumber
& & \quad -\; \frac{2 (G_0'(x))^{2} h'(G_0(x)) [1 + \frac{1}{n} h(G_0(x)) - \frac{1}{n} G_0(x) h'(G_0(x))]}{n (1+ \frac{1}{n} h(G_0(x)))^{3}}.
\end{eqnarray}
Observe that
$$(\sigma(x) \gn'(x))^2 = \frac{(\sigma(x) G_0'(x))^2}{(n + h(G_0(x)))^3} \cdot \frac{n^3[1 + \frac{1}{n} h(G_0(x)) - \frac{1}{n} G_0(x) h'(G_0(x))]^2}{1 + \frac{1}{n} h(G_0(x))}.$$
By \cndref{AG0-unif-int} (and \eqref{bi-implies-ai} when necessary), there exist $y_1 > a$, $z_1 < b$ and $L < \infty$ such that the first factor is bounded for all $x \in [y_1,z_1]^c$.  The continuity of this factor implies it is also bounded on $[y_1,z_1]$.  Since the second factor is bounded by $4n^3$, $(\sigma \gn')^2$ is bounded.

Next, we show that $A\gn$ is bounded.  Using the expressions for $\gn'$ and $\gn''$, it follows that 
\begin{eqnarray} \label{AGn-bound} \nonumber
A\gn(x) &=& \frac{AG_0(x) [1 + \frac{1}{n} h(G_0(x)) - \frac{1}{n} G_0(x) h'(G_0(x))]}{(1 + \frac{1}{n} h(G_0(x)))^2} - \frac{(\sigma(x)G_0'(x))^{2} G_0(x) h''(G_0(x))}{2n(1+\frac{1}{n} h(G_0(x)))^{2}} \\ \nonumber
& & \qquad -\; \frac{(\sigma(x)G_0'(x))^{2} h'(G_0(x)) [1 + \frac{1}{n} h(G_0(x)) - \frac{1}{n} G_0(x) h'(G_0(x))]}{n (1+ \frac{1}{n} h(G_0(x)))^{3}} \\
&=:& A\gn^{(1)}(x) + e_n^{(2)}(x) + e_n^{(3)}(x).
\end{eqnarray}
We examine these terms carefully.  Recall, $AG_0(x) = F_0^* - c_0(x)$ for $x \in {\cal I}$.  

Consider first the case in which $c_0(a)=\infty$; $c_0(b)=\infty$ is handled similarly so is omitted. Let $y_1 > a$ be as in \cndref{AG0-unif-int}(a,i).  For $x\leq y_1$ such that $|G_0(x)| \geq 1$, $h(G_0(x)) = |G_0(x)|$ and $h''(G_0(x)) = 0$ so $A\gn^{(1)}(x)$ and $e_n^{(3)}(x)$ are uniformly bounded due to \cndref{AG0-unif-int}(a,i) while $e_n^{(2)}(x)=0$.  For $x \leq y_1$ such that $|G_0(x)| < 1$, \cndref{AG0-unif-int}(a,i) implies that both $c_0(x)$ and $(\sigma(x) G_0'(x))^2$ are uniformly bounded.  Since $h(x)$ and $h''(x)$ are also uniformly bounded for $|x| < 1$, it follows that $A\gn$ is bounded on $(a,y_1)$.

Now consider the case in which $c_0(a) < \infty$; again, the case $c_0(b) < \infty$ is handled similarly.  $A\gn^{(1)}$ remains bounded on $[a,y_1)$ since continuity of $c_0$ at $a$ implies there is some neighbourhood $[a,y_2)$ of $a$ such that $\left|\frac{2(F_0^*-c_0(x))}{(1+\frac{1}{n}h(G_0(x)))^2}\right| \leq 2c_0(a) + 2F_0^*+ 1$ for $x \in [a,y_2)$.  When $y_2 < y_1$, continuity implies $A\gn^{(1)}$ remains bounded on $[y_2,y_1)$.  Essentially the same analysis as above but using \cndref{AG0-unif-int}(a,ii) handles $e_n^{(2)}$ and $e_n^{(3)}$ on $(a,y_1)$. Thus $AG_n$ is bounded in $(a,y_1)$ and this relation extends to include $x=a$.  

Similar arguments with regard to the boundary $b$ show that $A\gn$ is bounded in $(z_1,b)$, extending to include $b$ under \cndref{AG0-unif-int}(b,ii) when $c_0(b) < \infty$.  By continuity, $A\gn$ is bounded on $[y_1,z_1]$, establishing that $A\gn$ is bounded.

We now verify the boundary behaviour of \defref{class-D-def}.  
Regarding $A\gn$, we must show it is continuous at finite natural boundaries for which $c_0$ is finite.  We examine the boundary $a$; the analysis for $b$ is similar.  If $|G_0(a)| < \infty$, then \cndref{AG0-unif-int}(c) implies $AG_n(a) = \lim_{x\rightarrow a} AG_n(x)$ exists and is finite.  Now consider the case in which $|G_0(a)| = \infty$.  Then for $x$ sufficiently small (without loss of generality $x \leq y_1$), $|G_0(x)| \geq 1$ so $h(G_0(x)) = |G_0(x)|$ and $h''(G_0(x)) = 0$.  It then immediately follows that $e_n^{(2)}(x)=0$ while $A\gn^{(1)}(x)$ converges to $0$.  For $e_n^{(3)}(x)$, \cndref{AG0-unif-int}(a,ii) implies that for $x \leq y_1$, 
$$\frac{(\sigma(x)G_0'(x))^{2}}{(1+ \frac{1}{n} h(G_0(x)))^{3}} \leq \frac{(\sigma(x)G_0'(x))^{2} }{(1+ \frac{1}{n} |G_0(x)|)^{2 + \epsilon}} \cdot \frac{1}{(1+ \frac{1}{n} |G_0(x)|)^{1-\epsilon}} $$
which again converges to $0$ as $x\rightarrow a$.  Thus defining $AG_n(a) = 0$ makes $AG_n$ continuous at $a$.

Now, consider the case of $a$ being a reflecting boundary.  Using the definitions of $g_0$, $\zeta$ and $G_0$ in \eqref{g0-fn}  
and \eqref{G0-def}, respectively, it follows that
$$G_0'(x) = s(x) \int_x^b 2(c_0(v) - F_0^*)\, dM(v).$$
\cndref{AG0-unif-int}(c,ii) then implies that $|G_0'(a)| < \infty$.  From the expression for $\gn'(x)$ in \eqref{G0-prime-expression}, $\lim_{x\rightarrow a} \gn'(x)$ exists and is finite.

Finally, when $a$ is a sticky boundary and $c_0(a) < \infty$, \cndref{AG0-unif-int}(c,i) along with \eqref{G0-prime-expression} establishes that \defref{class-D-def}(b,iii) holds.
\end{proof}

\begin{proof}[Proof of \lemref{AGn-BGn-conv}]
We consider the convergence of $G_n$ first.  For all $x \in {\cal I}$, $|G_0(x)| < \infty$, so 
$$\lim_{n\rightarrow \infty} G_n(x) = \lim_{n\rightarrow \infty} \frac{G_0(x)}{1+\frac{1}{n} h(G_0(x))} = G_0(x).$$
This result holds also at the boundaries whenever $G_0$ is bounded there.  

Now consider the case in which $|G_0(a)| = \lim_{x\rightarrow a} |G_0(x)| = \infty$.  Observe that for $x \in {\cal I}$,
\begin{equation} \label{G0-integral-form}
G_0(x) = \int_{x_0}^x \int_u^b 2[c_0(v)-F_0^*]\, dM(v)\, dS(u).
\end{equation} 
When $a$ is attainable, $\zeta(a) > -\infty$ so if $c_0$ were bounded in a neighbourhood of $a$, $G_0$ would also be bounded in the neighbourhood.  Thus $c_0(x) \rightarrow \infty$ as $x \rightarrow a$ and as a result, $G_0(a) = \lim_{x\rightarrow a} G_0(x) = -\infty$.  When $a$ is a natural boundary, $\zeta(a) = -\infty$ and \cndref{extra-cnd} implies $c_0(a) > F_0^*$ so again $G_0(a) = -\infty$.  Then regardless of the type of boundary, for $x$ sufficiently close to $a$, $G_0(x) \leq -1$.  Since $h(x) = |x|$ on the set where $|x| \geq 1$, 
$$\lim_{x\rightarrow a} G_n(x) = \lim_{x\rightarrow a} \frac{G_0(x)}{1+\frac{1}{n}|G_0(x)|} = -n =: G_n(a)$$
so again $\lim_{n\rightarrow \infty} G_n(a) = -\infty = G_0(a).$

A similar argument at the boundary $b$ establishes that $\lim_{n\rightarrow \infty} G_n(b) = G_0(b)$ and therefore $G_n$ converges pointwise to $G_0$.  It is therefore immediate that $BG_n$ converges pointwise to $BG_0$ on $\overline{\cal R}$.

Turning to $AG_n$, recall from \eqref{AGn-bound} that $AG_n = AG_n^{(1)} + e_n^{(2)} + e_n^{(3)}$.  
Since $G_0 \in C^2({\cal I})$, $G_0(x)$ and $G_0'(x)$ are finite for $x \in {\cal I}$.  It then follows that $\lim_{n\rightarrow \infty} AG_n(x) = AG_0(x)$ for $x \in {\cal I}$.  A careful examination of the convergence at the boundaries is required.

Assume $c_0$ is finite at the boundaries.  When $|G_0(a)| < \infty$, \cndref{AG0-unif-int}(c) implies $\lim_{x\rightarrow a} \sigma(x) G_0'(x) = K_1$ for some finite $K_1$.  Denote this limit by $\sigma(a) G_0'(a)$ and set $AG_0(a) = F_0^* - c_0(a)$.  Then 
$$\lim_{x\rightarrow a} AG_n(x) = AG_n^{(1)}(a) + e_n^{(2)}(a) + e_n^{(3)}(a) =: AG_n(a).$$
As before, it follows that $\lim_{n\rightarrow \infty} AG_n(a) = AG_0(a)$.  A similar analysis applies at the boundary $b$ when $|G_0(b)| < \infty$.  Observe that, under the assumption that $c_0$ is finite at the boundaries, $G_0$ is bounded at the boundaries when $a$ is attainable and when $b$ is an entrance boundary.

Now consider the case in which $c_0(a) < \infty$ and $|G_0(a)| = \infty$.  Thus $a$ is a natural boundary and hence \cndref{extra-cnd} implies that $G_0(a) = -\infty$.  Then for some $y_1 > a$, $G_0(x) \leq -1$ for all $a < x < y_1$ and as a result, $h(G_0(x)) = |G_0(x)|$ and $h''(G_0(x)) = 0$.  Examining the expression for $AG_n$, we see 
$$AG_n(x) = \frac{AG_0(x)}{(1 + \frac{1}{n} |G_0(x)|)^2} - \frac{(\sigma(x)G_0'(x))^{2}\mbox{ sgn}(G_0(x))}{n (1+ \frac{1}{n} |G_0(x)|)^{3}}.$$
Since $AG_0(x) = F_0^* - c_0(x)$, $c_0(a)<\infty$ and $|G_0(a)| = \infty$, the first term converges to $0$ as $x\rightarrow a$.  Using \cndref{AG0-unif-int}(a,ii), we have
$$\left|\frac{(\sigma(x)G_0'(x))^{2} \mbox{ sgn}(G_0(x))}{n (1+ \frac{1}{n} |G_0(x)|)^{3}}\right| \leq \frac{Ln^2}{(1+|G_0(x)|)^{1-\epsilon}}$$
so this term also converges to $0$ as $x\rightarrow a$.  Therefore $AG_n(a) = 0$ and it follows that $\lim_{n\rightarrow \infty} AG_n(a) = 0 > F_0^* - c_0(a) = AG_0(a)$.  A similar argument using \cndref{AG0-unif-int}(b,ii) establishes that $AG_n(b) \geq AG_0(b)$ for each $n$ when $|G_0(b)| = \infty$.
\end{proof}

\section{A Counter-intuitive Example} \label{appendix-C}
This appendix presents an example of an inventory model and a particular ordering policy for which the long-term average cost is finite, and hence $\{\mu_{0,t_j}\}$ is tight for any $\{t_j\}$ with $t_j\rightarrow \infty$, but the corresponding sequence $\{\mu_{1,t_j}\}$ is not tight.

Consider the classical drifted Brownian motion inventory model of \eqref{dbm-dyn} in  Section \ref{sect:examples}:
\begin{equation} \label{dbm-dyn-appendix}
X_0(t) = W(t) - t, \qquad t\geq 0,
\end{equation} where for notational simplicity,    the initial inventory level is $x_0=0$ and the drift and diffusion coefficients are $\mu=-1$ and $\sigma=1$.  The cost functions of \eqref{dbm-cost} are specified  as 
$$c_0(x) = 2|x|, \quad \forall x \in \RR, \qquad \mbox{and}\qquad c_1(y,z) = k_1 + (z-y), \quad \forall (y,z) \in \overline{\cal R}.$$

A special ordering policy $(\tau,Y)$ will now be described. It runs in cycles, each of which is composed of two phases.  For cycle $i=1,2,3,\ldots$, Phase 1 consists of using the $(0,1)$-ordering policy a total of $2^{i-1}$ times; the length of each sub-cycle is a random variable having mean $1$.  Phase 2 involves a single $(0,2^{(i-1)/2})$-ordering policy followed immediately by using the $(2^{(i-1)/2},2^{(i-1)/2})$-ordering policy $2^{i-1}$ times.

The formal description of the ordering policy is now given. 

\begin{defn}[{\em The Policy}\/ $(\tau,Y)$] \label{stoch-tauY-def}
For cycle $i=1$, define 
$$\left\{\begin{array}{rcl}
\tau_{1,1} &=& 0, \\ Y_{1,1} &=& 1, \end{array} \right. \qquad \mbox{and} \qquad 
\left\{\begin{array}{rcl}
\tau_{1,2} &=& \inf\{t\geq \tau_{1,1}: X(t) = 0\}, \\ Y_{1,2} &=& 1, \end{array} \right. \qquad 
\left\{\begin{array}{rcl}
\tau_{1,3} &=& \tau_{1,2}, \\ Y_{1,3} &=& 0, \end{array} \right.$$
and for cycle $i = 2,3,4,\ldots$, define the orders in Phase 1 to be 
\begin{align*}
\begin{cases} 
\tau_{i,1} = \inf\{t\geq \tau_{i-1,2^i+1}: X(t) = 0\}, \\ Y_{i,1} = 1, \end{cases}     
\!\!  \begin{cases} 
\tau_{i,j}  =  \inf\{t\geq \tau_{i,j-1}: X(t) = 0\}, \\ Y_{i,j}  =  1, \end{cases}  &   j=2, \ldots, {2^{i-1}}, 
\end{align*}
and the orders in Phase 2 by 
\begin{align*}
\begin{cases} 
\tau_{i,2^{i-1}+1} = \inf\{t\geq \tau_{i,2^{i-1}}: X(t) = 0\}, \\ Y_{i,2^{i-1}+1} = 2^{(i-1)/2}, \end{cases}  
\begin{cases} 
\tau_{i,j}   =  \tau_{i,2^{i-1}+1},  \\
Y_{i,j} = 0, \end{cases}  & j = 2^{i-1}+2, \ldots, 2^i+1.
\end{align*}
\end{defn}

\begin{rem} \label{no-0-size-order}
When formulating the ordering costs, traditionally no distinction is made between not ordering and ordering nothing, with no cost incurred in either case.  In contrast, the formulation in this paper has orders of size $0$ incur the fixed cost $k_1$.  The policy $(\tau,Y)$ uses $0$-size orders to create non-trivial masses in the average ordering measure $\mu_{1,t}$ at arbitrarily large values on the diagonal $z=y$ without affecting the length of Phase 2, provided $t$ is sufficiently large.  Under the traditional formulation, it is possible to place many orders of suitably small sizes, resulting in similar masses in neighbourhoods near the diagonal at arbitrary distances from the origin (with $t$ large), such that the length of Phase 2 is barely increased.  The analysis is essentially the same as in this manuscript but requires more careful bookkeeping without affecting the limiting results.  The costly $0$-size orders considerably simplify the computations.
\end{rem}

The main result of this appendix can now be stated. 

\begin{thm}
For the drifted Brownian motion inventory model, let $(\tau,Y)$ be the ordering policy of \defref{stoch-tauY-def}, $X$ be the resulting inventory process, and $\{\mu_{0,t}\}$ and $\{\mu_{1,t}\}$, respectively, be the corresponding average expected occupation and ordering measures defined in \eqref{mus-t-def}.  Then 
(a) $J_0(\tau,Y) < \infty$; (b) $\{\mu_{0,t}: t > 0\}$ is tight as $t\rightarrow \infty$; and (c) $\{\mu_{1,t}: t > 0\}$ is not tight.
\end{thm}

\begin{proof}
This theorem is proven in pieces.  \propref{lta-holding-costs} shows that the long-term average holding costs are bounded and thus an argument as in the proof of \propref{mu0-tightness} establishes the tightness of the average expected occupation measures $\{\mu_{0,t}\}$ as $t\rightarrow \infty$.  Finally, \propref{ordering-measure-results} shows both that the long-term average ordering costs are finite and that the average expected ordering measures $\{\mu_{1,t}\}$ are not tight.
\end{proof}

Our analysis depends on a careful construction of the inventory process $X$ under the ordering policy $(\tau,Y)$ of \defref{stoch-tauY-def}. Independent copies of the diffusion $X_0$ of \eqref{dbm-dyn} are pieced together at the jump times (see, e.g., the appendix of \cite{chri:14} for such a construction) with the implication that the various ordering sub-cycles are independent.  

The initial analysis examines the long-term average cost of the $(\tau,Y)$ policy.  It begins by focusing on the holding costs.  Notice that the only orders which affect the length of cycle $i$ are the $2^{i-1}$ times that the $(0,1)$ policy is used and the one time that the $(0,2^{(i-1)/2})$ policy occurs; the $2^{i-1}$ times that orders of size $0$ are placed do not change the state of the inventory or lengthen the sub-cycles so have no affect on the holding costs.  Thus it is sufficient to restrict the analysis solely to the non-zero orders.

Let ${\cal S}:=\{\sigma_1, \sigma_2,\sigma_3, \ldots\}$ denote the times of the non-zero orders.  Let $i\in \NN$ denote the cycle and $j\in \{1,2,\ldots,2^{i-1}+1\}$ be the number of the non-zero order within cycle $i$.  Then observe that $\sigma_n = \tau_{i,j}$ where $n=2^{i-1}+i+j-2$.  The ensuing computations are simplified by a shift in the index for $j$.  For $i\geq 2$, define $\tau_{i,0} = \tau_{i-1,2^{i-1}+1}$ so that order number ``zero'' of the $i^{\mbox{\footnotesize th}}$ cycle is the last non-zero order, in fact the large order, of cycle $(i-1)$.  

Our first result gives a strong law of large numbers result for the cycle lengths.

\begin{prop} \label{slln}
Let $(\tau,Y)$ be the ordering policy of \defref{stoch-tauY-def} and ${\cal S}$ be the times of the non-zero orders.  Then
\begin{equation} \label{slln-sigma-n}
\lim_{n\rightarrow \infty} \frac{\sigma_n}{n} = 1 \; (a.s.).
\end{equation}
\end{prop} 

\begin{proof}
Define the independent (but not identically distributed) random variables 
$$\beta_0 = \sigma_1 = 0, \quad \mbox{and} \quad \beta_k = \sigma_{k+1} - \sigma_k, \qquad k\in \NN;$$  
thus $\beta_k$ gives the random length of time of the $k^{\mbox{\footnotesize th}}$ inter-order interval.  Then,  
$\sigma_n = \sum_{k=0}^{n-1} \beta_k, $  $  n \in \NN.$

Notice that, apart from $\beta_0=0$, $\beta_n$ is either the length of a $(0,1)$ sub-cycle or a $(0,2^{(i-1)/2})$ sub-cycle.  More precisely, for each $i\ge 1$, $\beta_{2^{i}+i-1} $ is the length of the cycle arising from the $(0,2^{(i-1)/2})$ sub-cycle and for $n\neq 2^i-1+i$, $\beta_n$ is the length of a $(0,1)$ sub-cycle.  Using the Laplace transform of the hitting time of a drifted Brownian motion process (see Formula 2.0.1 (p.\;295) of \cite{boro:02}), one can determine that 
$$\begin{array}{lclclcll} 
\EE[\beta_{2^{i-1}+i+j-2}]&=& 1, & \mbox{and} & \mbox{Var}(\beta_{2^{i-1}+i+j-2})&=&1, & \; j=1,2,\ldots, 2^{i-1},\/ i\ge 1, \\
\EE[\beta_{2^{i}+i-1}]&=& 2^{(i-1)/2}, & \mbox{and} & \mbox{Var}(\beta_{2^{i}+i-1})&=& 2^{(i-1)/2}, & \;  i\geq 1.
\end{array}$$
Next observe 
\begin{align*}
\sum_{n=2}^\infty \frac{\mbox{Var}(\beta_n)}{n^2} &= 
\sum_{i=2}^\infty \sum_{j=0}^{2^{i-1}}\frac{\mbox{Var}(\beta_{2^{i-1}+i+j-2})}{(2^{i-1}+i+j-2)^2} \\
& \leq \sum_{i=1}^\infty \sum_{j=1}^{2^{i-1}} \frac{1}{(2^{i-1}+i+j-2)^2} + \sum_{i=1}^\infty \frac{2^{(i-1)/2}}{(2^{i-1}+i-1)^2} \\
&\leq \sum_{n=1}^\infty \frac{1}{n^2} + \sum_{i=1}^\infty \frac{1}{2^{3(i-1)/2}} < \infty.
\end{align*} 
By Kolmogorov's Strong Law of Large Numbers (cf.\ Theorem 2 (p.\;389) of \cite{shir:96}), it follows that $$\lim_{n\to\infty} \biggl(\frac{1}{n} \sum_{k=1}^n \beta_k - \overline{\mu}_n \biggr) =0,  \quad a.s.,$$  in which $\overline{\mu}_n = \frac{1}{n} \sum_{k=1}^n \EE[\beta_k]$.  Note that $\sigma_{n+1} =  \sum_{k=0}^n \beta_k  = \sum_{k=1}^n \beta_k $. Thus the above equation can be rewritten as  $\lim_{n\to \infty}  ( \frac{\sigma_{n+1}}{n } - \overline{\mu}_{n}  )  =0 $ a.s., which in turn implies  
  \eqref{slln-sigma-n}  if  we can show that $\overline{\mu}_n \rightarrow 1$ as $n\rightarrow \infty$.

We now analyze the convergence of $\overline{\mu}_n$.  Again for $n\geq 2$, write $n=2^{i-1}+i+j-2$ with $i \ge 2$ and $0\leq j \leq 2^{i-1}$.  Notice that for $1 \leq k \leq i-1$, cycle $k$ contains $2^{k-1}$ sub-cycles generated by $(0,1)$ ordering policies having a mean length of $1$ and a single $(0,2^{(k-1)/2})$ sub-cycle with mean length $2^{(k-1)/2}$ while the partial cycle $i$ has $j$ sub-cycles from $(0,1)$, 
Thus for $n \geq 2$, 
\begin{align*}
\mbox{$\frac{1}{n}$} \sum_{k=1}^n \EE[\beta_k] &
= \frac{1}{2^{i-1}+i+j-2} \left(\sum_{k=1}^{i-1} (2^{k-1} + 2^{(k-1)/2}) + j\right) \\
&= \frac{1}{2^{i-1}+i+j-2} \left(2^{i-1} - 1 +\frac{2^{(i-1)/2}-1}{2^{1/2}-1} + j\right)  \\
& = \frac{1 -\frac{1}{2^{i-1}+j} + \frac{1}{2^{1/2}-1} \cdot\frac{2^{(i-1)/2}-1}{2^{i-1}+j} }{1 + \frac{i-2}{2^{i-1}+j}}.
\end{align*} 
Obviously we have $$\lim_{i\to\infty}\frac{1}{2^{i-1}+j}=\lim_{i\to\infty}\frac{2^{(i-1)/2}-1}{2^{i-1}+j} =\lim_{i\to\infty} \frac{i-2}{2^{i-1}+j} =0 \text{ for each }j =0, 1, \dots, 2^{i-1}.$$ 
Therefore it follows that as $n\to\infty$ (and hence   $i\rightarrow \infty$),  $\frac{1}{n} \sum_{k=1}^n \EE[\beta_k] $ converges to 1.
\end{proof}

For this model, we now establish a variant of the elementary renewal theorem; the fact that the cycles are not identically distributed means that the theorem cannot simply be applied.

\begin{prop} \label{elementary-renewal}
For $t \geq 0$, let $N(t) = \sum_{i=1}^\infty I_{\{\sigma_i \leq t\}} = \max\{n: \sigma_n \leq t\}$ be  the number of orders of positive size by time $t$.  Then 
\begin{equation}
\label{eq-Nt/t conv}
\lim_{t\rightarrow \infty} \frac{N(t)}{t} = 1,\quad  (\text{a.s. and in } L^{1}).
\end{equation}
\end{prop}

\begin{proof}
First by the definition of $N(t)$, $\sigma_{N(t)} \leq t < \sigma_{N(t)+1}$ for each $t \geq 0$.  Thus for each $t \geq 0$,
$$\frac{\sigma_{N(t)}}{N(t)} \leq \frac{t}{N(t)} < \frac{\sigma_{N(t)+1}}{N(t)} = \frac{N(t)+1}{N(t)} \cdot \frac{\sigma_{N(t)+1}}{N(t)+1}.$$
As $t\rightarrow \infty$, \propref{slln} implies first that $N(t) \rightarrow \infty\; (a.s.)$ and then establishes the a.s. convergence of \eqref{eq-Nt/t conv}. 
In order to prove 
the $L^{1}$ convergence of \eqref{eq-Nt/t conv},  it is necessary to show that the collection $\{\frac{N(t)}{t}: t\geq 1\}$ is uniformly integrable.  By Lemma~3 (p.\ 190) of \cite{shir:96}, it suffices to show that $\sup_{t\geq 1} \EE[(\frac{N(t)}{t})^2] < \infty$.  To this end, define the ordering policy $(\wdt\tau,\wdt Y)$ which always uses the $(0,1)$ policy and denote by  $\{\wdt{\sigma}_n: n\in\NN\}$  the ordering times.  Define the corresponding renewal process $\wdt N$ by  
$\wdt N(t) = \max\{n: \wdt\sigma_n \leq t\} = \sum_{i=1}^\infty I_{\{\wdt\sigma_i \leq t\}}.$
It then follows that $N(t) \leq \wdt N(t)$ for each $t \geq 0$ and hence
$\EE [ (\frac{N(t)}{t} )^2 ] \leq \EE [ (\frac{\wdt N(t)}{t} )^2 ].$
Using a standard renewal argument (see, e.g., the proof of Theorem~5.5.2 (pp.\;143,144) of \cite{chun:01}), it follows that $ \EE [ (\frac{\wdt N(t)}{t} )^2 ]$ 
is uniformly bounded for $t \geq 1$, establishing the uniform integrability of $\{\frac{N(t)}{t}: t \geq 1\}$ and hence   the $L^{1}$ convergence. 
\end{proof}

The next step on the way to showing $J_0(\tau,Y) < \infty$ is to analyze the holding costs over a single cycle.  Observe that $X(\sigma_i)$ is either $1$ or $2^{(k-1)/2}$ for some $k$; to simplify notation, let $z$ represent either value.  Next, $X(t) = z - (t-\sigma_i) + W(t-\sigma_i)$ for $t \in [\sigma_i,\sigma_{i+1})$ since no orders are placed on the interval $(\sigma_i,\sigma_{i+1})$ and, by the definition of $\sigma_{i+1}$, $z - (\sigma_{i+1}-\sigma_i) + W(\sigma_{i+1}-\sigma_i) = 0$.  Again to simplify notation, make the change of time $s=t-\sigma_i$ and define $\tau=\sigma_{i+1}-\sigma_i$.  Thus, $X$ satisfies $X(s) = z - s + W(s)$ for $s\in [0,\tau)$.  


Define 
\begin{equation} \label{cycle-holding-cost}
\Theta_\tau = \int_0^\tau c_0(X(s))\, ds.
\end{equation}
Then by Proposition~2.6 of \cite{helm:15b}, it follows that $\EE[\Theta_\tau] = z^2 + z$. 

We now establish a result similar to the law of large numbers for the holding costs.  

\begin{prop} \label{cycle-cost-asymptotics}
Let $(\tau,Y)$ be given by \defref{stoch-tauY-def} 
and $X$   the resulting inventory process. 
 Then
\begin{equation}
\label{eq-limsup-holding-cost}
 \limsup_{n\rightarrow \infty} \mbox{$\frac{1}{n}$} \sum_{k=1}^{n} \EE\left[\int_{\sigma_{k}}^{\sigma_{k+1}} c_0(X(s))\, ds\right] \le  3.
\end{equation}
\end{prop}

\begin{proof} For simplicity of notation, denote $\Theta_{k} : = \int_{\sigma_{k}}^{\sigma_{k+1}} c_0(X(s))\, ds $ for $k =1, 2, \dots$
As in the proof of \propref{slln}, write the index $n$ as $n=2^{i-1}+i+j-2$ in which $i=2,3,\ldots$ and $j=0,1,\ldots,2^{i-1}$.  Recall, the large orders have indices with $j=0$ and hence $n = 2^{i-1} + i -2 $ for $i \ge 2$ so the formula for $\EE[\Theta_\tau]$ above establishes that 
\begin{align*}
\EE[\Theta_{2^{i-1}+i+j-2}] = 2, & \quad j=1,\ldots,2^{i-1},\;  \ \text{ and } \
\EE[\Theta_{2^{i}+i-1}] = 2^{i-1} + 2^{(i-1)/2}, & \quad  i\in \NN.
\end{align*}
Consequently, for $n\geq 2$, 
\begin{align*}
 \mbox{$\frac1n$} \sum_{\ell=1}^{n} \EE[\Theta_{\ell}] 
 &  =  \frac{1}{2^{i-1}+i+j-2}   \left(\sum_{k=1}^{i-1} (2^{k-1}\cdot 2 + 2^{k-1} + 2^{(k-1)/2}) +j \cdot 2\right)  \\
 & =  \frac{1}{2^{i-1}+i+j-2} \left(3(2^{i-1} -1) +  \frac{2^{(i-1)/2}-1}{2^{1/2} -1} + 2 j \right) \\
 & \le   \frac{3 ( 2^{i-1}  +j ) + \frac{2^{(i-1)/2}-1}{2^{1/2} -1}-3}{2^{i-1}+i+j-2} 
  = \frac{3 + \frac{1}{2^{1/2}-1} \cdot\frac{2^{(i-1)/2}-1}{2^{i-1}+ j} - \frac{3}{2^{i-1}+ j}} {1 + \frac{i-2}{ 2^{i-1} + j}}.
\end{align*} Using similar computations as those in the end of   the proof of Proposition \ref{slln}, we see immediately that this ratio converges to 3 as $n\to \infty $ and hence $i \to \infty$. This gives \eqref{eq-limsup-holding-cost} as desired. 
\end{proof}

We now parlay the asymptotics relative to cycles to verify that the long-term average holding costs related to $(\tau,Y)$ are finite.
\begin{prop} \label{lta-holding-costs}
Let $(\tau,Y)$ be given by \defref{stoch-tauY-def} 
and $X$   the resulting inventory process. 
 Then
$$\limsup_{t\rightarrow \infty} \mbox{$\frac{1}{t}$} \EE\left[\int_0^t c_0(X(s))\, ds\right] \leq 3.$$
\end{prop}

\begin{proof}
Again, denote the non-zero ordering times by ${\cal S}$ and observe that the orders of size $0$ do not affect the cycle lengths.  Again, for each $t \geq 0$, recall $N(t) = \max\{n: \sigma_n \leq t\}$ so that $\sigma_{N(t)} \leq t < \sigma_{N(t)+1}$.  Thus for positive $t$,
\begin{eqnarray*}
\mbox{$\frac{1}{t}$} \int_0^t c_0(X(s))\, ds &\leq& \mbox{$\frac{1}{t}$} \int_0^{\sigma_{N(t)+1}} c_0(X(s))\, ds  
 \leq  \mbox{$\frac{1}{t}$} \int_0^{\sigma_{N(t)+2}} c_0(X(s))\, ds =  \mbox{$\frac{1}{t}$} \sum_{j=1}^{N(t)+1} \Theta_{j}, 
\end{eqnarray*}
where as in the proof of \propref{cycle-cost-asymptotics}, we employed the notation $\Theta_j = \int_{\sigma_j}^{\sigma_{j+1}} c_0(X(s))\, ds$ for each $j \in \NN$.  Noting that $\{N(t)=0\} = \emptyset$ for each $t$, consider the expectation of the right-hand term above:
\begin{align*} \label{approx-wald} 
\EE\Bigg[\sum_{j=1}^{N(t)+1}\! \Theta_j\Bigg] &= \EE\Bigg[\sum_{k=1}^\infty I_{\{N(t)=k\}}\! \sum_{j=1}^{k+1} \Theta_j\Bigg]  
 =  \EE\Bigg[\sum_{j=1}^\infty \Theta_j\sum_{k=j-1}^\infty I_{\{N(t)=k\}}\Bigg] = \sum_{j=1}^\infty \EE\big[\Theta_j I_{\{N(t) \geq j-1\}}\big].
\end{align*}
Observe that $\{N(t) \geq j-1\} = \{N(t)<j-1\}^c = \{\sigma_{j-1} > t\}^c$ is independent of the process over the interval $[\sigma_j,\sigma_{j+1})$.  Therefore
\begin{align*}
\sum_{j=1}^\infty \EE\left[\Theta_j I_{\{N(t) \geq j-1\}}\right] &= \sum_{j=1}^\infty \EE[\Theta_j] \EE[I_{\{N(t) \geq j-1\}}] 
= \sum_{j=1}^\infty \EE[\Theta_j] \left(\sum_{k=j-1}^\infty \PP(N(t)=k)\right) \\
&=\sum_{k=1}^\infty \sum_{j=1}^{k+1} \EE[\Theta_j] \PP(N(t)=k).
\end{align*}
By \propref{cycle-cost-asymptotics}, $\limsup_{k\rightarrow \infty} \frac{1}{k}\sum_{j=1}^{k} \EE[\Theta_j] \le 3$ so for any $\epsilon > 0$ there exists some $K_0 < \infty$ such that $\frac{1}{k+1} \sum_{j=1}^{k+1} \EE[\Theta_j] <   3+ \epsilon$ for all $k \geq K_0$.  Thus
 \begin{align*}
&\sum_{k=1}^\infty \sum_{j=1}^{k+1} \EE[\Theta_j] \PP(N(t)=k)  \\
 &\ \ =  \sum_{k=1}^{K_0-1}   \sum_{j=1}^{k+1} \EE[\Theta_j] \PP(N(t)=k) +  \sum_{k=K_0}^\infty \biggl(\mbox{$\frac{1}{k+1}$} \sum_{j=1}^{k+1} \EE[\Theta_j]\biggr) (k+1) \PP(N(t)=k) \\
&\ \ \leq \sum_{k=1}^{K_0-1}  \sum_{j=1}^{k+1} \EE[\Theta_j]  \PP(N(t)=k)  +\sum_{k=K_0}^\infty (3 + \epsilon) (k+1) \PP(N(t)=k) \\
&\ \ \leq \sum_{k=1}^{K_0-1} \sum_{j=1}^{k+1} \EE[\Theta_j] + (3 + \epsilon)\EE[N(t)] +3 + \epsilon.
\end{align*}Since the first and last summands are constant, combining these upper bounds, dividing by $t$ and using \propref{elementary-renewal} yields
$$\limsup_{t\rightarrow \infty} \mbox{$\frac{1}{t}$}\EE\left[\int_0^t c_0(X(s))\,ds\right] \leq \lim_{t\rightarrow \infty} \frac{(3 + \epsilon)\EE[N(t)]}{t} =3 + \epsilon.$$
The result now follows since $\epsilon> 0$ is arbitrary.\end{proof}

The final task is to verify that the long-term average ordering costs are finite and that $\{\mu_{1,t}\}$ is not tight as $t\rightarrow \infty$.  The next proposition addresses both of these concerns since the analysis is very similar.

\begin{prop} \label{ordering-measure-results}
Let $(\tau,Y)$ be given by \defref{stoch-tauY-def} for the drifted Brownian motion inventory model.  For $t > 0$, define $\mu_{1,t}$ by \eqref{mus-t-def}.  Then 
\begin{equation}
\label{eq-limsup-order-cost}
 \limsup_{t\rightarrow \infty} \int c_1(y,z)\, \mu_{1,t}(dy\times dz) \le 3k_1+2 
\end{equation}
and $\{\mu_{1,t}: t > 1\}$ is not tight. 
\end{prop}

\begin{proof}
We first address the lack of tightness for $\{\mu_{1,t}\}$.  Let $\Gamma\subset \overline{\cal R}$ be any compact set.  Then there exists some $N_0$ such that for all $i \geq N_0$, $(2^{(i-1)/2},2^{(i-1)/2}) \in \Gamma^c$.  

Again, denote the times of non-zero orders by ${\cal S} = \{\sigma_n: n\in \NN\}$ and for $n\geq 2$, write $n=2^{i-1}+i+j-2$ with $i \ge 2$ and $j=0,\ldots,2^{i-1}$.  Recall, the ``large'' order of size $2^{(i-1)/2}$ occurs at time $\sigma_{2^i+i-1}$.  Under policy $(\tau,Y)$, there are a further $2^{i-1}$ orders of size $0$ at time $\sigma_{2^i+i-1}$; denote this common time of ordering by $\wdt\sigma_{2^i+i-1,j}$ for $j=1,\ldots,2^{i-1}$ for each for the $0$-size orders.  Thus
\begin{align*}
\mu_{1,t}(\Gamma^c) &= \mbox{$\frac{1}{t}$} \EE\Bigg[\sum_{n=1}^\infty I_{\{\sigma_n \leq t\}} I_{\Gamma^c}(X(\sigma_n-),X(\sigma_n))  \\
& \qquad +\; \sum_{i=1}^\infty \sum_{j=1}^{2^{i-1}} I_{\{\wdt\sigma_{2^i+i-1,j}\leq t\}} I_{\Gamma^c}(X(\wdt\sigma_{2^i+i-1,j}-),X(\wdt\sigma_{2^i+i-1,j})) \Biggr] \\
&\geq \mbox{$\frac{1}{t}$} \EE\Bigg[\sum_{i=N_0}^\infty \sum_{j=1}^{2^{i-1}} I_{\{\wdt\sigma_{2^i+i-1,j}\leq t\}} I_{\Gamma^c}(2^{(i-1)/2},2^{(i-1)/2}) \Bigg] 
= \mbox{$\frac{1}{t}$} \EE\Bigg[\sum_{i=N_0}^\infty 2^{i-1} I_{\{\sigma_{2^i+i-1}\leq t\}} \Bigg].
\end{align*}
For each $t \geq 0$, define the processes $I$ and $J$ such that $N(t) = 2^{I(t)-1}+I(t)+J(t)-2$ in which $I(t)$ denotes the cycle in which order $N(t)$ occurs and $0 \leq J(t) \leq 2^{I(t)-1}$.  Since $N(t)$ is the number of non-zero orders placed by time $t\geq 0$, order number $N(t)$ is the $J(t)^{\mbox{\footnotesize th}}$ order within cycle $I(t)$; again $J(t)=0$ corresponds to the large order of the previous cycle.  Using the processes $I$ and $J$, it follows that for $0 \leq J(t) \leq 2^{I(t)-1}$,
\begin{align} \label{mu1-lb} \nonumber 
\mu_{1,t}(\Gamma^c) &\geq \mbox{$\frac{1}{t}$} \EE\Bigg[\sum_{i=N_0}^\infty 2^{i-1} I_{\{\sigma_{2^i+i-1}\leq t\}} \Bigg] &&\hspace{-35pt} \geq \mbox{$\frac{1}{t}$} \EE\Bigg[\sum_{\ell=2^{N_0}+N_0-1}^{I(t)-1} 2^{\ell-1}\Bigg] \\ 
&= \mbox{$\frac{1}{t}$} \EE\Bigg[\sum_{\ell=1}^{I(t)-1} 2^{\ell-1} - \sum_{\ell=1}^{2^{N_0}+N_0-2} 2^{\ell-1}\Bigg] 
&&\hspace{-35pt}= \mbox{$\frac{1}{t}$} \left(\EE\left[2^{I(t)-1}\right] - 2^{2^{N_0}+N_0-2}\right).
\end{align}
 By \lemref{elementary-renewal}, $N(t) \rightarrow \infty\; (a.s.)$ as $t\rightarrow \infty$ so $I(t) \rightarrow \infty$ as well.  Thus the asymptotics of $\mu_{1,t}(\Gamma^c)$ is determined by the asymptotics of the first summand above. 

We next determine bounds on $I(t)$.  Since $J(t) \leq 2^{I(t)-1}$ and $N(t) = 2^{I(t)-1}+I(t)+J(t)-2$, we have $N(t) \leq 2^{I(t)}+I(t) -2$ and hence 
\begin{equation} \label{bd1}
2^{I(t)-1} \geq \mbox{$\frac{1}{2}$} (N(t) - I(t) +2).
\end{equation}
Since $I(t)\ge 1 $ and  $J(t) \ge  0$, $N(t) \ge  2^{I(t)-1}-1$ so 
\begin{equation} \label{I-minus-one-lb}
I(t)-1 < \log_2(N(t)+1) \leq \log_2(N(t)+N(t)) = \mbox{$\log_2(\frac{N(t)\cdot 2t}{t}) = \log_2(\frac{N(t)}{t}) + \log_2(2t)$}.
\end{equation}
Using this estimate in \eqref{bd1} yields
$$2^{I(t)-1} > \mbox{$\frac{1}{2}$} (N(t) - \log_2(\mbox{$\frac{N(t)}{t}$}) - \log_2(2t)+ 1).$$
Employing this lower bound in \eqref{mu1-lb} and Jensen's inequality on the second summand, we have
$$\mu_{1,t}(\Gamma^c) \geq \mbox{$\frac{1}{2}$} \cdot \left(\frac{\EE[N(t)]}{t} - \frac{\log_2(\EE[\frac{N(t)}{t}])}{t} - \frac{\log_2(2t)}{t}\right) \stackrel{t\rightarrow \infty}{\longrightarrow} \mbox{$\frac{1}{2}$};$$
note that we have also used \propref{elementary-renewal} on the second summand.  Therefore for any $\epsilon < \frac{1}{2}$, $\mu_{1,t}(
\Gamma^c) > \epsilon$ for all $t$ sufficiently large.  Hence $\{\mu_{1,t}\}$ is not tight as $t\rightarrow \infty$.

Consider now the total ordering costs by time $t > 0$.  First, denote $(\tau,Y) = \{(\tau_k,Y_k): k \in \NN\}$ to capture all of the orders.  Next, denote the non-zero orders by ${\cal S}$ and as above, for each $i \in \NN$ and $j=1,\ldots, 2^{i-1}$, let $\wdt\sigma_{2^i+i-1,j} = \sigma_{2^i+i-1}$ be the common time of the $0$-size orders in cycle $i$.  

Since $t$ is finite and $0 \leq J(t) \leq 2^{I(t)-1} $, 
\begin{align*}
\lefteqn{\EE\left[\sum_{k=1}^\infty I_{\{\tau_k\leq t\}} c_1(X(\tau_k-),X(\tau_k))\right]} \\
   & =  \EE\left[\sum_{i=1}^{I(t)-1} \left[2^{i-1}(k_1+1)  + (k_1 + 2^{(i-1)/2}) + 2^{i-1} k_{1}\right] + J(t) (k_{1}+1)\right]\\ 
    & \le     \EE\left[\left( (2k_1+1) (2^{I(t)-1}-1) + \frac{2^{(I(t)-1)/2}-1}{2^{1/2}-1} + k_1 (I(t)-1)\right)+ 2^{I(t) -1} (k_{1}+1)\right]\\
&=  \EE\left[(3k_1+2) 2^{I(t)-1}  + \frac{2^{(I(t)-1)/2}-1}{2^{1/2}-1} + k_1 I(t)  -3 k_{1} -1 \right].
\end{align*}
Since $I(t) \ge 1$ and $ J(t) \geq 0$, it follows that $2^{I(t)-1}    \le N(t) + 1 $ and so $2^{(I(t)-1) /2} \le (N(t)+1)^{1/2}$.  Using \eqref{I-minus-one-lb}, we also have $I(t) <  \log_{2}(\frac{N(t)}{t}) + \log_{2}(2t) + 1$. 
Then it follows from Jensen's inequality  that \begin{align*}
\EE&\left[\sum_{k=1}^\infty I_{\{\tau_k\leq t\}} c_1(X(\tau_k-),X(\tau_k))\right] \\
      & \le  (3k_{1} +2) \EE[N(t)]    + \frac{     \EE[(N(t)+1)^{1/2}]-1}{2^{1/2}-1}  + k_{1} \EE\left[ \log_{2}(N(t)/t ) + \log_{2}(2t) + 1\right] + 1 \\ 
      & \le   (3k_{1} +2) \EE[N(t)]    + \frac{   (\EE[N(t)+1])^{1/2} - 1} {2^{1/2}-1}    + k_{1}\big [ \log_{2} (\EE[N(t)/t] ) + \log_{2}(2t) + 1\big] + 1. 
\end{align*}
Now divide both sides by $t$, and then send $t\to \infty$, obtaining \eqref{eq-limsup-order-cost}  from Proposition \ref{elementary-renewal}.
\end{proof}

\begin{rem}[Final Comments]
Under further analysis, one can obtain more precise results about the costs related to the ordering policy $(\tau,Y)$ of \defref{stoch-tauY-def}.  As in the analysis of \propref{cycle-cost-asymptotics}, a lower bound on the limit inferior of the Ces\`{a}ro mean of the expected cycle costs can be shown to be $\frac{5}{2}$.  With more extensive calculations including the variance of the holding costs per cycle, these bounds can be shown to be tight and moreover that 
$$\liminf_{t\rightarrow \infty} \mbox{$\frac{1}{t}$} \int_0^t c_0(X(s))\, ds = \mbox{$\frac{5}{2}$} \; (a.s.) \quad \mbox{and} \quad \limsup_{t\rightarrow \infty} \mbox{$\frac{1}{t}$} \int_0^t c_0(X(s))\, ds = 3 \; (a.s.).$$
\end{rem}

\end{appendices}
\bibliographystyle{apalike}

\end{document}